\documentclass[10pt]{article}

\usepackage{url}
\usepackage{mathtools}
\usepackage{amssymb}
\usepackage{amsmath}
\usepackage{bbm}
\usepackage{mathrsfs}
\usepackage{graphicx}
\usepackage{amsthm}
\usepackage{empheq}
\usepackage{latexsym}
\usepackage{enumitem}
\usepackage{eurosym}
\usepackage{dsfont}
\usepackage{appendix}
\usepackage{color} 
\usepackage[unicode]{hyperref}
\usepackage{frcursive}
\usepackage[utf8]{inputenc}
\usepackage[T1]{fontenc}
\usepackage{geometry}
\usepackage{multirow}
\usepackage{todonotes}
\usepackage{lmodern}
\usepackage{anyfontsize}
\usepackage{stmaryrd}
\usepackage{natbib}
\usepackage{cleveref}
\usepackage{comment}
\usepackage{tcolorbox}

\setcitestyle{numbers,open={[},close={]}}

\definecolor{red}{rgb}{0.7,0.15,0.15}
\definecolor{green}{rgb}{0,0.5,0}
\definecolor{blue}{rgb}{0,0,0.7}
\hypersetup{colorlinks, linkcolor={red},citecolor={green}, urlcolor={blue}}
			
\makeatletter \@addtoreset{equation}{section}

\newtheorem{theorem}{Theorem}[section]
\newtheorem{assumption}[theorem]{Assumption}
\newtheorem{corollary}[theorem]{Corollary}
\newtheorem{example}[theorem]{Example}

\newtheorem{lemma}[theorem]{Lemma}
\newtheorem{proposition}[theorem]{Proposition}

\newtheorem{definition}[theorem]{Definition}
\newtheorem{remark}[theorem]{Remark}

\DeclareUnicodeCharacter{014D}{\=o}
\def \E{\mathbb{E}}
\def \F{\mathbb{F}}

\def \H{\mathbb{H}}

\def \L{\mathbb{L}}
\def \M{\mathbb{M}}
\def \N{\mathbb{N}}

\def \P{\mathbb{P}}
\def \Q{\mathbb{Q}}
\def \R{\mathbb{R}}
\def \S{\mathbb{S}}

\def\Ac{{\cal A}}
\def\Bc{{\cal B}}
\def\Cc{{\cal C}}
\def\Dc{{\cal D}}

\def\Fc{{\cal F}}
\def\Gc{{\cal G}}

\def\Ic{{\cal I}}

\def\Mc{{\cal M}}
\def\Nc{{\cal N}}

\def\Pc{{\cal P}}

\def\Rc{{\cal R}}

\def\Vc{{\cal V}}
\def\Wc{{\cal W}}

\def\Yc{{\cal Y}}

\def\IC{{\cal IC}}

\newcommand{\smallfont}[1]{\text{\fontsize{4}{4}\selectfont$#1$}}
\newcommand{\tinyfont}[1]{\text{\fontsize{3}{3}\selectfont$#1$}}

\allowdisplaybreaks

\author{Daniel {\sc Kr{\v{s}}ek} \footnote{ETH Z\"urich, Department of Mathematics, Switzerland, daniel.krsek@math.ethz.ch} \and Dylan {\sc Possama\"{i}} \footnote{ETH Z\"urich, Department of Mathematics, Switzerland, dylan.possamai@math.ethz.ch.} }

\title{Randomisation with moral hazard: a path to existence of optimal contracts}

\date{\today}	

\begin{document}

\maketitle

\begin{abstract} We study a generic principal--agent problem in continuous time on a finite time horizon. We introduce a framework in which the agent is allowed to employ measure-valued controls and characterise the continuation utility as a solution to a specific form of a backward stochastic differential equation driven by a martingale measure. We leverage this characterisation to prove that, under appropriate conditions, an optimal solution to the principal's problem exists, even when constraints on the contract are imposed. In doing so, we employ compactification techniques and, as a result, circumvent the typical challenge of showing well-posedness for a degenerate partial differential equation with potential boundary conditions, where regularity problems often arise.

\end{abstract}

\section{Introduction} 
The principal--agent problem is a standard model in microeconomics with applications in various other fields. It describes a situation in which a principal (`she/her') hires an agent (`he/him') to take control over an asset or a project in exchange for a salary. With a contract in place that describes such rewards, the agent aims to maximise his utility and acts accordingly. The principal, anticipating the behaviour of the agent, looks for a contract maximising her own utility. We consider a situation with moral hazard, that is, it is assumed that the actions of the agent are unobservable to the principal and cannot be contracted upon. As a consequence, the optimal contract is expected to be designed in such a way that it incentivises the agent to work in a way that aligns with the principal's objectives. The primary goal of our work is to provide a general result regarding the existence of an optimal contract, even when potential constraints are in place. We consider a continuous-time framework and generalise the standard Brownian model to a one that incorporates relaxed controls for the agent's problem. This, in turn, provides the principal with more flexibility when choosing the contract and allows us to use compactness and continuity arguments to establish the existence of an optimal solution within a specific class of contracts.

\medskip The principal--agent problem has been extensively studied by many economists and mathematicians. Continuous-time model research was pioneered by  {\rm\citeauthor*{holmstrom1987aggregation} \cite{holmstrom1987aggregation}}, where, despite working in a relatively simple setting, the authors were able to derive an explicit solution to the principal's problem. This research was subsequently followed by many others, including {\rm\citeauthor*{schattler1993first} \cite{schattler1993first,schattler1997optimal}}, {\rm\citeauthor*{muller1998first} \cite{muller1998first,muller2000asymptotic}}, {\rm\citeauthor*{williams2008dynamic} \cite{williams2008dynamic}} and {\rm\citeauthor*{cvitanic2012contract} \cite{cvitanic2012contract}}. They employed various techniques to investigate the principal's problem, one of the main one being analysis of the agent's value process, given a contract offered by the principal, and then maximising over all possible contracts in the principal's problem. The main limitation is that solving the principal's problem is often challenging and there is no unified method due to its non-standard nature. The problem has recently regained the interest of the research community with the seminal paper by {\rm\citeauthor*{sannikov2008continuous} \cite{sannikov2008continuous}}, which studies this problem on an infinite time horizon and analyses the conditions under which contract termination occurs. This work was later formalised and extended by {\rm\citeauthor*{possamai2020there} \cite{possamai2020there}}. This approach using what \citeauthor*{sannikov2008continuous} coined a \emph{martingale method} was then successfully leveraged in several specific examples, including \citeauthor*{biais2007dynamic} \cite{biais2007dynamic}, {\rm\citeauthor*{biais2010large} \cite{biais2010large}}, \citeauthor*{pages2014mathematical} \cite{pages2014mathematical}, or {\rm\citeauthor*{capponi2015dynamic} \cite{capponi2015dynamic}}.

\medskip It was not until 2018 that a systematic method, relying greatly on \citeauthor*{sannikov2008continuous}'s insight---which itself could be traced back to static models, such as in \citeauthor*{spear1987repeated} \cite{spear1987repeated}---for studying a large class of models was rigorously introduced in the work of {\rm\citeauthor*{cvitanic2018dynamic} \cite{cvitanic2017moral,cvitanic2018dynamic}}, which was later extended or made use of in numerous other papers.\footnote{A non-comprehensive list includes \citeauthor*{abi2024gaussian} \cite{abi2024gaussian}, {\rm\citeauthor*{aid2022optimal} \cite{aid2022optimal}}, \citeauthor*{aid2023principal} \cite{aid2023principal}, \citeauthor*{alasseur2020principal} \cite{alasseur2020principal}, {\rm\citeauthor*{aurell2022optimal} \cite{aurell2022optimal}}, {\rm\citeauthor*{baldacci2023design} \cite{baldacci2023design}}, {\rm\citeauthor*{baldacci2021optimal} \cite{baldacci2021optimal}}, \citeauthor*{baldacci2022governmental} \cite{baldacci2022governmental}, {\rm\citeauthor*{baldacci2023market} \cite{baldacci2023market}}, {\rm\citeauthor*{bensalem2023continuous} \cite{bensalem2023continuous}}, {\rm\citeauthor*{cvitanic2018asset} \cite{cvitanic2018asset}}, \citeauthor*{el2021optimal} \cite{el2021optimal}, {\rm\citeauthor*{elie2019tale} \cite{elie2019tale}}, {\rm\citeauthor*{el2021optimal} \cite{el2021optimal}}, {\rm\citeauthor*{elie2021mean} \cite{elie2021mean}}, \citeauthor*{firoozi2021principal} \cite{firoozi2021principal}, \citeauthor*{hajjej2017optimal} \cite{hajjej2017optimal}, \citeauthor*{hajjej2022optimal} \cite{hajjej2022optimal}, \citeauthor*{hernandez2019moral} \cite{hernandez2019moral}, \citeauthor*{hernandez2020bank} \cite{hernandez2020bank}, {\rm\citeauthor*{hernandez2023pollution} \cite{hernandez2023pollution}}, {\rm\citeauthor*{hernandez2024principal} \cite{hernandez2024principal}}, {\rm\citeauthor*{hernandez2024time} \cite{hernandez2024time}}, \citeauthor*{hu2023principal} \cite{hu2023principal}, {\rm\citeauthor*{hubert2023continuous} \cite{hubert2023continuous}}, \citeauthor*{hubert2022incentives} \cite{hubert2022incentives}, {\rm\citeauthor*{keppo2021dynamic} \cite{keppo2021dynamic}}, \citeauthor*{kharroubi2020regulation} \cite{kharroubi2020regulation}, \citeauthor*{lin2022random} \cite{lin2022random}, \citeauthor*{luo2021dynamic} \cite{luo2021dynamic}, {\rm\citeauthor*{martin2023risk} \cite{martin2023risk}}, \citeauthor*{mastrolia2015moral} \cite{mastrolia2015moral}, \citeauthor*{mastrolia2017moral} \cite{mastrolia2017moral}, \citeauthor*{mastrolia2018principal} \cite{mastrolia2018principal}, \citeauthor*{mastrolia2022agency} \cite{mastrolia2022agency}, and \citeauthor*{ren2023entropic} \cite{ren2023entropic}.}
This approach heavily hinged on the theory of backward stochastic differential equations (BSDEs), respectively second order backward stochastic differential equation (2BSDEs), by reformulating the principal's problem as a control problem through the introduction of the continuation utility of the agent as a controlled state process. While this represents a breakthrough result and transforms the problem into a classical optimal control problem, which can then be addressed using standard known methods, there is a lack of a general existence result for an optimal contract. This is mainly due to the complex nature of the resulting optimisation problem, since existence of a solution is then typically equivalent to the existence of a classical solution to the corresponding Hamilton--Jacobi--Bellman (HJB) partial differential equation (PDE). However, even when the agent controls only the drift coefficient, the principal's problem then involves control of the diffusion coefficient for the agent's continuation utility. Consequently, the PDE becomes generally fully non-linear elliptic or parabolic in dimension strictly greater than $1$ (typically at least $2$, and generically, when there is only one agent, in dimension $d$+1 where $d$ is the dimension of the process controlled by the agent), with an inherently degenerate diffusion operator. 

\medskip
Indeed, uniform ellipticity for these type of PDEs is a fundamental assumption to hope for classical solutions, as counter-examples in multi-dimensions exist, see for instance \citeauthor*{nadirashvili2007nonclassical} \cite{nadirashvili2007nonclassical,nadirashvili2008singular,nadirashvili2011octonions,nadirashvili2011singular,nadirashvili2013singular}. Even taking into account the fact that the PDE is of HJB--type, and therefore enjoys more structure, regularity results remain in general elusive, especially when the PDEs are written on a domain, which is the case when constraints are added on payments, and one can at most hope for some form of Sobolev regularity, which is usually too weak for our purpose, see \citeauthor*{krylov1995theorem} \cite{krylov1990smoothness,krylov1995theorem}, \citeauthor*{lions1983optimal} \cite{lions1983optimal}, or more recently \citeauthor*{strulovici2015smoothness} \cite{strulovici2015smoothness} as well as \citeauthor*{durandard2022smoothness} \cite{durandard2022smoothness}. Therefore, we inevitably face regularity issues, and obtaining a satisfactory result in a general framework thus seems extremely unlikely. As an example, let us mention the work of \citeauthor*{possamai2020there} \cite[Sections 6--8]{possamai2020there}, where due to a specific and rather simple framework the equation simplifies to a one-dimensional second-order elliptic PDE---in other words an ODE. Notwithstanding, even in such a case, which seems to be rather simple, the lack of uniform ellipticity makes a rigorous proof of existence of a regular solution is rather lengthy and tedious, and requires a considerable amount of work. Though this is not always the case, the other results we are aware of in the continuous-time literature either concern case with explicit solutions, or very specific models, see for instance {\rm\citeauthor*{cvitanic2018dynamic} \cite{cvitanic2018dynamic}}, {\rm\citeauthor*{hernandez2023pollution} \cite{hernandez2023pollution}}, A slight exception here is found in {\rm\citeauthor*{hernandez2024principal} \cite{hernandez2024principal}}, where the author provides conditions---which are naturally rather stringent---to ensure that the principal's PDE can be transformed into a semi-linear rather than a fully non-linear one, thus avoiding completely the degeneracy problem we mentioned.

\medskip Generally speaking, there are several approaches to tackle stochastic optimal control problems in continuous-time. The classical ones involve either solution of an HJB PDE or, in the weak formulation, solution of a BSDE or a 2BSDEs in the case of controlled volatility coefficient. The direct approach, in the stochastic setting introduced by {\rm\citeauthor*{el1987compactification} \cite{el1987compactification}} and later extended mainly by {\rm\citeauthor*{haussmann1990existence} \cite{haussmann1990existence}}, exploits compactification and continuity arguments to show the existence of an optimal control. See also the works of {\rm\citeauthor*{el1990martingale} \cite{el1990martingale}} and {\rm\citeauthor*{karoui2013capacities} \cite{karoui2013capacities,karoui2013capacities2}} for related results. This typically involves the introduction of so-called relaxed controls, that is, measure-valued controls. In the context of non-cooperative games, commonly used to find equilibria (see, for example, {\rm\citeauthor*{nash1951noncooperative} \cite{nash1951noncooperative}}), these are referred to as mixed strategies, see also \emph{e.g.} {\rm\citeauthor*{ankirchner2023correlation} \cite{ankirchner2023correlation}} for one of the most recent works.

\medskip The idea of randomised contracts was introduced in one-period models in the paper of {\rm\citeauthor*{kadan2017existence} \cite{kadan2017existence}}, where the authors proposed a method for showing the existence of an optimal mechanism within a suitably chosen class of incentive compatible mechanisms. Despite working in a different framework, the idea of randomising actions and then leveraging compactness arguments and appropriate continuity of the functions involved translates to our setting. Their work also provides examples of situations, when randomisation is necessary to obtain existence of optimal contract, making it natural to allow it, see the references therein. In addition to randomisation, compactification through the introduction of appropriate \emph{a priori} bounds on contracts  were also exploited in other one-period models to obtain abstract existence results, as in {\rm\citeauthor*{holmstrom1977incentives} \cite{holmstrom1977incentives,holmstrom1979moral}}, {\rm\citeauthor*{myerson1982optimal} \cite{myerson1982optimal}}, {\rm\citeauthor*{page1987existence} \cite{page1987existence,page1991optimal,page1992mechanism}}, {\rm\citeauthor*{balder1996existence} \cite{balder1996existence}}, {\rm\citeauthor*{jewitt2008moral} \cite{jewitt2008moral}}, {\rm\citeauthor*{renner2015polynomial} \cite{renner2015polynomial}}, {\rm\citeauthor*{backhoff2022robust} \cite{backhoff2022robust}}, {\rm\citeauthor*{ke2017existence} \cite{ke2017existence}}, and {\rm\citeauthor*{gan2022optimal} \cite{gan2022optimal}}.\footnote{There are three slight exceptions to the ideas of compactification and randomisation, namely {\rm\citeauthor*{grossman1983analysis} \cite{grossman1983analysis}}, but there the spaces involved are all finite, so that everything is bounded by default, {\rm\citeauthor*{kahn1993existence} \cite{kahn1993existence}}, which considers only pure adverse selection problems and relies on restrictions on the set of types, distributions and utilities considered, and finally {\rm\citeauthor*{carlier2005existence} \cite{carlier2005existence},} but there the authors are actually proving existence for a relaxed version of the problem which coincides with the original one when the so-called first-order approach is valid.} Among recent publications related to our results, we would like to mention that of {\rm\citeauthor*{djete2023stackelberg} \cite{djete2023stackelberg}}, who studies relaxed controls in the context of principal--multi-agent problem and mean-field games, and then employs a convexity argument, as also addressed in \cite{el1987compactification,haussmann1990existence}, to recover the standard, non-relaxed setting. Further, {\rm\citeauthor*{alvarez2023optimal} \cite{alvarez2023optimal}} consider randomised contracts in a continuous-time framework and show the existence of optimal contracts, albeit in a somewhat restrictive setting.  It is worth noting that neither work uses the reformulation introduced in \cite{cvitanic2018dynamic} and instead directly shows the existence of an optimal contract within a certain class.

\medskip Our main motivation for this work is severalfold. As mentioned earlier, when one considers the optimal control problem that the principal aims to solve in the context of \cite{cvitanic2018dynamic}, the typical approach is to write down the corresponding Hamilton--Jacobi--Bellman PDE, which tends to be already quite involved. The complexity deepens when one imposes constraints on the contract. One such scenario is imposing the condition
\[ 0 \leq \xi \leq \ell(X_{T}), \]  where $\xi$ represents the lump-sum payment, and $\ell(X_{T})$ is the liquidation value of the project at the terminal time $T>0$. This constraint is natural, as a sensible contract should avoid involving payments that exceed the value of the corresponding project or provide negative payments to the agent. In such a situation, solving the problem involves dealing with a PDE with boundary conditions given as solutions to other PDEs. We bypass this obstacle by introducing randomised actions, a concept closely related to those introduced in \cite{kadan2017existence}. However, it is important to emphasise that, while the authors therein allow for measure-valued pay-offs, we consider pay-offs that are real-valued variables, although they are measurable with respect to a somewhat larger $\sigma$-algebra.

\medskip Our main contribution is the introduction of a, to the best of our knowledge, novel framework in which the agent controls a process driven by a martingale measure. This naturally extends the typical Brownian setting, as seen in \cite{cvitanic2018dynamic,hernandez2024principal, djete2023stackelberg}, to the relaxed framework. We provide a description of the agent's problem in terms of a BSDE driven by the martingale measure and characterise the set of all incentive-compatible contracts within this framework. Next, we define the principal's problem and reformulate it as a weak control problem resembling the framework considered in \cite{el1987compactification,haussmann1990existence}.  We manage to establish the existence of a solution to the weak formulation of the principal's problem, provided that certain tightness criteria are met. Furthermore, achieving this is possible even when extremely general---even uncountably many---constraints on the contract are imposed. Notably, the existing literature has not yet treated such a general situation and potentially only boundedness or one-sided boundedness was considered. However, it is important to note that this weak formulation of the problem may not always be equivalent to the original, inherently strong, formulation. This discrepancy stems from measurability issues associated with the weak nature of the problem, as measurability properties are generally not preserved under weak convergence of measures. We propose two approaches to address this issue. In the first one, we assume that there is sufficient `randomness' in the original probability space for the agent, which is then thus also present in the strong principal's problem, thereby making the two formulations equivalent. Alternatively, in the second approach, we require the contract to be of feedback form, and thus to meet the appropriate measurability conditions. This, however, may result in a loss of closedness of the set of admissible controls in the weak formulation, and therefore we might also potentially lose the existence of an optimal contract within the specified class. We also emphasise the fact that these measurability issues are inherent to the problem, and are not linked to the relaxation we propose: they were already present in \cite{cvitanic2018dynamic} for instance, though this was not explicitly pointed out. As far as we know, our work is the first attempt to obtain general existence results within the continuous-time setting.

\medskip The rest of the paper is organised as follows. \Cref{prelim} provides several results and essentials of the theory of martingale measures. In \Cref{agent}, we introduce the relaxed problem of the agent, and in \Cref{sec:principal}, we present the corresponding principal's problem. Within \Cref{sec:principal}, we also introduce two weakened formulations of the problem, show existence of their solutions, and relate them to the original problem. In \Cref{sec:additional}, we comment on some specific features of our setting and discuss possible extensions. The appendix contains technical propositions, proofs, and presents results from the BSDE theory with martingale measures.

\medskip
{\small\textbf{Notations:} Let $\N^\star$ denote the set of positive integers. For an $n$-dimensional vector $v,$ where $n \in \N^\star,$ we denote by $v^1,\ldots,v^n$ its components. For $(u,v) \in \R^n\times\R^n,$ we denote by $u \cdot v$ the standard inner product and by $\lVert u\rVert$ the associated Euclidean norm. For $(n,m) \in \N^\star\times\N^\star,$ $\R^{n \times m}$ denotes the space of all $n \times m$ matrices with real entries. The entries of $A \in \R^{n \times m}$ are denoted by $A^{i,j},$ $i \in \{1,\ldots,n \},$ $j \in \{1,\ldots,m \},$ $A^\top$ denotes the transpose of $A,$ and ${\rm rank}(A)$ denotes the rank of $A$. For any matrices $A \in \R^{n \times m}$ and $B \in \R^{m \times n},$ we use the notation $A:B$ for the trace of $AB.$ We denote the identity matrix of size $n$ by $\text{I}_{n \times n}.$ If $(\Omega,\Fc)$ is a measurable space, $\Gc \subset \Fc$ a sub--$\sigma$-algebra of $\Fc$, we denote by $\L^0_n(\Gc)$ the set of all $\R^n$-valued, $\Gc$-measurable random variables. We write $\L^0(\Gc)$ for $\L^0_1(\Gc).$ If $\P$ is a probability measure on $(\Omega,\Fc)$ and $p>0$, we write $\L^p_n(\Gc,\P)$ for the set of all $\R^n$-valued, $\Gc$-measurable random variables, with finite $p$-th moment under $\P$. We write $\L^p(\Gc,\P)$ for $\L^p_1(\Gc,\P).$}

\section{Preliminaries} \label{prelim}

The aim of this section is to introduce key concepts from stochastic calculus for martingale measures and present some supporting results. For a more in-depth analyses, interested readers are referred to the works of  {\rm\citeauthor*{walsh1986introduction} \cite{walsh1986introduction}} and {\rm\citeauthor*{el1990martingale} \cite{el1990martingale}}. Throughout this section, we consider a fixed time horizon $T>0.$ Let us be further given a filtered probability space $(\Omega, \Fc,\F=(\Fc_t)_{t\in[0,T]},\P)$ satisfying the usual conditions, and let $E$ be a Polish space with its Borel $\sigma$-algebra $\Bc(E).$ We denote by $\Pc(E)$ the set of all Borel probability measures on $E$, endowed with its weak topology. 

\begin{definition} Let $U=\{ U(A) : A \in \Bc(E) \}$ be a family of $\R$-valued random variables and let $t \in [0,T]$. We say that $U$ is an $\L^2(\Fc_t,\P)$-valued measure if the following holds
\begin{enumerate}
\item[$(i)$] $U(A) \in \L^2(\Fc_t,\P)$, for any $A \in \Bc(E);$
\item[$(ii)$] $U(A\cup B)=U(A) + U(B),$ $\P$--{\rm a.s.}, for any $(A,B) \in \Bc(E)\times\Bc(E)$ such that $A \cap B = \emptyset;$
\item[$(iii)$] $\E^{\P} [U(A_n)^2] \searrow 0$ for any sequence $(A_n)_{n \in \N^\smallfont{\star}} \subset \Bc(E)$ such that $A_n \searrow \emptyset.$
\end{enumerate}
\end{definition}
In the terminology of {\rm\citeauthor*{el1990martingale} \cite{el1990martingale}} and {\rm\citeauthor*{walsh1986introduction} \cite{walsh1986introduction}}, $U$ is an `$\L^2$-valued measure{\rm'}, which corresponds to the fact that it is, indeed, a vector measure with values in $\L^2(\Fc_t,\P).$ However, let us point out that an $\L^2(\Fc_t,\P)$-valued measure is not a random measure as understood in the usual sense. One can, nonetheless, verify that if $U(A)$ is uncorrelated with $U(B)$ for any $(A,B) \in \Bc(E)^2$ such that $A \cap B=\emptyset$, and $\E^{\P} [U(A)]=0$ for any $A \in \Bc(E),$ then the function \[A \longmapsto \E^{\P}[U(A)^2],\;A \in \Bc(E),\] defines a deterministic Borel measure on $E.$  This is indeed satisfied for martingale measures characterised in the following definition, where, for a fixed $t,$ the map $A \longmapsto M_t(A)$ satisfies the conditions above.
 \begin{definition} \label{WNMM} We say that a collection of $(\F,\P)$-martingales $M=\{ (M_t(A))_{t \in [0,T]} : A \in \Bc(E) \}$ is a continuous $(\F,\P)$--martingale measure on $E$ if the following holds
\begin{enumerate}
\item[$(i)$] $(M_t(A))_{t \in [0,T]}$ is an $\R$-valued, continuous, square-integrable $(\F,\P)$-martingale with $M_0(A)=0$,  $\P${\rm--a.s.,} for every $A \in \Bc(E);$
\item[$(ii)$] for $(A,B) \in \Bc(E)^2$ such that $A \cap B=\emptyset,$ $M(A)$ and $M(B)$ are orthogonal martingales$;$
\item[$(iii)$] for any $t \in [0,T],$ $M_t(\cdot)$ is an $\L^2(\Fc_t,\P)$-valued measure$;$
\item[$(iv)$] the quadratic variation of $M(A)$ is absolutely continuous with respect to the Lebesgue measure on $[0,T]$ for any $A \in \Bc(A).$
\end{enumerate}
\end{definition} 

\begin{remark}
\begin{itemize}
\item[$(i)$] Let us remark that, while the authors of {\rm\cite{el1990martingale}} and {\rm\cite{walsh1986introduction}} work with more general martingale measures with potentially discontinuous paths and general quadratic variations, we consider only a specific subclass of continuous martingale measures with quadratic variation absolutely continuous with respect to the Lebesgue measure. These can be viewed as mixtures of Wiener processes and appear naturally, \emph{e.g.}, when one considers measure-valued controls in martingale problems driven by a white noise, see also {\rm \Cref{exampleMM}}. We refer to {\rm \Cref{agent}} for more discussion.
\item[$(ii)$] Moreover, while previous works allow for martingale measures defined on a potentially smaller subring of $\Bc(E)$ $($see {\rm\cite[pages 84--85]{el1990martingale}}$)$, we require that $M_t(A)$ be defined and square integrable for any $t \in [0,T]$ and $A \in \Bc(E).$ We shall relax this by localisation techniques later on.
\end{itemize}
\end{remark}

Let us note that orthogonality more specifically means that $[ M(A),M(B) ] \equiv 0.$ In this particular case, this is equivalent to $M(A)M(B)$ being an $(\F,\P)$-martingale. More generally, one can easily verify that necessarily 
\[
[ M(A), M(B) ] = [ M(A \cap B ) ],\; (A,B) \in \Bc(E)^2,
\] and that $M(\emptyset)=0$ up to indistinguishability. As mentioned earlier, $M$ is not a measure-valued process in the standard sense, as defined \emph{e.g.} in \citeauthor*{carmona1999stochastic} \cite[Chapter 2]{carmona1999stochastic} or \citeauthor*{dawson1993measure} \cite{dawson1993measure}, see also \citeauthor*{khoshnevisan2009primer} \cite[page 6, Example 3.16]{khoshnevisan2009primer}. However, the quadratic variation process can be regularised as follows. Let us, for a Polish space $A,$ denote by $\Mc_f(A)$ the set of all finite Borel measures on $A.$

\begin{lemma} \label{disin} 
Let $M$ be a continuous $(\F,\P)$--martingale measure. Then there is a measurable map $m : \Omega \longrightarrow \Mc_f(E \times [0,T])$ such that 
\[
\P \big[ m(A \times (0,t])=[ M(A) ]_t,\; t \in [0,T]\big]=1,\; A \in \Bc(E).
\]
Moreover, $m$ disintegrates as $m(\mathrm{d}e,\mathrm{d}s)=m_s(\mathrm{d}e)\mathrm{d}s$ for some $($regular$)$ $\F$-predictable kernel $m : \Omega \times [0,T] \longrightarrow \Mc_f(E).$ 
\end{lemma}
\begin{proof} 
See \cite[Theorem 2.7]{walsh1986introduction} and \cite[Lemma III-1]{el1990martingale}.
\end{proof}
The measure $m$ from \Cref{disin} shall be referred to as the intensity of $M$ and we will write $[ M(\cdot) ]=m$ for brevity.

\begin{example} \label{exampleMM} Let us have a deterministic measure $\mu \in \Pc(E)$ and let $M$ be a martingale measure with intensity $m(\mathrm{d}e,\mathrm{d}t)=\mu(\mathrm{d}e)\mathrm{d}t.$ It is then easy to see that for $A \in \Bc(E)$ the stochastic process $M(A)$ is a continuous martingale, $M_0(A)=0,$ $\P${\rm--a.s.,} and $[ M(A) ]_t =\mu(A) t,$ $ t \in [0,T].$ Thus, if $\mu(A)=0,$ then $M(A)\equiv 0$ and, if this is not the case, then the process \[ W^A \coloneqq  \frac{M(A)}{\sqrt{\mu(A)}},\] is an $(\F,\P)$--Brownian motion. Similar conclusion can be obtained if the quadratic variation is a general random measure from {\rm\Cref{disin}} by setting \[W^A \coloneqq \int_0^\cdot \frac{1}{\sqrt{m_s(A)}} \mathrm{d}M_s(A). \] This supports the intuition that $M$ is a, potentially random, $`$mixture$\text{'}$ of Wiener processes, which is given by the quadratic variation. It also gives insight into why we restrict our attention to martingale measures with quadratic variations projections of which are absolutely continuous with respect to the Lebesgue measure.
\end{example}

There is, as expected, a localised version of martingale measures.

\begin{definition} 
We say that a collection $M=\{ (M_t(A))_{t \in [0,T]} : A \in \Bc(E) \}$ is a continuous \emph{local} $(\F,\P)$--martingale measure if there exists a non-decreasing sequence of $\F$--stopping times $(\tau_n)_{n \in \N^\smallfont{\star}}$ such that $ \lim_{n \rightarrow \infty} \tau_n=T$, $\P${\rm--a.s.}, and $M^{\tau_\smallfont{n}}=\{ (M_{t\wedge \tau_\smallfont{n}}(A))_{t \in [0,T]} : A \in \Bc(E) \}$ is a continuous $(\F,\P)$--martingale measure for every $n \in \N^\star.$
\end{definition}

Given a continuous local $(\F,\P)$--martingale measure $M$, using continuity of the paths, one can always take 
\[ 
\tau_n\coloneqq\inf \big\{ t \in [0,T] : [ M(E) ]_t \geq n  \big\} \wedge T;\; n \in \N^\star, 
\]
as a localising sequence. Indeed, we then have for any $t \in [0,T]$ and $A \in \Bc(E)$ that
\[ 
[ M(A) ]_{t \wedge \tau_\smallfont{n}} \leq [ M(E) ]_{t \wedge \tau_\smallfont{n}} \leq n,\;\P\text{\rm--a.s.} 
\]

\begin{remark} Since we will be dealing exclusively with continuous---local---martingale measure, we shall simply use the terminology---local---martingale measure instead, even though we consider only a special subclass.
\end{remark}

\begin{definition} We say that a collection $M=\{ (M_t(A))_{t \in [0,T]} : A \in \Bc(E) \}$ is a $d$-dimensional $($local$)$ martingale measure on $E$, if every component of $M$ is a $($local$)$ martingale measure.
\end{definition}

It is possible to construct an integral with respect to a martingale measure by It\^{o}'s construction, this is mainly due to \citeauthor*{walsh1986introduction} \cite[page 292]{walsh1986introduction}. Note that since $M(A)$ is a martingale for any set $A\in \Bc(E),$ it is thus possible to integrate with respect to the time variable in the stochastic sense. However, the map $A \longmapsto M_t(A)$ for a fixed $t \in [0,T]$ and $\omega \in \Omega$ is not a measure \emph{per se}. The stochastic integral rather utilises the $\L^2$-orthogonal structure in time and space and the construction is done simultaneously with respect to both variables as an It\^{o} integral.

\medskip
Let us denote by $\Pc\Mc(\F)$ the $\F$-progressive $\sigma$-algebra and let us, for a local $(\F,\P)$--martingale measure $M$, denote by $\L^2(\F,\P,M(\cdot))$ the set of all $\Pc\Mc(\F) \otimes \Bc(E)$-measurable functions $Z: \Omega \times [0,T] \times \Bc(E) \longrightarrow \R$ such that
\[ 
\E^\P \bigg[ \int_0^T \int_E \lvert Z_s(e) \rvert^2 m(\mathrm{d}e,\mathrm{d}s) \bigg] < \infty. 
\] 
Furthermore, $\L^2_{\rm loc}(\F,\P,M(\cdot))$ denotes the set of all $\Pc\Mc(\F) \otimes \Bc(E)$-measurable functions $Z: \Omega \times [0,T] \times \Bc(E) \longrightarrow \R$ such that there exists a non-decreasing sequence of $\F$--stopping times $(\tau_n)_{n \in \N^\smallfont{\star}}$ with $ \lim_{n \rightarrow \infty} \tau_n=T$, $\P$--a.s., and $Z\mathbf{1}_{\llbracket0,\tau_\smallfont{n}\rrbracket} \in \L^2(\F,\P,M(\cdot))$ for every $n \in \N^\star,$ where $\mathbf{1}_{\llbracket0,\tau_\smallfont{n}\rrbracket}$ is the process defined by  
\[ 
\mathbf{1}_{\llbracket0,\tau_\smallfont{n}\rrbracket}(t,\omega) \coloneqq \begin{cases} 1,\;  {\rm if}\; 0\leq t \leq \tau_n(\omega), \\ 0,\; {\rm otherwise.}
\end{cases} 
\]

Let $M$ be a local $(\F,\P)$--martingale measure, and let $Z \in \L^2_{\rm loc}(\F,\P,M(\cdot)).$ Then there exists a unique continuous local $(\F,\P)$--martingale measure 
\[
(Z \bullet M)_t(A)\coloneqq\int_0^t \int_A Z_s(e)M(\mathrm{d}e,\mathrm{d}s),\; A \in \Bc(E),\; t \in [0,T], \] 
obtained by the standard It\^{o} construction, see \cite[page 292]{walsh1986introduction} and \citep[page 86]{el1990martingale} for the case $Z \in \L^2(\F,\P,M(\cdot))$, while the extension to $\L^2_{\rm loc}(\F,\P,M(\cdot))$ using localisation arguments is straightforward. If $Y \in \L^2_{\rm loc}(\F,\P,M(\cdot)),$ then 
\[ 
\big[ (Y \bullet M)(A), (Z \bullet M)(B)\big]_t=\int_0^t \int_{A \cap B} Y_s(e) Z_s(e)m(\mathrm{d}e,\mathrm{d}s),\; (A,B) \in \Bc(E)^2,\; t \in [0,T].
\] 
Moreover, if $Z \in \L^2(\F,\P,M(\cdot))$, then $Z \bullet M$ is a true $(\F,\P)$--martingale measure in the sense of \Cref{WNMM}.

\medskip
One way of constructing a martingale measure using an existing one is taking the image measure under some measurable mapping. The following lemmata describe their fundamental properties.

\begin{lemma}\label{PF1} 
Let $F$ be a Polish space and let $Z: \Omega \times [0,T] \times E \longrightarrow F$ be $\Pc\Mc(\F) \otimes \Bc(E)$-measurable. Let $M$ be an $(\F,\P)$--martingale measure $($resp. a local $(\F,\P)$--martingale measure$)$ and let us define
\begin{align*}
N_t(A)&\coloneqq\int_0^t \int_E \mathbf{1}_{\{Z_\smallfont{s}(e)\in A\}} M(\mathrm{d}e,\mathrm{d}s),\; A \in \Bc(F),\; t \in [0,T],\\
n((0,t]\times A)&\coloneqq \int_0^t \int_E \mathbf{1}_{\{Z_\smallfont{s}(e)\in A\}} m(\mathrm{d}e,\mathrm{d}s), \; A \in \Bc(F),\; t \in [0,T].
\end{align*}
 Then $N$ is a well-defined $(\F,\P)$--martingale measure $($resp. a local $(\F,\P)$--martingale measure$)$ on $F$ such that
\[
[ N(\cdot) ]=n,\; N(F)=M(E),\; \text{\rm and}\; n_t(F)=m_t(E),\; \mathrm{d}t \otimes \P\text{\rm--a.e.}
\]
\end{lemma}
\begin{proof}
The proof can be found in \cite[Proposition II-4]{el1990martingale}.
\end{proof}

We shall use the shorthand notation $Z_\#M\coloneqq N$ and $Z_\#m\coloneqq n$ for the objects introduced in \Cref{PF1}. The next lemma is technical, so that we postpone its proof to \Cref{proofApp}.
\begin{lemma} \label{pf2} Let $M$, $N$, and $Z$ be as in {\rm \Cref{PF1}}. Let $g: \Omega \times [0,T] \times F \longrightarrow \R$ be $\Pc\Mc(\F) \otimes \Bc(F)$-measurable. Then
\[
 \int_0^T \int_E g_s^+(Z_s(e)) m(\mathrm{d}e,\mathrm{d}s)<\infty,\; \text{\rm if and only if}\;  \int_0^T \int_F  g_s^+(f)  n(\mathrm{d}f,\mathrm{d}s)<\infty. 
 \]
In the affirmative case
\[ 
\int_0^t \int_{Z_\smallfont{s}^{\smallfont{-}\smallfont{1}}(A)}  g_s(Z_s(e)) m(\mathrm{d}e,\mathrm{d}s)=\int_0^t \int_A  g_s(f) n(\mathrm{d}f,\mathrm{d}s),\; t \in [0,T],\; A \in \Bc(F). 
\]
Similarly, $(g \circ Z) \in \L^2_{\rm loc}(\F,\P,M(\cdot))$ if and only if $g \in \L^2_{\rm loc}(\F,\P,N(\cdot))$ and if this is the case, then
\[
\int_0^t \int_{Z_\smallfont{s}^{\smallfont{-}\smallfont{1}}(A)}  g_s(Z_s(e)) M(\mathrm{d}e,\mathrm{d}s)=\int_0^t \int_A  g_s(f) N(\mathrm{d}f,\mathrm{d}s),\; t \in [0,T],\;\P\text{\rm--a.s.},\; A \in \Bc(F).  \]
\end{lemma}

Let us recall that $\Pc(E)$ denotes the set of all Borel probability measures on $E$ endowed with its weak topology. The following lemma shows that every intensity can be written as a push-forward of a diffuse deterministic measure with the same mass.

\begin{lemma} \label{pf} Let $\mu \in \Pc(E)$ be a diffuse probability measure and let $F$ be a Polish space. Then for every $\F$--progressively measurable $\Pc(F)$-valued process $n$, there exists a $\Pc\Mc(\F) \otimes \Bc(E)$-measurable process $Z: \Omega \times [0,T] \times \Bc(E) \longrightarrow F$ such that 
\[
n_t(A)=\int_E \mathbf{1}_{\{Z_\smallfont{t}(e)\in A\}}\mu(\mathrm{d}e),\;A \in \Bc(E),\;\mathrm{d}t \otimes \P\text{\rm--a.e.}
\]
In other words, $n=Z_\#(\mu \otimes \mathrm{d}t).$
\end{lemma}
\begin{proof} It is a direct consequence of \cite[Theorem III-2]{el1990martingale}.
\end{proof}

Let us end this section with a result stating that the completed canonical $\sigma$-algebra of an $(\F,\P)$--martingale measure is countably generated. The proof is technical and is postponed to \Cref{proofApp}.

\begin{lemma} \label{FXfiltr} Let $\Nc$ denote the set of all $\P$-null sets and let $M$ be a $d$-dimensional $(\F,\P)$--martingale measure for some $d \in \N^\star$. Then there exists a countable system of open sets $(G_j)_{j \in \N^\smallfont{\star}}$ in $\Bc(E)$ such that the following holds for any $t \in [0,T]$
\[
\sigma \big( M_s(A) : s \in [0,t],\; A \in \Bc(E) \big) \vee \sigma(\Nc)=\bigvee_{(n,k) \in \N^\smallfont{\star}\times \N^\smallfont{\star}} \sigma \big( M_{it/2^\smallfont{n}}(G_j) : i\in\{0,\ldots,2^n\},\; j\in\{1,\ldots,k \}\big)\vee \sigma(\Nc).
\] 
\end{lemma}

\section{Relaxed agent's problem} \label{agent}

When dealing with a controlled stochastic differential equation driven by a white noise in the weak sense, such as the state process of an agent, as considered in \citeauthor*{cvitanic2018dynamic} \cite{cvitanic2018dynamic}, one typically works on the canonical space $\Cc([0,T],\R^d)$ equipped with a set of probability measures that correspond to the weak solutions of the controlled equation. These probabilities are either given as solutions to martingale problems or, in cases where the volatility coefficient is uncontrolled, one might, under appropriate assumptions, define this set as a subset of equivalent measures to the law of the solution to the drift-less equation, thus making use of the Girsanov transformation.

\medskip 
However, in the relaxed setting where measure-valued controls are considered, solutions to martingale problems do not correspond to weak solutions of equations driven by a standard Wiener process. Instead, they agree with laws of solutions to equations driven by a mixture of Wiener processes, weighted according to the chosen measure---these are known as martingale measures and were originally introduced by {\rm\citeauthor*{walsh1986introduction} \cite{walsh1986introduction}} and subsequently studied by many others in the context of stochastic partial differential equations. Their  relation to martingale problems was introduced and studied in the work of {\rm\citeauthor*{el1990martingale} \cite{el1990martingale}}. In such cases, if one wishes to define the set of measures through the Girsanov transformation in the situation of uncontrolled volatility, it is thus natural to introduce the reference measure as the law of the solution to the drift-less SDE driven by a martingale measure. It is indeed clear that an appropriate Girsanov transformation of the driving martingale measure needs to be introduced.

\medskip
In this case where there is no control in the dynamics of the drift-less equation, a natural question to ask is what makes an appropriate canonical `mixture' of Wiener processes. The answer is essentially provided in \Cref{pf}, as every measure can be represented as a push-forward of a deterministic reference measure, as long as the latter is diffuse. We thus start with a martingale measure with such a diffuse deterministic intensity, keeping in mind that its specific choice is irrelevant.

\subsection{The model}
Let $V$ be a Polish space with a diffuse Borel probability measure $\mu.$ Let us be given a filtered probability space $(\Omega,\Fc,\F,\P^0)$ satisfying the usual conditions, with $\Fc=\Fc_T$, supporting a $k$-dimensional, $k \in \N^\star$, continuous martingale measure $M^0$ on $V$, with quadratic variation
\[
\big[ M^{0,i}(A),M^{0,j}(B) \big]_t=\mathbf{1}_{\{i=j\}} \int_0^t \int_V \mathbf{1}_{\{v \in A \cap B\}} m^0(\mathrm{d}v,\mathrm{d}t),\; t \in [0,T],\; (A,B) \in \Bc(V)^2,\; (i,j) \in \{1,\ldots,k \}^2,
\]
where $m^0(\mathrm{d}v,\mathrm{d}t)=\mu(\mathrm{d}v)\mathrm{d}t.$ In other words
\[ 
\big[ M^{0,i}(A),M^{0,j}(B)\big ]_t=\mathbf{1}_{\{i=j\}}\mu(A \cap B)t,\; t \in [0,T],\; (A,B) \in \Bc(V)^2,\; (i,j) \in \{1,\ldots,k \}^2.
\]
Let $d \in \N^\star,$ $d \geq k,$ and let $\Cc([0,T],\R^d)$ denote the space of all continuous functions from $[0,T]$ to $\R^d.$ Let $U$ be a control space, which is assumed to be Polish, and let us be given measurable maps
\begin{align*}
\sigma : [0,T] \times \Cc([0,T],\R^d) \longrightarrow \R^{d \times k}, \; \lambda: [0,T] \times \Cc([0,T],\R^d) \times U \longrightarrow \R^k.
\end{align*}

The agent controls the $\R^d$-valued process $X$ given by
\begin{align} \label{x} 
X_t=x_0+\int_0^t \int_V \sigma(s,X_{\cdot \wedge s})M^0(\mathrm{d}v,\mathrm{d}s)=\int_0^t \sigma(s,X_{\cdot \wedge s})\mathrm{d}W^0_s,\; t \in [0,T], 
\end{align} 
where $x_0 \in \R^d$ is a given constant, $W^0\coloneqq M^0(V)$ is a $k$-dimensional, $(\F,\P^0)$--Brownian motion, thanks to L\'{e}vy's characterisation. Even though $X$ could have been introduced using just the Brownian motion $W^0,$ the formulation involving $M^0$ allows more flexibility when doing Girsanov transformation of the noise. This shall become clear later on.

\begin{remark}
Note that, while it would be natural to work on a space that is in certain sense canonical, we work on a general filtered probability space. Though we might restrict to the canonical case later as well, allowing for more generality shall become useful when we treat the principal's problem in {\rm\Cref{relPrincip}}.
\end{remark}

We shall work under the following conditions.
\begin{assumption}\label{asssigmamu} We have
\begin{enumerate}
\item[$(i)$]  $\sigma$ is bounded, such that {\rm\Cref{x}} admits unique strong solution, and ${\rm rank}(\sigma(t,X_{\cdot \wedge t}))=k$, $\mathrm{d}t \otimes \P^0${\rm--a.e.}$;$ 
\item[$(ii)$] $\lambda$ is bounded.
\end{enumerate}
\end{assumption}

\begin{remark} $(i)$ Existence and uniqueness of a strong solution holds if, \emph{e.g.} $\sigma$ is Lipchitz-continuous in $x$ uniformly in $t$ and bounded.

\medskip $(ii)$ Boundedness of the coefficients $\sigma$ and $\lambda$ is, similarly as in \emph{e.g.} {\rm\cite{cvitanic2018dynamic,djete2023stackelberg}}, assumed for simplicity. It can be largely relaxed or adjusted to model-specific assumptions, see for instance {\rm\cite{hernandez2024principal,baldacci2021optimal}}. 

\medskip $(iii)$ The assumption on the rank of the matrix $\sigma$ is, however, fundamental $($see {\rm\Cref{invsigma}} and the proof of {\rm\Cref{ICchar}}$)$ and in a sense reinforces the usual assumption on martingale representation, as assumed \emph{e.g.} in {\rm\cite{cvitanic2018dynamic,hernandez2024principal}}, see  {\rm\Cref{martRep}}. In fact, we can drop this assumption in the case $\F=\F^{X(\cdot)}$ and assume that the martingale representation holds in the sense of {\rm\Cref{martRep}}. See discussions before {\rm\Cref{ICcharFB}} for the definition of $\F^{X(\cdot)}.$

\medskip $(iv)$ Let us also point out that the assumption on the rank is substantially weaker than the ellipticity assumption considered \emph{e.g.} in {\rm\cite{djete2023stackelberg}}, since it allows, for instance, that all coordinates of $X$ are driven by the same one-dimensional Brownian motion. See also {\rm \Cref{rem:rank_mat}} below.
\end{remark}

\begin{remark}\label{rem:rank_mat} The assumption on the rank from {\rm \Cref{asssigmamu}} is, in fact, in certain sense without loss of generality. Indeed, let $\tilde{A} \in \R^{d \times n}$, where $n \in \N^\star,$ be arbitrary. Let us denote $k \coloneqq {\rm rank} (\tilde{A} \tilde{A}^\top).$ Then clearly $d \geq k.$ Let us denote by $ \tilde{A} \tilde{A}^\top=Q \Lambda Q^\top$ the spectral decomposition of $\tilde{A} \tilde{A}^\top$ and let us set \[ A \coloneqq Q\Lambda^\frac{1}{2}_0, \] where $\Lambda_0 \in \R^{d \times k}$ is obtained from $\Lambda$ by omitting potential columns with zeros. Then, $A \in \R^{d \times k}$ and 
\[ A A^\top=Q\Lambda^\frac{1}{2}_0 \Lambda^\frac{1}{2}_0 Q^\top=Q\Lambda^\frac{1}{2} \Lambda^\frac{1}{2} Q^\top=Q\Lambda Q^\top=\tilde{A} \tilde{A}^\top. \] Moreover, ${\rm rank}(A)=k.$ 

\medskip In particular, for every measurable function $\tilde{\sigma} : [0,T] \times \Cc([0,T],\R^d) \longrightarrow \R^{d \times n}$ such that ${\rm rank}(\tilde{\sigma}\tilde{\sigma}^\top(t,x))=k$ for every $(t,x) \in [0,T] \times \Cc([0,T],\R^d),$ there exists a measurable function $\tilde{\sigma} : [0,T] \times \Cc([0,T],\R^d) \longrightarrow \R^{d \times k}$ satisfying  ${\rm rank}(\sigma(t,x))=k$ and $\tilde{\sigma}\tilde{\sigma}^\top(t,x)=\sigma\sigma^\top(t,x)$ for every $(t,x) \in [0,T] \times \Cc([0,T],\R^d).$
\end{remark}

 Abusing notations, we define a $d$-dimensional $(\F,\P^0)$--martingale measure $X$ on $V$ by 
\begin{equation*} 
X_t(\cdot)\coloneqq x_0+\int_0^t \int_\cdot \sigma(s,X_{\cdot \wedge s})M^0(\mathrm{d}v,\mathrm{d}s). 
\end{equation*} 
In particular, $X(V)=X.$ Then, thanks to \Cref{asssigmamu}, we know there exists a measurable function $\sigma_{\rm left}^{-1} : [0,T] \times \Cc([0,T],\R^d) \longrightarrow \R^{k \times d}$ such that
 \[ \sigma_{\rm left}^{-1}(t, X_{\cdot \wedge t})\sigma(t,X_{\cdot \wedge t}) = \text{I}_{k \times k},\; \mathrm{d}t \otimes \P^0\text{\rm--a.e.}  \]
Consequently, we have
\begin{align} \label{invsigma}
M^0_t(\cdot) = \int_0^t \int_\cdot \sigma_{\rm left}^{-1}(t, X_{\cdot \wedge t})\sigma(t,X_{\cdot \wedge t})  M^0(\mathrm{d}v,\mathrm{d}s)=\int_0^t \int_\cdot \sigma_{\rm left}^{-1}(t, X_{\cdot \wedge t})  X(\mathrm{d}v,\mathrm{d}s),\; t\in [0,T],\; \P^0\text{\rm--a.s.}
\end{align}

\begin{remark} \label{rem:corresp} In particular, the augmented filtrations generated by the martingale measures $X$ and $M^0$ coincide. Indeed, the martingale measure $X$ is adapted to the augmented filtration generated by $M^0$ due to existence of a strong solution to {\rm\Cref{x}}. Conversely, $M^0$ is adapted to the augmented filtration generated by $X$ due to equality \eqref{invsigma}.
\end{remark}

\medskip Let us now define the set of controls for the agent. The idea is to make use of \Cref{pf} and define the set of relaxed controls of the agent at time $t \in [0,T]$ as the set of all image measures of the measure $\mu.$ The agent takes control of the process $X$ by means of choosing a $\Pc\Mc(\F) \otimes \Bc(V)$-measurable $U$-valued process
\[
(t,\omega,v) \longmapsto A_t(v)(\omega). 
\]
We denote the set of all controls for the agent by $\Ac$, that is to say
\[ 
\Ac \coloneqq  \big\{ A: [0,T] \times \Omega \times V \longrightarrow U: A\; \text{\rm is}\; \Pc\Mc(\F) \otimes \Bc(V)\text{\rm -measurable}  \big\}.
\]
For every $A \in \Ac,$ we define a measure $\P^A$, equivalent to $\P^0$ (recall that $\lambda$ is bounded) by
\[
\frac{\mathrm{d}\P^A}{\mathrm{d} \P^0}\coloneqq \exp \bigg( \int_0^T\int_V \lambda(s,X_{\cdot \wedge s},A_s(v))\cdot M^0(\mathrm{d}v,\mathrm{d}s) - \frac{1}{2} \int_0^T\int_V \| \lambda(s,X_{\cdot \wedge s},A_s(v))\|^2m^0(\mathrm{d}v,\mathrm{d}s)\bigg). 
\]

Due to Girsanov's theorem, if we define
\[
\widetilde{M}^A_t(\cdot)\coloneqq M^0_t(\cdot)-\int_0^t\int_\cdot \lambda(s,X_{\cdot \wedge s},A_s(v)) m^0(\mathrm{d}v,\mathrm{d}s),\; t \in [0,T],
\] then $\widetilde{M}^A$ is an $(\F,\P^A)$--martingale measure with intensity $m^0.$ One can verify this applying the Girsanov transformation to every process $M^0(A),\; A \in \Bc(V).$ Consequently, $X$ has the following dynamics
\begin{equation*}
X_t=x_0+\int_0^t \int_V\sigma(s,X_{\cdot \wedge s})\lambda(s,X_{\cdot \wedge s},A_s(v)) m^0(\mathrm{d}v,\mathrm{d}s)+\int_0^t\int_V \sigma(s,X_{\cdot \wedge s})\,\widetilde{M}^A(\mathrm{d}v,\mathrm{d}s),\; t \in [0,T],\; \P^{A}\text{\rm--a.s.}
\end{equation*} 

Let us then define $M^A\coloneqq A_\#\widetilde{M}^A$ and $m^A\coloneqq A_\# m^0.$ Then we can write (see \Cref{pf2})
\begin{equation*}
X_t=x_0+\int_0^t \int_U \sigma(s,X_{\cdot \wedge s})\lambda(s,X_{\cdot \wedge s},u) \,m^A(\mathrm{d}u,\mathrm{d}s)+\int_0^t\int_U \sigma(s,X_{\cdot \wedge s})\,M^A(\mathrm{d}u,\mathrm{d}s),\; t \in [0,T],\; \P^{A}\text{\rm--a.s.}
\end{equation*}

\begin{remark}
$(i)$ Due to {\rm \Cref{pf}}, we have that every $\Pc(U)$-valued, $\F$--progressively measurable process can be written as a push-forward of $m^0.$ Hence, our setting exactly corresponds to the usual relaxed control formulation $($see for instance {\rm\citeauthor*{el1990martingale} \cite{el1990martingale}} or {\rm\citeauthor*{haussmann1990existence} \cite{haussmann1990existence}}$)$. We shall from now on interchangeably use the term control both for a function $A \in \Ac$ and for the corresponding measure $m^A$.

\medskip $(ii)$ If the agent should choose a control in the standard sense, that is, choose an $\F$--progressively measurable process $(\omega,t) \longmapsto A_t(\omega)$ or, equivalently, $m^A(\mathrm{d}a,\mathrm{d}s)(\omega)=\delta_{A_s(\omega)}(\mathrm{d}v)\mathrm{d}s$, then $X$ follows \begin{equation*}
X_t=x_0+\int_0^t \sigma(s,X_{\cdot \wedge s})\lambda(s,X_{\cdot \wedge s},A_s) \,\mathrm{d}s+\int_0^t \sigma(s,X_{\cdot \wedge s})\,\mathrm{d}W^A_s,\; t \in [0,T],\; \P^{A}\text{\rm--a.s.,}
\end{equation*} where \[W^A\coloneqq M^A(U)=W^0-\int_0^\cdot \lambda(s,X_{\cdot \wedge s},A_s) \mathrm{d}s\] is an $(\F,\P^A)$--Brownian motion. Our setting is thus an extension of the standard framework as considered in {\rm\cite{cvitanic2018dynamic}}.

\medskip
$(iii)$ In the light of {\rm \Cref{pf}}, we see that the choice of $V$ and $\mu$ is irrelevant. Any diffuse measure on a Polish space would do as we just need to fix a canonical diffuse intensity to define dynamics under $\P^0.$
\end{remark}

\subsection{The agent's problem and incentive compatibility}

We consider contracts consisting of a lump-sum payment at the terminal time $T$. These shall be represented by random variables  $\xi \in \L^0(\Fc_T).$ Let us now formulate the agent's optimisation problem. Let us be given a contract $\xi \in \L^0(\Fc_T)$ and a Borel-measurable running cost for the agent $f : [0,T]\times \Cc([0,T],\R^d) \times U \longrightarrow \R.$ Further, let $U_a: \R \longrightarrow \R$ be a Borel-measurable utility function, and consider a discount factor characterised by a Borel-measurable function $k: [0,T] \times \Cc([0,T],\R^d) \times U \longrightarrow \R.$ Our main assumptions on these functions are given below.

\begin{assumption}\label{assUkf} We have
\begin{itemize}
\item[$(i)$] $U_a$ is injective$;$
\item[$(ii)$] $k$ is uniformly bounded from below$;$
\item[$(iii)$] the function $\hat{f}: [0,T] \times \Cc([0,T],\R^d) \longrightarrow \R \cup \{ - \infty\}$ defined by \begin{align} \label{hatf}
\hat{f}(s,x)=\inf_{u \in U}\big\{ f(s,x,u) \big\},\; (s,x) \in [0,T] \times \Cc([0,T],\R^d),
\end{align} is Borel-measurable and there exists $p>1$ such that
\begin{align*} 
\E^{\P^\smallfont{0}} \bigg[ \bigg(\int_0^T \int_U  \hat{f}^2(s,X_{\cdot \wedge s})  \mathrm{d}s  \bigg)^\frac{p}{2}\bigg] < \infty.
\end{align*}
\end{itemize}
\end{assumption}

\begin{remark} Injectivity of the utility function is assumed since we shall deal with its inverse mapping in what follows. One-sided boundedness of the discount factor is commonly assumed in the literature and is of technical nature. While it may in principle be relaxed, we might in general face technical issues when defining the value process of the agent. At last, let us mention that the last item in {\rm\Cref{assUkf}} is assumed for well-posedness of the {\rm BSDE} in {\rm \Cref{ICchar}}, see {\rm \Cref{lips}}. Vaguely speaking, it asserts that while the running cost $f$ might be unbounded below or even attain the value $-\infty,$ it is sufficiently integrable for at least some controls. One such example would be the quadratic cost.
\end{remark}

Let us set for notational simplicity
\[
K_{t}^A(X_{\cdot \wedge t})\coloneqq \exp \bigg( - \int_0^t \int_U k(s,X_{\cdot \wedge s},u) m^A(\mathrm{d}u,\mathrm{d}s) \bigg),\; t \in [0,T],\; A\in\Ac.
\]

Let us denote by $V^a_0(\xi)$ the utility promised to the agent. That is to say
\[ 
V^a_0(\xi)\coloneqq \sup_{A \in \Ac} v^a_0(\xi,A), \;\text{with}\;  v^a_0(\xi,A)\coloneqq\E^{\P^{\smallfont{A}}} \bigg[K_{T}^A(X_{\cdot \wedge T}) U_a(\xi)- \int_0^T \int_U K_{s}^A(X_{\cdot \wedge s}) f(s,X_{\cdot \wedge s},u)m^A(\mathrm{d}u,\mathrm{d}s)  \bigg],\; A \in \Ac. 
\]

Finally, $\Ac^\star(\xi)$ denotes the set of optimal responses to $\xi$, that is
\[ 
\Ac^\star(\xi)\coloneqq \big\lbrace A \in \Ac: V^a_0(\xi)=v^a_0(\xi,A) \big\rbrace.
\]

Among all contracts $\xi \in \L^0(\Fc_T),$ we are interested in those to which the agent has an optimal response. The following definition is up to slight model-specific variations standard in the literature, see for instance \cite{cvitanic2018dynamic,djete2023stackelberg,hernandez2024principal}.

\begin{definition}
A contract $\xi \in \L^0(\Fc_T)$ is called incentive compatible, denoted by $\xi \in \IC$, if
\begin{itemize}
\item[$(i)$] there is $p>1$ such that $ \E^{\P^\smallfont{0}} [\lvert U_a(\xi) \rvert^{p}]<\infty;$
\item[$(ii)$] $\Ac^\star(\xi) \neq \emptyset.$
\end{itemize}
\end{definition}

\begin{remark}\label{int} Note that if $U \in \L_n^p (\Fc_T,\P^0),$ for some $p>1$ and $n \in \N^\star,$ then  $U \in \L^q_n(\Fc_T,\P^A),$ for any $1<q<p$ and $A \in \Ac.$ This can be justified as follows
\[
\E^{\P^\smallfont{A}} \big[ \lVert U \rVert^{q}\big]=\E^{\P^\smallfont{0}}\bigg[ \frac{\mathrm{d}\P^A}{\mathrm{d}\P^0} \lVert U \rVert^{q} \bigg] \leq \bigg(\E^{\P^\smallfont{0}}\bigg[ \frac{\mathrm{d}\P^A}{\mathrm{d}\P^0}^\frac{r}{r-1} \bigg]\bigg)^\frac{r-1}{r} \Big( \E^{\P^\smallfont{0}}\big[\lVert U \rVert^{q  r}\big] \Big)^\frac{1}{r},\; r \in (1,\infty).
\]
It suffices to take $r$ such that $qr<p$ and realise that the Radon--Nikod\'{y}m density $\mathrm{d}\P^A/\mathrm{d}\P^0$ has finite moments of any order under $\P^0$ due to boundedness of $\lambda.$ As a consequence, if $\xi \in \Ic\Cc,$ then $v^a_0(\xi,A)$ is well-defined as an element of $[-\infty, \infty)$ for any $A\in \Ac,$ thanks to {\rm \Cref{assUkf}}.
\end{remark}

Motivated by {\rm\citeauthor*{cvitanic2018dynamic} \cite{cvitanic2018dynamic}}, our aim is to give a characterisation of incentive compatible contracts by means of a dynamic programming representation. Having such a result not only reformulates the principal's problem in a more tractable way, but also provides us with the set of optimal responses for the agent to a given contract. Since this approach heavily depends on the theory of BSDEs and our framework does not fit within the usual setting, we need to introduce suitable spaces of functions and an appropriate notion of BSDEs. Let us first define the Hamiltonian corresponding to the agent's problem. We set
\begin{align*} 
h(t,x,y,z,u)&\coloneqq\sigma(t,x)\lambda(s,x,u)\cdot z - f(s,x,u) - k(t,x,u)y, \;(t,x,y,z,u) \in [0,T]\times\Cc([0,T],\R^d)\times\R\times\R^d\times U, \\
H(t,x,y,z)&\coloneqq\sup_{u \in U}\big\{  h(t,x,y,z,u)\big\},\; (t,x,y,z) \in [0,T]\times\Cc([0,T],\R^d)\times\R\times\R^d.
\end{align*}

Note that this definition agrees with the standard Hamiltonian as known from optimal control theory, see \emph{e.g.} {\rm \citeauthor*{touzi2013optimal} \cite{touzi2013optimal}}. We impose the following condition, which is also standard.

\begin{assumption} \label{ass:exist.minim}
There exists at least one Borel-measurable function $a : [0,T] \times \Cc([0,T],\R^d) \times \R \times \R^d \longrightarrow U$ with
\[
H(t,x,y,z)=h\big(t,x,y,z,a(t,x,y,z)\big),\;  (t,x,y,z) \in [0,T]\times\Cc([0,T],\R^d)\times\R\times\R^d.
\]
We denote by $\Upsilon$ the set of all such functions $a.$
\end{assumption}

\begin{remark}
Note that under this assumption, the function $\hat{f}$ defined by {\rm\Cref{hatf}} is Borel-measurable since $\hat{f}(t,x)=-H(t,x,0,0).$
\end{remark}

Let us, for $p \in [1,\infty)$, define the following spaces of processes.
\begin{itemize}
\item $\S^p(\F,\P^0)$ denotes the space of $\R$-valued, $\F$-adapted, c\`{a}dl\`{a}g  processes $Y$ such that 
\[
\| Y \|_{\S^\smallfont{p}(\F,\P^\smallfont{0})}^p\coloneqq \E^{\P^\smallfont{0}}\bigg[ \sup_{t \in [0,T]}  \lvert Y_s\rvert ^p \bigg] < \infty. 
\]
\item $\L^p_d(\F,\P^0,X(\cdot))$ is the space of all $\Pc\Mc(\F) \otimes \Bc(V)$-measurable processes $Z : \Omega \times [0,T] \times V \longrightarrow \R^d$ such that 
\[
\|  Z \|_{\L^\smallfont{p}_\smallfont{d}(\F,\P^\smallfont{0},X(\cdot))}^p\coloneqq    \E^{\P^\smallfont{0}} \bigg[  \int_0^T \int_V \big\| Z_s(v)^\top \sigma(x,X_{\cdot \wedge s}) \big\|^2 m^0(\mathrm{d}v,\mathrm{d}s) \bigg]^\frac{p}{2}< \infty. \]
\item $\M^p(\F,\P^0)$ is the space of all $\R$-valued c\'{a}dl\'{a}g, $(\F,\P^0)$-martingales $M$ such that $M_0=0$, $\P^0$--a.s., and 
\[
\| M \|^p_{\M^\smallfont{p}(\F,\P^\smallfont{0})}\coloneqq \E^{\P^0}\big[ [M]_T^\frac{p}{2}\big] < \infty. 
\]
\item $\M^p_{X(\cdot)^\smallfont{\perp}}(\F,\P^0)$ is the subspace of $\M^p(\F,\P^0)$ consisting of all $L \in \M^p(\F,\P^0)$ such that $[ X(\cdot), L ] \equiv 0,$ where the latter denotes the property
 \[
 \bigg[ \int_0^\cdot \int_V Z_s(v)\cdot X(\mathrm{d}v,\mathrm{d}s), \int_0^\cdot\,\mathrm{d}L_s \bigg] \equiv 0,\; Z \in \L^2_d(\F,\P^0,X(\cdot)). 
 \]
\end{itemize}
 
We further define the space $\S^p_{\rm loc}(\F,\P^0)$ as the set of $\R$-valued, $\F$-adapted c\`{a}dl\`{a}g processes $Y$ for which there exists a non-decreasing sequence of $\F$--stopping times $(\tau_n)_{n \in \N^\smallfont{\star}}$ such that the stopped process $Y_{\cdot \wedge \tau_\smallfont{n}}$ belongs to $\S^p(\F,\P^0)$ for every $n \in \N^\star$ and $\lim_{n \rightarrow \infty}\tau_n = T$, $\P^0$--a.s. Spaces $\L^p_{d,\mathrm{loc}}(\F^0,\P^0,X(\cdot)),$ $\M^p_{\rm loc}(\F,\P^0),$ and $\M^p_{X(\cdot)^\smallfont{\perp},\mathrm{loc}}(\F,\P^0)$ are defined analogously.

\medskip
The following theorem gives the main reduction result of this section. The correspondence to Theorem 4.2 in \cite{cvitanic2018dynamic} is to be noted. Let us emphasise that, unlike in \cite{cvitanic2018dynamic}, we need to consider an orthogonal martingale in our setting. This is due to having a general filtration $\F.$ As we shall see later, one can get rid of this part by restricting to the filtration generated by the martingale measure $X,$ which would then correspond to the standard form of a BSDE. However, working with a general filtration shall prove useful in the sequel.

\begin{theorem} \label{ICchar}
Let {\rm Assumptions \ref{asssigmamu}, \ref{assUkf}} and {\rm\ref{ass:exist.minim}} hold and let $\xi \in \L^0(\Fc_T).$ Then, $\xi \in \IC$ if and only if $\xi=U_a^{-1}(Y_T)$ for some $Y \in \S^p(\F,\P^0)$ with $Y_T \in U_a(\R),\; \P^0\text{\rm--a.s.,}$ of the form
\begin{align} \label{ICcontract} 
Y_t&=Y_0-\int_0^t \int_V H(s,X_{\cdot \wedge s},Y_s,Z_s(v))m^0(\mathrm{d}v,\mathrm{d}s) + \int_0^t \int_V Z_s(v) \cdot X(\mathrm{d}v,\mathrm{d}s) + \int_0^t\mathrm{d}L_s,\; t \in [0,T],\;\P^0\text{\rm--a.s.,} 
\end{align} 
for some $(Z,L) \in  \L^p_d(\F,\P^0,X(\cdot))\times \M^p_{X(\cdot)^\smallfont{\perp}}(\F,\P^0).$ Moreover, in the affirmative case, it holds $\E^{\P^0}[Y_0]=V_0^a(\xi)$ and $\Ac^\star(\xi)= \{ A^\star_a : a \in \Upsilon\},$ where for any $a \in \Upsilon$ we define
\[ A^\star_a : (t,\omega,v) \longmapsto a(t,X_{\cdot \wedge t}(\omega),Y_t(\omega),Z_t(v)(\omega)),\; (\omega, t,v) \in \Omega \times [0,T] \times V.  \]
\end{theorem}
\begin{proof}
Let us first assume that $\xi=U_a^{-1}(Y_T)$ for some $Y \in \S^p(\F,\P^0),$ $Y_T \in U_a(\R),\; \P^0\text{\rm--a.s.,}$ of the form \eqref{ICcontract}. Clearly, $U_a(\xi) \in \L^p(\Fc_T,\P^0)$. We need to verify $\Ac^\star(\xi) \neq \emptyset.$ To this end, let us define a family of processes 
\begin{align*}
R_t^A&\coloneqq K_t^A(X_{\cdot \wedge t})Y_t-\int_0^t \int_V K_s^A(X_{\cdot \wedge s})f(s,X_{\cdot \wedge s},A_s(v)) m^0(\mathrm{d}v,\mathrm{d}s) \\
&=K_t^A(X_{\cdot \wedge t})Y_t-\int_0^t \int_V K_s^A(X_{\cdot \wedge s})f(s,X_{\cdot \wedge s},u) m^A(\mathrm{d}u,\mathrm{d}s),\; t \in [0,T],\; A \in \Ac.
\end{align*} 
Note that $R^A$ is a well-defined process with values in $[-\infty, \infty)$ for any $A \in \Ac.$ We have that 
\[
R_T^A=K_T^A(X_{\cdot \wedge T}) U_a(\xi)-\int_0^T \int_V K_s^A(X_{\cdot \wedge s}) f(s,X_{\cdot \wedge s},u)\,m^A(\mathrm{d}u,\mathrm{d}s),
\] 
and 
\begin{align*} 
R^A_t&=K_T^A(X_{\cdot \wedge T}) U_a(\xi)+\int_t^T \int_V K_s^A(X_{\cdot \wedge s}) \big(H(s,X_{\cdot \wedge s},Y_s,Z_s(v))-\sigma(s,X_{\cdot \wedge s})\lambda(s,X_{\cdot \wedge s},A_s(v))\cdot Z_s(v)\big)m^0(\mathrm{d}v,\mathrm{d}s)\\
&\quad+\int_t^T \int_V  K_s^A(X_{\cdot \wedge s})\big( f(s,X_{\cdot \wedge s},A_s(v)) +k(s,X_{\cdot \wedge s},A_s(v))Y_s\big)m^0(\mathrm{d}v,\mathrm{d}s)\\
&\quad  - \int_t^T \int_V K_s^A(X_{\cdot \wedge s}) Z_s(v)^\top \sigma(s,X_{\cdot \wedge s}) \widetilde{M}^A(\mathrm{d}v,\mathrm{d}s) - \int_t^T K_s^A(X_{\cdot \wedge s})\mathrm{d}L_s,\; t \in [0,T],\; \P^A\text{\rm--a.s.} 
\end{align*} 
Using \Cref{invsigma} we conclude that $[ X(\cdot),L ] \equiv 0$ implies $[ M^0(\cdot),L ] \equiv 0$ and, hence, $L$ remains an $(\F,\P^A)$-martingale,  since it is $\P^A$-integrable due to \Cref{int} and for any $t\in[0,T]$
\begin{align*}
\E^{\P^A}[L_T \vert \Fc_t]=\frac{1}{\E^{\P^\smallfont{0}}\big[\frac{\mathrm{d}\P^\smallfont{A}}{\mathrm{d}\P^\smallfont{0}}\big\vert \Fc_t\big]}\E^{\P^\smallfont{0}}\bigg[\frac{\mathrm{d}\P^A}{\mathrm{d}\P^0} L_T\bigg\vert \Fc_t\bigg] =\frac{1}{\E^{\P^\smallfont{0}}\big[\frac{\mathrm{d}\P^\smallfont{A}}{\mathrm{d}\P^\smallfont{0}}\big\vert \Fc_t\big]}\E^{\P^\smallfont{0}}\bigg[\frac{\mathrm{d}\P^A}{\mathrm{d}\P^0} \bigg\vert \Fc_t\bigg] \E^{\P^\smallfont{0}}[ L_T\vert \Fc_t]=L_t,\; \P^A\text{\rm--a.s.}
\end{align*}
It follows that $R^A$ is an $(\F,\P^A)$--super-martingale for any $A \in \Ac$, and it is a martingale if and only if $A=A^\star_a$ for some $a \in \Upsilon.$ Integrability can be verified as in \Cref{int}. It is then straightforward to verify that $\Ac^\star(\xi)=\{A^\star_a : a \in \Upsilon \}$ and $\E^{\P^0}[Y_0]=V_0^a(\xi).$ In particular, $\xi \in \Ic\Cc.$

\medskip 
For the converse implication, let us assume that $\xi \in \IC.$ Since the functions $\sigma$ and $\lambda$ are assumed to be bounded and $k$ is bounded from below, it follows that $H$ is Lipschitz-continuous in $(y,z)$, uniformly in the other variables. Because $U_a(\xi) \in \L^p(\Fc_T,\P^0)$ by assumption, we conclude using \Cref{Exbsde} that the BSDE
\begin{align*} 
Y_t&=U_a(\xi)+\int_t^T \int_V H(s,X_{\cdot \wedge s},Y_s,Z_s(v))m^0(\mathrm{d}v,\mathrm{d}s) - \int_t^T \int_V Z_s(v) \cdot X(\mathrm{d}v,\mathrm{d}s) - \int_t^T\mathrm{d}L_s,\; t \in [0,T],\; \P^0\text{\rm--a.s.,} 
\end{align*} 
has a unique solution $(Y,Z,L) \in \S^p(\F,\P^0) \times \L^p_d(\F,\P^0,X(\cdot))\times \M^p_{X(\cdot)^\smallfont{\perp}}(\F,\P^0).$ Since $U_a$ is injective, it is clear that $\xi=U_a^{-1}(Y_T).$ Moreover, it follows from the first part that $\Ac^\star(\xi)= \{ A^\star_a : a \in \Upsilon\}.$
\end{proof}

As mentioned earlier, one might ditch the orthogonal martingale part by restricting to a smaller filtration and obtain a result resembling the standard form of BSDEs. This is achieved thanks to martingale representation for martingale measures, see \Cref{martRep}.

\medskip
More specifically, let us define $\F^{X(\cdot)}=(\Fc_t^{X(\cdot)})_{t\in[0,T]}$ as the $\P^0$-augmentation of $\F^{o,X(\cdot)}=(\Fc_t^{o,X(\cdot)})_{t\in[0,T]}$ where
\[
 \Fc^{o,X(\cdot)}_t\coloneqq \sigma \big(X_s(A) : s \in [0,t],\; A \in \Bc(V) \big),\; t \in [0,T]. 
\]
Note that $\F^{X(\cdot)}$ satisfies the usual conditions by definition and $\Fc_0^{X(\cdot)}$ is $\P^0$-trivial due to \Cref{blumethal}.
\begin{theorem} \label{ICcharFB}
Assume $\F=\F^{X(\cdot)}$ and let {\rm Assumptions \ref{asssigmamu}, \ref{assUkf}} and {\rm\ref{ass:exist.minim}} hold. Then $\xi \in \L^0(\Fc_T)$ is an incentive compatible contract if and only if $\xi=U_a^{-1}(Y_T)$ for some $Y \in \S^p(\F,\P^0)$ with $Y_T \in U_a(\R),\; \P^0\text{\rm--a.s.,}$ of the form
\begin{align} \label{ICcontractFB} 
Y_t&=y_0-\int_0^t \int_V H(s,X_{\cdot \wedge s},Y_s,Z_s(v))m^0(\mathrm{d}v,\mathrm{d}s) + \int_0^t \int_V Z_s(v) \cdot X(\mathrm{d}v,\mathrm{d}s),\; t \in [0,T],\;\P^0\text{\rm--a.s.,}
\end{align} 
for some $Z \in \L^p_d(\F,\P^0,X(\cdot))$ and $y_0 \in \R.$ Moreover, in the affirmative case it holds $y_0=V_0^a(\xi),$ and $\Ac^\star(\xi)= \{ A^\star_a : a \in \Upsilon\},$ where for $a \in \Upsilon,$ we define 
\[ A^\star_a : (\omega, t,v) \longmapsto a(t,X_{\cdot \wedge t}(\omega),Y_t(\omega),Z_t(v)(\omega)),\; (\omega, t,v) \in \Omega \times [0,T] \times V.  \]
\end{theorem}

\begin{proof} 
Analogous arguments as in the proof of Theorem \ref{ICchar} using \Cref{Ex2bsde} and \Cref{blumethal}.
\end{proof}

\begin{remark} Since the agent controls the drift only, it would, generally speaking, be possible to introduce the Girsanov change of measure by means of
\[ \frac{\mathrm{d}\P^{m^\smallfont{A}}}{\mathrm{d}\P^0}\coloneqq  \exp \bigg( \int_0^T\int_V \lambda(s,X_{\cdot \wedge s},a) m_s^A(\mathrm{d}a) \cdot \mathrm{d}W^0_s - \frac{1}{2} \int_0^T \bigg\| \int_V \lambda(s,X_{\cdot \wedge s},A_s(v)) m^A_s(\mathrm{d}a)\bigg\|^2 \mathrm{d}s\bigg).\]
Indeed, one then obtains that the process \[ W^{m^\smallfont{A}} \coloneqq W^0-\int_0^\cdot  \int_A\lambda(s,X_{\cdot \wedge s},a) m_s^A(\mathrm{d}a) \mathrm{d}s, \] is an $(\F,\P^{m^\smallfont{A}})$--Brownian motion, thus ending up with the right relaxed dynamics for the process $X$. However, when $\xi$ is $\Fc_T^{X(\cdot)}$-measurable, the correct form of the {\rm BSDE} in this formulation would be the usual
\[
Y_t=\xi + \int_t^T \int_V H(s,X_{\cdot \wedge s},Y_s,Z_s) \mathrm{d}s - \int_t^T \int_V Z_s \cdot \mathrm{d}W^0_s - \int_t^T \mathrm{d}L_s,\; t \in [0,T],\;\P^0\text{\rm--a.s.,} \] with $(Y,Z,L)$ lying in appropriate spaces and $L$ being orthogonal to $W^0.$ In fact, it can readily be seen that the agent's control problem would agree with the standard weak control problem. Thus, even though the principal would be allowed to choose contracts from a larger class, the agent would not be allowed to control the dynamics of $X(\cdot)$ appropriately and even though there would be an optimal control, the class of his admissible responses would be somewhat restricted. In our formulation, we clearly see that whenever $Z$ is randomised, \emph{i.e.} it depends on $v$ in a non-trivial way, the optimal action of the agent is also necessarily randomised unless $a \in \Upsilon$ is somehow degenerate, which typically does not occur. Since the process $Z$ plays the role of a control for the principal in the reformulated problem, see {\rm \Cref{thm:ref_principal}}, we shall allow for it to be randomised to be able to use compactification techniques. However, we keep in mind that if $\xi$ is $\Fc_T^{X}$-measurable, where $\F^X$ is the filtration generated by the process $X$, or if $Z$ is not randomised, then the agent does not have to randomise his actions either and we recover the standard weak framework, see also {\rm \Cref{sec:non-relaxed}}.
\end{remark}

\section{Relaxed principal's problem: different points of view} \label{sec:principal}
The principal is concerned with finding a contract $\xi$ among all incentive compatible contracts maximising her utility. Doing so, she anticipates the agent's behaviour and expects him to behave in an optimal way. Let us be given reservation utility $r_0 \in \R \cup \{-\infty\}.$ We assume that the agent has the participation constraint 
\[
V_0^a(\xi) \geq r_0.
\]
In other words, we assume that the agent will refuse the contract should his utility fall below threshold $r_0$. 

\medskip We are given a Borel-measurable utility function $U_p : \Cc([0,T],\R^d) \times \R \longrightarrow \R$ for the principal. We further assume that she is allowed to only choose within the class of contracts satisfying certain constraints given by 
\[
\E^{\P^{A^\smallfont{\star}}} \big[g^i(X_{\cdot \wedge T},\xi)\big]\leq 0,\;k\in\N^\star,\; i \in I,
\] 
where $A^\star\in \Ac^\star(\xi)$ and $I$ is a---possibly uncountable---index set and for any $i\in I$, $g^i : \Cc([0,T],\R^d) \times \R \longrightarrow \R \cup \{ \infty \}$ are given Borel-measurable functions. 

\begin{remark} \label{rem:constraints} $(i)$ We note that our setting allows for more general constraints without further complication of the analysis. More specifically, one can consider constraints of the type
\[ \E^{\P^{A^\smallfont{\star}}} \bigg[ \int_0^T \int_{\R^d} f^i(s,X_{\cdot \wedge s},z) m^Z(\mathrm{d}z,\mathrm{d}s) + g^i(X_{\cdot \wedge T},\xi) \bigg]\leq 0,\; i \in I, \] 
where $A^\star \in \Ac^\star(\xi)$ and for any $i\in I$, $f^i: [0,T]\times \Cc([0,T],\R^d) \times \R^d \longrightarrow \R \cup \{ \infty \}$ and $g^i : \Cc([0,T],\R^d) \times \R \longrightarrow \R \cup \{ \infty \}$ are Borel-measurable maps.

\medskip
$(ii)$ The type of constraints we consider also includes so-called hard constraints. Let us fix a set $C \in \Bc(\Cc([0,T],\R^d)) \otimes \Bc(\R).$ If we define \[g(x,c)\coloneqq 1-\mathbf{1}_{\{(x,c)\in C\}},\; (x,c) \in \Cc([0,T],\R^d) \times \R,\] then $\E^{\P^{A^\smallfont{\star}}}[g(X_{\cdot \wedge T},\xi)]\leq 0$ holds if and only if $\P^{A^\smallfont{\star}}[(X_{\cdot \wedge T},\xi) \in C]=1.$ If $C$ is closed, then $g$ is lower-semicontinuous.
\end{remark}

\begin{example} \label{constraints} One example of such a constraint is imposing
\[0 \leq \xi \leq \ell(X_{\cdot \wedge T}). \] That is to say, the contract is not allowed to exceed the liquidation value of the project and is required to be non-negative. Note that this constraint can be rewritten as
\begin{equation*}
\E^{\P^{A^\smallfont{\star}}} [g^1(X_{\cdot \wedge T},\xi)] \leq 0,\; {\rm and}\; \E^{\P^{A^\smallfont{\star}}} [g^2(X_{\cdot \wedge T},\xi)] \leq 0,
\end{equation*} with
\begin{align*}  
g^1(x,c) \coloneqq1- \mathbf{1}_{\{c\geq 0\}}, \; {\rm and}\;
g^2(x,c) \coloneqq1- \mathbf{1}_{\{\ell(x) \geq c \}},\; (x,c) \in \Cc([0,T],\R^d)\times \R.
\end{align*}
Moreover, if $\ell$ is upper-semicontinuous, then $g^i,\; i \in \{1,2\},$ are jointly lower-semicontinuous.
\end{example}

The problem of the principal can be formulated as follows
\begin{align*} 
V_0^p \coloneqq \sup_{(\xi,A^\smallfont{\star}) \in \Phi}\E^{\P^{A^\smallfont{\star}}} \big[U_p(X_{\cdot \wedge T},\xi)\big],
\end{align*}
where
\[ \Phi=\big\{(\xi,A^\star) \in \Ic\Cc \times \Ac^\star(\xi): V_0^a(\xi) \geq r_0,\; \E^{\P^{A^\smallfont{\star}}} [g^i(X_{\cdot \wedge T},\xi)]\leq 0  ,\; i \in I \big\}. \]

Note that we assume that, if indifferent, the agent chooses a response which benefits the principal. This is a standard simplification, see \emph{e.g.} \cite{cvitanic2018dynamic,elie2019tale,djete2023stackelberg}, and can be interpreted as a recommendation from the principal to the agent. 

\medskip
Since the principal maximises an expectation under $\P^{A^\smallfont{\star}},$ it shall be useful to rewrite dynamics of the value process $Y$ under this measure. Let us thus fix a contract $\xi \in \IC$ and let $A^\star \in \Ac^\star(\xi).$ We then have by \Cref{ICchar} that $A^\star=A^\star_a$ for some $a \in \Upsilon$ and, by Girsanov theorem, it holds
\begin{align*} Y_t&=U_a(\xi)-\int_t^T \int_V \Big(f\big(s,X_{\cdot \wedge s},a(s,X_{\cdot \wedge s},Y_s,Z_s(v))\big)+k\big(s,X_{\cdot \wedge s},a(s,X_{\cdot \wedge s},Y_s,Z_s(v))\big)Y_s \Big)m^{0}(\mathrm{d}v,\mathrm{d}s) \\ 
&\quad- \int_t^T \int_V Z_s(v)^\top \sigma(s,X_{\cdot \wedge s}) \widetilde{M}^{A^\smallfont{\star}}(\mathrm{d}v,\mathrm{d}s) - \int_t^T\mathrm{d}L_s,\; t \in [0,T],\; \P^{A_a^\smallfont{\star}}\text{\rm--a.s.}, 
\end{align*} 
and, if we define a martingale measure on $\R^d$ and its intensity by $M^Z\coloneqq Z_\# \widetilde{M}^{A^\smallfont{\star}},\; \text{\rm and}\; m^Z\coloneqq Z_\# m^0,$ then 
\begin{align*} 
Y_t&=U_a(\xi)-\int_t^T \int_V \Big(f\big(s,X_{\cdot \wedge s},a(s,X_{\cdot \wedge s},Y_s,z)\big)+k\big(s,X_{\cdot \wedge s},a(s,X_{\cdot \wedge s},Y_s,z)\big)Y_s\Big)m^{Z}(\mathrm{d}z,\mathrm{d}s) \\
&\quad - \int_t^T \int_V z^\top\sigma(s,X_{\cdot \wedge s}) M^{Z}(\mathrm{d}z,\mathrm{d}s) - \int_t^T\mathrm{d}L_s\\
&=Y_0+\int_0^t \int_V \Big(f\big(s,X_{\cdot \wedge s},a(s,X_{\cdot \wedge s},Y_s,z)\big)+k\big(s,X_{\cdot \wedge s},a(s,X_{\cdot \wedge s},Y_s,z)\big)Y_s\Big)m^{Z}(\mathrm{d}z,\mathrm{d}s) \\
&\quad + \int_0^t \int_V z^\top\sigma(s,X_{\cdot \wedge s}) M^{Z}(\mathrm{d}z,\mathrm{d}s) + \int_0^t\mathrm{d}L_s,\; t \in [0,T],\; \P^{A_a^\smallfont{\star}}\text{\rm--a.s.}, 
\end{align*}

and the agent's optimal control can be written as $m^{A^\smallfont{\star}}=\hat{A}_{a \#}^\star m^Z,$ where $\hat{A}^\star_{a}$ is the $\Pc\Mc(\F) \otimes \Bc(\R^d)$-measurable process given by
\[ \hat{A}^\star_{a} : (\omega, t, z) \longmapsto a(t,X_{\cdot \wedge t}(\omega),Y_t(\omega),z),\; (\omega, t,z) \in \Omega \times [0,T] \times \R^d.  \]

The aim now is to reformulate the principal's optimisation problem as a dynamic control problem by means of \Cref{ICchar}, resp. \Cref{ICcharFB}, thus making use of the approach introduced in \cite{cvitanic2018dynamic}. To that end, let us introduce the following notation.

\begin{definition} $(i)$ $\Vc$ denotes the set of all $\Pc(\R^d)$-valued processes $m^Z$ for which there is $p>1$ and $Z \in \L^p_d(\F,\P^0,X(\cdot))$ such that $m^Z=Z_\#m^0.$ Equivalently, thanks to {\rm\Cref{pf}}, this is the set of all $\F$--progressively measurable, $\Pc(\R^d)$-valued processes $m^Z$ such that there exists $p>1$ for which 
\[
\E^{\P^\smallfont{0}} \bigg[  \bigg(\int_0^T \int_{\R^\smallfont{d}} \big\| z^\top\sigma(x,X_{\cdot \wedge s}) \big\|^2\ m^Z(\mathrm{d}z,\mathrm{d}s)\bigg)^\frac{p}{2}\bigg]< \infty. 
\]
$(ii)$ $\Wc$ denotes the set of all martingales $L \in \M^p_{X(\cdot)^\smallfont{\perp}}(\F,\P^0)$ for some $p>1.$ That is, $\Wc\coloneqq\bigcup_{p>1} \M^p_{X(\cdot)^\smallfont{\perp}}(\F,\P^0).$

\medskip
$(iii)$ $\mathfrak{Y}$ denotes the set of all $Y_0 \in \L^p(\Fc_0,\P^0)$ for some $p>1.$ That is, $ \mathfrak{Y}\coloneqq\bigcup_{p>1} \L^p(\Fc_0,\P^0).$
\end{definition}

\begin{remark} Given fixed $(Y_0, m^Z,L) \in \mathfrak{Y} \times \Vc \times \Wc,$ there exists $p>1$ and a unique process $Y \in \S^p(\F,\P^0)$ such that \begin{equation} \label{forwardY}
Y_t=Y_0-\int_0^t \int_V H(s,X_{\cdot \wedge s},Y_s,z)m^Z(\mathrm{d}z,\mathrm{d}s) + \int_0^t \int_V z^\top\sigma(x,X_{\cdot \wedge s}) \cdot M^Z(\mathrm{d}z,\mathrm{d}s) + \int_0^t\mathrm{d}L_s,\; t \in [0,T],\;\P^0\text{\rm--a.s.} 
\end{equation}
In other words, $Y$ is the unique strong solution to \Cref{forwardY}. This is due to the Lipschitz-continuity of the coefficients. To emphasise dependence on $(Y_0,m^Z,L),$ we shall write $Y^{Y_\smallfont{0},m^\smallfont{Z},L}$ for the corresponding process.

\medskip
Similarly, if $\F=\F^{X(\cdot)},$ then given $(y_0, m^Z) \in \R \times \Vc,$ there exists a unique strong solution in $\S^p(\F,\P^0)$ for some $p>1$ to the equation \begin{align*}
Y_t&=y_0-\int_0^t \int_V H(s,X_{\cdot \wedge s},Y_s,z)m^Z(\mathrm{d}z,\mathrm{d}s) + \int_0^t \int_V z^\top\sigma(x,X_{\cdot \wedge s}) \cdot M^Z(\mathrm{d}z,\mathrm{d}s),\; t \in [0,T],\;\P^0\text{\rm--a.s.,}
\end{align*} and we shall write $Y^{y_\smallfont{0},m^\smallfont{Z}}$ for the corresponding process.
\end{remark}
We are now ready to reformulate the problem of the principal. Thanks to \Cref{ICchar,ICcharFB} we have the following result. The proof is straightforward and is thus omitted.

\begin{theorem} \label{thm:ref_principal} Under {\rm Assumptions \ref{asssigmamu}, \ref{assUkf}} and {\rm\ref{ass:exist.minim}}, we have the following equality
\begin{multline} 
 V_0^p= \sup \Big\{ \E^{\P^{A^\smallfont{\star}_{\smallfont a}}}\big[ U_p(X_{\cdot \wedge T},U_a^{-1}(Y_T^{Y_\smallfont{0},m^\smallfont{Z},L}))\big] : (Y_0,a,m^Z,L) \in \mathfrak{Y} \times \Upsilon \times \Vc \times \Wc, \; \E^{\P^{A^\smallfont{\star}_{\smallfont a}}}[Y_0] \geq r_0,\; Y_T^{Y_\smallfont{0},m^\smallfont{Z},L} \in U_a(\R),\\
 \E^{\P^{A^\smallfont{\star}_{\smallfont a}}}\big[g^i\big(X_{\cdot \wedge T},U_a^{-1}(Y_T^{Y_\smallfont{0},m^\smallfont{Z},L})\big)\big]\leq 0,\; i \in I  \Big\}. \label{generalF} \end{multline}
Moreover, if $\F=\F^{X(\cdot)},$ then 
\begin{multline}  V_0^p= \sup \Big\{ \E^{\P^{A^\smallfont{\star}_{\smallfont{a}}}}\big[U_p(X_{\cdot \wedge T},U_a^{-1}(Y_T^{y_\smallfont{0},m^\smallfont{Z}}))\big]  : (y_0,a,m^Z) \in [r_0,\infty) \times \Upsilon \times \Vc, \\
Y_T^{y_\smallfont{0},m^\smallfont{Z}} \in U_a(\R),\; \E^{\P^{A^\smallfont{\star}_{\smallfont a}}}\big[g^i\big(X_{\cdot \wedge T},U_a^{-1}(Y_T^{y_\smallfont{0},m^\smallfont{Z}})\big)\big]\leq 0,\; i \in I  \Big\}. \label{FX} 
\end{multline}
Besides, every maximiser $(Y_0^\star,m^{Z,\star},L^\star)$, resp. $(y_0^\star,m^{Z,\star}),$ gives an optimal contract $\xi^\star = U_a^{-1}(Y_T^{Y_\smallfont{0}^\smallfont{\star},m^{\smallfont{Z}\smallfont{,}\smallfont{\star}},L^\smallfont{\star}})$, resp. $\xi^\star = U_a^{-1}(Y_T^{y_\smallfont{0}^\smallfont{\star},m^{\smallfont{Z}\smallfont{,}\smallfont{\star}}}).$
\end{theorem}

\begin{remark}
$(i)$ Note that equality \eqref{FX} is analogous to results already know in the literature, see \emph{e.g.} {\rm\cite[Theorem 4.2]{cvitanic2018dynamic}}. Result \eqref{generalF} is then an obvious generalisation to the case when one considers a general filtered space and thus loses the martingale representation with respect to the martingale measure $X$. Indeed, it is obvious that the contracts allowed in our setting form a strictly larger class than those we would consider if we had not introduced relaxed controls and the process $X$ had been driven by a standard Wiener process, see also {\rm\Cref{sec:non-relaxed}} for more details.

\medskip $(ii)$ Despite the resemblance of these results to the previously known theory, we would like to point out that both problems formulated here are non-standard as perceived from the point of view of what has been treated in this direction so far. First, while both problems involve relaxed controls, they are inherently, for a fixed $a \in \Upsilon,$ strong control problems in the sense that we work on a fixed probability space and, given parameters $(Y_0,m^Z,L),$ resp. $(y_0,m^Z),$ the process $Y^{Y_\smallfont{0},m^\smallfont{Z},L},$ resp. $Y^{y_\smallfont{0},m^\smallfont{Z}},$ is a strong solution to the associated {\rm SDE}. This is substantially different from what was considered in the literature, namely works {\rm\cite{el1987compactification,haussmann1990existence}}. For this reason, we consider the weak formulation of these problems in {\rm\Cref{relPrincip,sec:feedbackContr}} and then connect those two approaches. Moreover, because we consider the orthogonal martingale part $L$ as a control in the former, we face a problem which, to the best of our knowledge, has not been treated yet.

\medskip $(iii)$ Finally, let us point out that, due to {\rm\Cref{blumethal}}, we can restrict in the problem \eqref{FX} to optimisation over deterministic $y_0 \in \R.$ For a general filtration, this is obviously not the case and thus we have to consider general $Y_0 \in \mathfrak{Y}.$
\end{remark}

In the sequel, as mentioned, we shall consider the weak formulation of these two control problems. Let us emphasise that while the latter problem, that is if we work in the canonical filtration, provides a neater and more tractable reformulation, if we consider the corresponding weak control problem, the assumption on adaptedness of $m^Z$ to the filtration $\F^{X(\cdot)}$ seems to be a major complication since it is not preserved under weak convergence of measures. This issue, however, is \emph{not} due to the relaxed nature of the problem and is present even if we consider the weak formulation in the standard framework as in \cite{cvitanic2018dynamic,elie2019tale,hernandez2024principal} and many other works.

\medskip On the contrary, in the relaxed setting we will be able to show existence of optimal contracts provided that the filtered probability space is `large' enough. We refer to {\rm\Cref{sect:compar.form}} for more details. In the case of the canonical filtration, we will derive in {\rm\Cref{sec:feedbackContr}} a result which is somewhat unsatisfactory since the optimal contract obtained there needs not have the required measurability properties. Similar issues with a lack of the correct measurability of the weak limit are present also in the recent work of {\rm \citeauthor*{djete2023stackelberg} \cite{djete2023stackelberg}}, who employs arguments akin to the ones we use in the following sections. Note that while in the $n$-player game therein, the contract $\xi$ does not depend on the controls chosen by the agents, while in the mean-field game obtained as the limit for $n \longrightarrow \infty,$ the contract in general depends on the control adopted by the representative agent. Similarly, existence of an optimal contract is achieved only in a larger class of contracts, which are not necessarily feedback.

\begin{remark} Maximisation over functions $a \in \Upsilon$ will not be treated in this article and, thus, $a \in \Upsilon$ is mostly either assumed to be fixed or it is assumed that there exists a unique element in $\Upsilon$ in what follows. Indeed, this is a problem rather disjoint from our main concerns. It is trivial if the set $\Upsilon$ is finite, but in general, however, existence of a maximiser is not readily achievable due to a lack of a sufficient structure on $\Upsilon$ itself, and a case-by-case analysis would be required.
\end{remark}

\subsection{Weak principal's problem}  \label{relPrincip}
In this section, we consider a weakened version of the problem \eqref{generalF}. That is, as pointed out in the previous section, we reformulate an optimal control problem which is intrinsically a strong one, as a weak one. In words, we reformulate the problem so that it in a sense resembles the setting considered in \cite{el1987compactification,haussmann1990existence} and then employ similar arguments to show that, provided that certain compactness criteria are met, there exists a solution. Of course, there is then a need to show that these two formulations are related, and we address this question in \Cref{sect:compar.form}. Let us once again point out that the generally c\`{a}dl\`{a}g orthogonal martingale is a rather distinct feature of our problem.

\medskip We fix a function $a \in \Upsilon$ throughout. In particular, we thus take \Cref{ass:exist.minim} to be valid. To simplify some technicalities, we impose the following condition.

\begin{assumption}\label{discF}  The discount factor of the agent does not depend on the control. That is to say, $k: [0,T] \times \Cc([0,T],\R^d)\longrightarrow \R.$
\end{assumption}

\begin{remark} It is clear that under {\rm\Cref{discF}} any maximiser of the Hamiltonian is independent of the $y$-entry. That is, $a (t,x,y,z)=a(t,x,z),\; a \in \Upsilon.$ We note that {\rm\Cref{discF}} is not without loss of generality. Even though allowing for the general case would \emph{not} pose an unsurpassable problem for the analysis, it would require a more general setting, see {\rm\Cref{rem_disc_fact}}, and is not considered for simplicity.
 \end{remark}

Roughly speaking, we consider the following optimal control problem (this shall be formalised later on)
\[ \sup_{\{Y_\smallfont{0} : \E^\smallfont{\P}[Y_\smallfont{0}] \geq r_\smallfont{0}\}} \sup_{m^\smallfont{Z},L} \E^\P \big[U_p(X_{\cdot \wedge T},\xi)\big], \] subject to
\[ \E^\P \big[g^i(X_{\cdot \wedge T}, U_a^{-1}(Y_T))\big] \leq 0, \; i \in I, \] where the state process $(X,Y)$ follows the dynamics, for $t\in[0,T]$
\begin{align*}
X_t&=x_0+\int_0^t \int_{\R^d} \sigma(s,X_{\cdot \wedge s}) \lambda(s,X_{\cdot \wedge s},a(s,X_{\cdot \wedge s},z))m^Z(\mathrm{d}z,\mathrm{d}s)+\int_0^t\int_{\R^d}  \sigma(s,X_{\cdot \wedge s})M^Z(\mathrm{d}z,\mathrm{d}s), \\
Y_t&=Y_0 + \int_0^t \int_{\R^\smallfont{d}} \Big( f(s,X_{\cdot \wedge s},a(s,X_{\cdot \wedge s},z))+k(s,X_{\cdot \wedge s})Y_s \Big) m^Z(\mathrm{d}z,\mathrm{d}s) +\int_0^t \int_{\R^d} z^\top \sigma(s,X_{\cdot \wedge s}) M^Z(\mathrm{d}z,\mathrm{d}s)+\int_0^t \mathrm{d}L_s.
\end{align*}

For technical reasons, we work shall with the discounted value process
\begin{align*} 
\mathrm{e}^{-\int_\smallfont{0}^\smallfont{t} k(u,X_{\smallfont{\cdot}\smallfont{\wedge}\smallfont{u}}) \mathrm{d}u}Y_t &= Y_0 + \int_0^t \int_{\R^\smallfont{d}} \mathrm{e}^{-\int_\smallfont{0}^\smallfont{s} k(u,X_{\smallfont{\cdot} \smallfont{\wedge} \smallfont{u}}) \mathrm{d}u} f\big(s,X_{\cdot \wedge s},a(s,X_{\cdot \wedge s},z)\big) m^Z(\mathrm{d}z,\mathrm{d}s) 
\\ & \quad +\int_0^t \int_{\R^\smallfont{d}} \mathrm{e}^{-\int_\smallfont{0}^\smallfont{s} k(u,X_{\smallfont{\cdot} \smallfont{\wedge}\smallfont{u}}) \mathrm{d}u} z^\top \sigma(s,X_{\cdot \wedge s}) M^Z(\mathrm{d}z,\mathrm{d}s)+\int_0^t \mathrm{e}^{-\int_\smallfont{0}^\smallfont{s} k(u,X_{\smallfont{\cdot} \smallfont{\wedge}\smallfont{u}}) \mathrm{d}u}\mathrm{d}L_s,\; t \in [0,T]. 
\end{align*}

This allows us to write $\mathrm{e}^{-\int_\smallfont{0}^\smallfont{\cdot} k(u,X_{\smallfont{\cdot} \smallfont{\wedge}\smallfont{u}}) \mathrm{d}u}Y=\Yc+U,$ where, still for $t\in[0,T]$
\begin{align}
\Yc_t &\coloneqq \int_0^t \int_{\R^\smallfont{d}} \mathrm{e}^{-\int_\smallfont{0}^\smallfont{s} k(u,X_{\smallfont{\cdot} \smallfont{\wedge} \smallfont{u}}) \mathrm{d}u}  f\big(s,X_{\cdot \wedge s},a(s,X_{\cdot \wedge s},z)\big) m^Z(\mathrm{d}z,\mathrm{d}s)  +\int_0^t \int_{\R^d} e^{-\int_0^s k(u,X_{\cdot \wedge u}) \mathrm{d}u} z^\top \sigma(s,X_{\cdot \wedge s}) M^Z(\mathrm{d}z,\mathrm{d}s), \label{eqn:YC} \\
U_t&\coloneqq Y_0+\int_0^t \mathrm{e}^{-\int_\smallfont{0}^\smallfont{s} k(u,X_{\smallfont{\cdot} \smallfont{\wedge} \smallfont{u}}) \mathrm{d}u} \mathrm{d}L_s. \nonumber
\end{align}

\begin{remark} \label{rem_disc_fact} $(i)$ We emphasise that, since $k$ is independent of the process $Y,$ the right-hand sides in the definitions of $\Yc$ and $U$ are independent of $Y$ as well. That is, the processes $\Yc$ and $U$ are only coupled through the process $X.$ This allows us to formulate the control problem in terms of the continuous process $\Yc$ and the---potentially discontinuous---process $U.$ In full generality, one would need to work with the discontinuous process $Y$ and formulate the martingale problem in {\rm\Cref{rule}} on the Skorokhod space.

\medskip $(ii)$ The initial value of the process $Y_0$, that is the utility of the agent,    is the initial value of the orthogonal martingale part $U$ in our reformulation. We thus impose that the process $\Yc$ satisfies $\Yc_0=0,$ $\P${\rm--a.s.}, and the orthogonal martingale $U$ meets the constraint $\E^{\P} [U] \geq r_0.$ 
 \end{remark}

Let us now introduce the weak reformulation of the problem. Let us for $(t,x,z) \in [0,T] \times \Cc([0,T],\R^d) \times \R^d$ denote
\begin{align*}
\bar{\sigma}(t,x,z)&\coloneqq\begin{pmatrix}
\sigma(t,x) \\
\exp\big(-\int_0^t k(u,x_{\cdot \wedge u}\big) \mathrm{d}u) z^\top \sigma(t,x)
\end{pmatrix},\; a(t,x,z)\coloneqq\bar{\sigma}(t,x,z)\bar{\sigma}^\top(t,x,z),\\
b(t,x,z)&\coloneqq \begin{pmatrix}
\sigma(t,x)\lambda(t,x,a(t,x,z)),\exp\big(-\int_0^t k(u,x_{\cdot \wedge u}) \mathrm{d}u\big) f(t,x,a(t,x,z))
\end{pmatrix}.
\end{align*}
Let $\Cc_c^2(\R^{d+1},\R)$ denote the space of compactly supported twice continuously differentiable functions from $\R^{d+1}$ to $\R$ and let us for $\phi \in \Cc_c^2(\R^{d+1},\R)$ define for any $(t,x,y,z) \in [0,T]\times \Cc ([0,T],\R^d) \times \R \times \R^d$
\begin{align*} 
\mathscr{L}\phi (s,x_{\cdot \wedge s},y_s,z)\coloneqq b(s,x_{\cdot \wedge s},z)\cdot \mathrm{D}\phi(x_s,y_s) + \frac{1}{2} a(s,x_{\cdot \wedge s},z) : \mathrm{D}^2\phi(x_s,y_s),\end{align*}
where $D\phi$ and $D^2\phi$ denote the gradient and the Hessian of $\phi,$ respectively.
Let us further denote for $(x,y,u) \in \Cc ([0,T],\R^d) \times \R \times \R$

\begin{align*}
F(x,y,u)\coloneqq \begin{cases} U_p \Big(x,U_a^{-1}\Big(\mathrm{e}^{ \int_\smallfont{0}^\smallfont{T}  k(t,x_{\smallfont{\cdot} \smallfont{\wedge} \smallfont{t}}) \mathrm{d}t}(y+u)\Big)\Big), \; \text{if}\; \mathrm{e}^{ \int_\smallfont{0}^\smallfont{T}  k(t,x_{\smallfont{\cdot} \smallfont{\wedge} \smallfont{t}}) \mathrm{d}t}(y+u) \in U_a(\R), \\
-\infty, \; \text{\rm otherwise}, \end{cases} 
\end{align*}
and for $i \in I$
\begin{align*}
h^i(x,y,u)\coloneqq \begin{cases} g^i\Big(x,U_a^{-1}\Big(\mathrm{e}^{\int_\smallfont{0}^\smallfont{T}  k(t,x_{\smallfont{\cdot} \smallfont{\wedge} \smallfont{t}}) \mathrm{d}t}(y+u)\Big)\Big), \; \text{if}\; \mathrm{e}^{ \int_\smallfont{0}^\smallfont{T}  k(t,x_{\smallfont{\cdot} \smallfont{\wedge} \smallfont{t}}) \mathrm{d}t}(y+u) \in U_a(\R), \\
\infty, \;\text {\rm otherwise}. \end{cases}
\end{align*}

\begin{remark} Indeed, by requiring $\E^\P \big[h^i(X_{\cdot \wedge T},\Yc_t,U)\big] \leq 0,$ $i \in I$, we also impose the condition $\mathrm{e}^{ \int_\smallfont{0}^\smallfont{T}  k(t,X_{\smallfont{\cdot} \smallfont{\wedge} \smallfont{t}}) \mathrm{d}t}(\Yc_T+U) \in U_a(\R),$ $\P${\rm--a.s.} If there are no such contraints, we set $I\coloneqq\{1\}$ and 
\[h^1(x,y,u)\coloneqq 1-\mathbf{1}_{\{ \exp( \int_\smallfont{0}^\smallfont{T}  k(t,x_{\smallfont{\cdot} \smallfont{\wedge} \smallfont{t}}) \mathrm{d}t)(y+u) \in U_a(\R) \}},
 \] to ensure that this condition is satisfied.
\end{remark}

We assume the following conditions.

\begin{assumption} \label{assPrincipal} We have that
\begin{enumerate}
\item[$(i)$] $\sigma$ is bounded, the functions $a$ and $b$ are continuous in $(x,z)$ for every $t$ and there is a constant $C>0$ such that 
\begin{align} \label{boundFA}
\lVert a(t,x,z) \rVert  &\leq C \big(1 + \lVert x \rVert + \lVert z^\top \sigma(t,x) \rVert^2 \big)\; {\rm and}\;\lVert b(t,x,z) \rVert \leq C \big(1 + \lVert x \rVert + \lVert z^\top \sigma(t,x) \rVert \big);
\end{align}
\item[$(ii)$] \label{assF} the function $F$ is upper-semicontinuous, concave in $(y,u)$ for every $x$, and there is $C>0$ such that 
\[
F^+(x,y,u) \leq C\big(1+\lVert x \rVert +\lvert y \rvert +\lvert u \rvert \big),\; (x,y,u) \in \Cc([0,T],\R^d) \times \R \times \R;
\]
\item[$(iii)$] \label{assh} for any $i\in I$, $h^i$ is lower-semicontinuous and there exists $C_i>0$ such that
\begin{equation*}
h^i(x,y,u) \geq -C_i \big(1+ \lVert x \rVert + \lvert y \rvert +\lvert u \rvert  \big).
\end{equation*}
Moreover, $ h^i(x,y,\,\cdot\,)$ is convex for any $(x,y) \in \Cc([0,T],\R^d) \times \R$.
\end{enumerate}
\end{assumption}

\begin{remark} \label{rem:convex} $(i)$ The convexity of functions $h^i(x,y,\cdot),$ $i  \in I,$ can be relaxed to the following condition: the set $M^i(x,y)\coloneqq\{ u : h^i(x,y,u)\leq 0 \}$ is convex and the function $h^i(x,y,\cdot)$ is convex on $M^i(x,y)$ for every $(i,x,y) \in I \times \Cc([0,T],\R^d) \times \R.$

\medskip $(ii)$ The convexity and concavity of the involved functioned is assumed so that the principal cannot benefit from even further enlarging the filtration, see {\rm \Cref{largerfiltr}}. In {\rm \Cref{sec:feedbackContr}} we do not consider larger filtrations and hence this assumption is dropped.

\medskip $(iii)$ The maps $F$ and $h^i,$ $i \in I,$ are, in fact, functions of $(x,y+u)$ so some expressions can be slightly simplified. 
\end{remark}

We now introduce the canonical space. Let us denote by $\Mc([0,T],\R^d)$ the set of measurable functions $m^Z: [0,T] \longrightarrow \Pc(\R^d)$ endowed with the coarsest topology making the maps
\[ 
m^Z \longmapsto \int_0^T \int_{\R^d} \varphi(t,z) m^Z_t(\mathrm{d}z)\mathrm{d}t,  \] 
continuous for every $\varphi: [0,T] \times \R^d \longrightarrow \R$ Borel-measurable, bounded and continuous in $z$ for every $t \in [0,T].$ This topology renders $\Mc([0,T],\R^d)$ Polish, see \cite[page 863]{haussmann1990existence} for details. Let us denote by $m^Z$ the canonical process and the canonical $\sigma$-algebra by $\F^{m^\smallfont{Z}}.$ \emph{i.e.}
\[ 
\Fc_t^{m^\smallfont{Z}}\coloneqq\sigma \bigg(\int_0^s \int_{\R^\smallfont{d}} \varphi(u,z) m^Z_u(\mathrm{d}z)\mathrm{d}u : s\in[0, t],\; \varphi\; \text{is Borel-measurable and bounded} \bigg),\; t \in [0,T]. 
\]

Next, we consider the space $\Cc([0,T],\R^{d+1})$ with canonical process $(X,\Yc)\in \R^d \times \R$ and canonical $\sigma$-algebra $\F^{X,\Yc}.$ Let us set $\Omega^c\coloneqq \Cc([0,T],\R^{d + 1}) \times \Mc([0,T],\R^d)$ and $\F^c\coloneqq\F^{X,\Yc} \otimes \F^{m^\smallfont{Z}}.$ Finally, we denote the canonical random variable on $\Omega^c \times \R$ by $(X,\Yc,m^Z, U)$. Note that we do not require any continuity of the filtration. The set of controls is introduced in the following definition.

\begin{definition} \label{rule}
Let $q>1$ and $q^\prime>1.$ We say that a Borel probability measure $\P$ on $(\Omega^c \times \R, \Fc_T^c \otimes \Bc(\R))$ belongs to the set of control rules for the principal with constants $(q,q^{\prime})$, denoted by $\Rc_{q,q^\smallfont{\prime}},$ if the following hold
\begin{enumerate}[label = $(\roman*)$]
\item \label{rule1} the constants $q>1$ and $q^\prime>1$ are such that 
\[
\E^{\P}\bigg[\bigg(\int_0^T\int_{\R^\smallfont{d}} \lVert z^\top \sigma(t,X_{\cdot \wedge t}) \rVert^2 m^Z_t(\mathrm{d}z)\mathrm{d}t\bigg)^\frac{q}{2}\bigg] < \infty \; {\rm and}\; \E^{\P}\big[\lvert U\rvert^{q^\smallfont{\prime}}\big] < \infty;\]
\item \label{rule2} the process 
\[ 
M^\phi_t\coloneqq\phi(X_t,\Yc_t)-\int_0^t \int_{\R^\smallfont{d}} \mathscr{L}\phi (s,X_{\cdot \wedge s},\Yc_s,z)  m^Z_s(\mathrm{d}z)\mathrm{d}s,\; t \in [0,T], 
\] 
is an $(\F^c,\P)$-martingale for every $\phi \in \Cc_c^2(\R^{d+1},\R);$
\item \label{rule3} $(X_0,\Yc_0,m^Z_0)=(x_0,0,\delta_0),\;\P$--{\rm a.s.}$;$
\item for any $i\in I$ \label{rule4} 
\begin{equation*}
\E^\P \big[h^i(X_{\cdot \wedge T},\Yc_t,U)\big] \leq 0;
\end{equation*}
\item \label{rule5} $\forall  \phi \in \Cc_c^2(\R^{d+1},\R)$, we have
\[
\big[ M^\phi_\cdot, \E^{\P}[ U \vert \Fc_\cdot^c ] \big]^{\F^\smallfont{c},\P} \equiv 0,\; \text{\rm and}\;  \E^\P [U] \geq r_0. 
\]
\end{enumerate}
Moreover, we set $\Rc=\bigcup_{(q,q^\smallfont{\prime})\in(1,+\infty)^\smallfont{2}}\Rc_{q,q^\smallfont{\prime}}.$
\end{definition}

We are hence interested in the following maximisation 
\begin{equation}\label{eq:reformulation}
\sup_{\P \in \Rc}\E^{\P}\big[F(X_{\cdot \wedge T},Y_T,U)\big].
\end{equation}
\begin{remark} 
Let us once again emphasise that at this point, it is not clear that problem \eqref{eq:reformulation} is in any sense equivalent to the principal's problem introduced in the previous section. We shall compare these two formulations in {\rm\Cref{sect:compar.form}} and provide a sufficient condition under which this problem gives an optimal contract.
\end{remark}

\begin{remark} In our setting, the orthogonal martingale $(U_t)_{t\in[0,T]}$ is only represented by the random variable $U$ and then recovered through the relation $U_t\coloneqq\E^{\P}[U \vert \Fc_t^c].$ As such, it would then be natural to write $\E^{\P}[U \vert \Fc_T^c]$ instead of $U$ in the constraints as well as in the objective function. However, in the light of {\rm \Cref{largerfiltr}} we see that it is not optimal to choose $\P \in \Rc$ such that $\P[U \neq \E^{\P}[U \vert \Fc_T^c]]>0.$
\end{remark}

We now state the main theorem of this section where we show that there exists a solution to the problem introduced above. In order to do so, we need to compactify the set of controls and then obtain existence of an optimiser on every such compact set. It is, indeed, natural to impose such compactness as the problem might, in general, be unbounded and thus existence of a solution is not guaranteed. We provide more discussions below as well as in \Cref{subsec:compactness}. In specific models this would then have to be verified on a case-by-case basis. The bulk of the proof is mostly technical and is postponed to \Cref{subsec:proof_of_THM}.

\begin{remark} It is standard in the literature, see \emph{e.g.} {\rm\citeauthor*{mirrlees1974notes} \cite{mirrlees1974notes}}, {\rm\citeauthor*{alvarez2023optimal} \cite{alvarez2023optimal}},
{\rm\citeauthor*{holmstrom1977incentives} \cite{holmstrom1977incentives,holmstrom1979moral}},
{\rm\citeauthor*{page1987existence} \cite{page1987existence,page1991optimal,page1992mechanism}},
{\rm\citeauthor*{backhoff2022robust} \cite{backhoff2022robust}},
{\rm\citeauthor*{balder1996existence} \cite{balder1996existence}},
{\rm\citeauthor*{ke2017existence} \cite{ke2017existence}},
{\rm\citeauthor*{jewitt2008moral} \cite{jewitt2008moral}} and
{\rm\citeauthor*{balder1996existence} \cite{balder1996existence}}, to impose boundedness or one-sided boundedness on the set of contracts to obtain compactness. In our setting, we only require boundedness of the moments of the control processes, which is generally substantially weaker.
\end{remark}

Let us for $(q,q^\prime) \in (1,\infty)^2$ define 
\begin{align*}
K_{q,q^\smallfont{\prime},R}\coloneqq\bigg\lbrace \P \in \Rc : \E^{\P}\bigg[ \int_0^T \int_{\R^\smallfont{d}} \lVert z \rVert^{2q} m^Z_s(\mathrm{d}z)\mathrm{d}s + \lvert U \rvert^{q^{\prime}} \bigg]\leq R \bigg\rbrace.
\end{align*} 
We have the following result.

\begin{theorem} \label{minim} Let {\rm Assumptions \ref{discF}} and {\rm \ref{assPrincipal}} hold. Then the function $\P \longmapsto \E^{\P}[F(X_{\cdot \wedge T},\Yc,U)]$  admits a maximiser on $K_{q,q^\smallfont{\prime},R}$ for any $R>0$ and $(q,q^\prime) \in (1,\infty)^2$ satisfying $\frac{1}{q}+\frac{1}{q^\prime} < 1,$ such that $K_{q,q^\smallfont{\prime},R}$ is non-empty.
\end{theorem}
\begin{proof} Because every upper-semicontinuous function attains its maximum on a compact set, the result follows immediately from \Cref{lem:compactness,lem:upper-sc}.
\end{proof}

Let us comment on potential drawbacks of this approach. In this generality, obtaining bounds on the moments of the process $m^Z$ of the form \[\E^{\P}\bigg[ \int_0^T \int_{\R^d} \lVert z \rVert^{2q} m^Z_s(\mathrm{d}z)\mathrm{d}s \bigg] \leq C, \] for some $C>0$ might not be readily achievable. For instance, using the theory of BSDEs, one typically obtains bounds on the norm
\[ \bigg(\E^{\P} \bigg[ \bigg(  \int_0^T \int_V \big\| z^\top \sigma(x,X_{\cdot \wedge s}) \big\|^2 m^Z_s(\mathrm{d}z)\mathrm{d}s \bigg)^\frac{p}{2} \bigg]\bigg)^\frac{1}{p}, \] in terms of $\E^{\P}[\lvert Y_T \rvert^p]$ for some $p>0,$ see \eqref{bdBSDE}. This, however, is in general insufficient as the norm is relatively weak for our purposes. For more refined results concerning bounds on the process $Z$, one has to employ more advanced techniques, such as \emph{e.g.} the Malliavin calculus, see for instance \citeauthor*{briand2013simple} 
\cite{briand2013simple} and \citeauthor*{harter2019stability} \cite{harter2019stability}. In general, verifying such a condition has to be done case-by-case. On the other hand, deriving bounds on the moments of $U$ is possible with the standard BSDE theory. We refer to \Cref{subsec:compactness} for more discussions.

\begin{remark} \label{rem:uniform_integr}
Analysing the proof of {\rm\Cref{minim}}, namely {\rm \Cref{tight,lem:compactness}}, we see that we can actually obtain a slightly stronger result. More specifically, let {\rm Assumptions \ref{discF}} and {\rm \ref{assPrincipal}} hold and let $(R,q,q^\prime) \in (0,\infty) \times (1,\infty)^2$ be such that $\frac{1}{q}+\frac{1}{q^\prime} < 1.$ Let further
\[K \subset \bigg\lbrace \P \in \Rc : \E^{\P}\bigg[ \bigg(\int_0^T \int_{\R^\smallfont{d}} \lVert z \rVert^{2} m^Z_s(\mathrm{d}z)\mathrm{d}s \bigg)^q + \lvert U \rvert^{q^{\prime}} \bigg]\leq R \bigg\rbrace, \] be a non-empty set such that
\[ \lim_{M \rightarrow \infty} \sup_{\P \in K} \E^{\P}\bigg[ \bigg( \int_0^T \int_{\R^\smallfont{d}} \lVert z \rVert^{2} \mathbf{1}_{\{\|z\| \geq M\}}  m^Z_s(\mathrm{d}z)\mathrm{d}s \bigg)^q \bigg]=0. \] 
Then, since the map $m^Z \longmapsto \int_0^T \int_{\R^\smallfont{d}} \lVert z \rVert^{2} \mathbf{1}_{\{\|z\| \geq M\}}  m^Z_s(\mathrm{d}z)\mathrm{d}s$ is lower-bounded and lower-semicontinuous, we have
\[ \lim_{M \rightarrow \infty} \sup_{\P \in \overline{K}} \E^{\P}\bigg[ \bigg( \int_0^T \int_{\R^\smallfont{d}} \lVert z \rVert^{2} \mathbf{1}_{\{\|z\| \geq M\}}  m^Z_s(\mathrm{d}z)\mathrm{d}s \bigg)^q \bigg]=\lim_{M \rightarrow \infty} \sup_{\P \in K} \E^{\P}\bigg[ \bigg( \int_0^T \int_{\R^\smallfont{d}} \lVert z \rVert^{2} \mathbf{1}_{\{\|z\| \geq M\}}  m^Z_s(\mathrm{d}z)\mathrm{d}s \bigg)^q \bigg]=0.\]
Moreover, adapting the proofs of {\rm\cite[Theorem A.8]{haussmann1990existence}} and {\rm \Cref{tight,lem:compactness}}, one can show that $K$ is relatively compact and
\[\overline{K} \subset \bigg\lbrace \P \in \Rc : \E^{\P}\bigg[ \bigg(\int_0^T \int_{\R^\smallfont{d}} \lVert z \rVert^{2} m^Z_s(\mathrm{d}z)\mathrm{d}s \bigg)^q + \lvert U \rvert^{q^{\prime}} \bigg]\leq R \bigg\rbrace. \]
Thus, the function $\P \longmapsto \E^{\P}[F(X_{\cdot \wedge T},\Yc,U)]$  admits a maximiser on $\overline{K},$ which is a subset of $\Rc.$
\end{remark}

\subsubsection{Comparison of the two formulations} \label{sect:compar.form}
As mentioned earlier, while the principal's problem introduced in the first part of \Cref{sec:principal} is a strong control problem, that is, we work on a fixed probability space, we later moderated this formulation to what is usually referred to as a weak control problem. As such, the analysis would not be complete without a comparison of these two approaches.

\medskip
Let Assumptions \ref{asssigmamu}, \ref{assUkf}, \ref{ass:exist.minim}, \ref{discF} and \ref{assPrincipal} hold and let $\P \in \Rc.$ Then, according to \Cref{ext1} there exists an extension, say $(\tilde{\Omega},\tilde{\Fc},\tilde{\F},\tilde{\P}),$ supporting a $k$-dimensional $(\tilde{\F},\tilde{\P})$-martingale measure $\widetilde{M}^0$ on $\Bc(V)$ with intensity $m^0$ such that
\begin{align*} 
X_t&=x_0+\int_0^t \int_V \sigma(s,X_{\cdot \wedge s}) \lambda\big(s,X_{\cdot \wedge s},a(s,X_{\cdot \wedge s},Z_s(v))\big) m^0(\mathrm{d}v,\mathrm{d}s)+\int_0^t\int_V \sigma(s,X_{\cdot \wedge s}) \widetilde{M}^0(\mathrm{d}v,\mathrm{d}s), \\
\Yc_t &= \int_0^t \int_{V} \mathrm{e}^{-\int_\smallfont{0}^\smallfont{s} k(u,X_{\smallfont{\cdot} \smallfont{\wedge} \smallfont{u}}) \mathrm{d}u} f\big(s,X_{\cdot \wedge s},a(s,X_{\cdot \wedge s},Z_s(v))\big)m^0(\mathrm{d}v,\mathrm{d}s) \\
&\quad+\int_0^t \int_{V} \mathrm{e}^{-\int_\smallfont{0}^\smallfont{s} k(u,X_{\smallfont{\cdot} \smallfont{\wedge} \smallfont{u}}) \mathrm{d}u} Z_s(v)^\top \sigma(s,X_{\cdot \wedge s}) \widetilde{M}^0(\mathrm{d}v,\mathrm{d}s),\; t \in [0,T],\; \tilde{\P}\text{\rm--a.s.},
\end{align*} where $Z$ is determined by $m^Z=Z_\# m^0$ and, moreover, $[ \widetilde{M}^0(\cdot),U ]^{\tilde{\F},\tilde{\P}} \equiv 0,$ where $U$ is the $\tilde{\F}$-martingale given by $U_t=\E^{\P}[U \vert \Fc^c_t],$ $t \in [0,T].$ Let us w.l.o.g. assume that this extension satisfies the usual conditions and that $U$ is right-continuous. If this is not the case, we consider the augmented filtration instead and then the c\`{a}dl\`{a}g version $U$, see \cite[Theorem 3.2.6]{weizsaecker1990stochastic}. Let us define $\tilde{\P}^0$ by 
\begin{align*}
\frac{\mathrm{d}\tilde{\P}^0}{\mathrm{d} \tilde{\P}}\coloneqq \exp &\bigg( -\int_0^T\int_V \lambda\big(s,X_{\cdot \wedge s},a(s,X_{\cdot \wedge s},Z_s(v))\big)\cdot M^0(\mathrm{d}v,\mathrm{d}s) 
\\&\quad - \frac{1}{2} \int_0^T\int_V \big\| \lambda\big(s,X_{\cdot \wedge s},a(s,X_{\cdot \wedge s},Z_s(v))\big)\big\|^2m^0(\mathrm{d}v,\mathrm{d}s)\bigg),
\end{align*} and the process $Y$ by \[ Y_t \coloneqq \exp \bigg(\int_0^t k(s,X_{\cdot \wedge s})\mathrm{d}s\bigg) ( \Yc_t + U_t),\; t \in [0,T]. \] Then under $\tilde{\P}^0$ we have for $M^0_t(\cdot)\coloneqq \widetilde{M}^0_t(\cdot)+\int_0^t \int_\cdot \lambda\big(s,X_{\cdot \wedge s},a(s,X_{\cdot \wedge s},Z_s(v))\big) m^0(\mathrm{d}v,\mathrm{d}s),$ $t \in [0,T],$
\begin{align*}
X_t&=x_0+\int_0^t\int_V \sigma(s,X_{\cdot \wedge s}) M^0(\mathrm{d}v,\mathrm{d}s), \\
Y_t &= U_0 \\
&\;\;+ \int_0^t \int_{V} \Big( f\big(s,X_{\cdot \wedge s},a(s,X_{\cdot \wedge s},Z_s(v))+k(s,X_{\cdot \wedge s}) Y_s-\sigma(s,X_{\cdot \wedge s}) \lambda\big(s,X_{\cdot \wedge s},a(s,X_{\cdot \wedge s},Z_s(v))\big)\cdot Z_s(v)\Big)m^0(\mathrm{d}v,\mathrm{d}s) \\
&\;\;+\int_0^t \int_{V} Z_s(v)^\top \sigma(s,X_{\cdot \wedge s}) M^0(\mathrm{d}v,\mathrm{d}s) + \int_0^t \mathrm{e}^{\int_\smallfont{0}^\smallfont{s} k(u,X_{\smallfont{\cdot} \smallfont{\wedge} \smallfont{u}}) \mathrm{d}u} \mathrm{d}U_s,\; t \in [0,T].
\end{align*}
Thus, if we denote $Y_0=U_0$ and $L \coloneqq\int_0^\cdot \mathrm{e}^{\int_\smallfont{0}^\smallfont{s} k(u,X_{\smallfont{\cdot} \smallfont{\wedge} \smallfont{u}}) \mathrm{d}u} \mathrm{d}U_s,$ we obtain
\begin{align*}
X_t&=x_0+\int_0^t\int_V \sigma(s,X_{\cdot \wedge s}) M^0(\mathrm{d}v,\mathrm{d}s), \\
Y_t &= Y_0 - \int_0^t \int_{\R^\smallfont{d}} H(s,X_{\cdot \wedge s},Y_s,Z_s(v))m^0(\mathrm{d}v,\mathrm{d}s)+\int_0^t \int_{V} Z_s(v)^\top \sigma(s,X_{\cdot \wedge s}) M^0(\mathrm{d}v,\mathrm{d}s) + \int_0^t  \mathrm{d}L_s,\; t \in [0,T].
\end{align*} Moreover, the property $[ \widetilde{M}^0(\cdot),U]^{\tilde{\F},\tilde{\P}} \equiv 0$ clearly implies $[ \widetilde{M}^0(\cdot),L]^{\tilde{\F},\tilde{\P}} \equiv 0$ and thus we also have $[ M^0(\cdot),L]^{\tilde{\F},\tilde{\P}^\smallfont{0}} \equiv 0.$ Further, $L$ is an $(\tilde{\F},\tilde{\P}^0)$-martingale by analogous arguments as in the proof of \Cref{ICchar}. We therefore recover the dynamics from \Cref{agent}.

\begin{corollary} \label{col:comparison} Let {\rm Assumptions \ref{asssigmamu}, \ref{assUkf}, \ref{discF}} and {\rm\ref{assPrincipal}} hold and assume that there is a unique $a \in \Upsilon$ and that there exists a solution $\P^\star \in \Rc$ to the problem \[\sup_{\P \in \Rc}\E^{\P}\big[F(X_{\cdot \wedge T},Y_T,U)\big], \] and denote by $(\tilde{\Omega},\tilde{\Fc},\tilde{\F},\tilde{\P}^0)$ the probability space described in the discussion above corresponding to $\P^\star$ and by $(\tilde{X},\tilde{Y})$ the processes introduced there. Assume that $(\Omega,\Fc,\F,\P^0),$ that is the probability space of the agent's problem, contains $(\tilde{\Omega},\tilde{\Fc},\tilde{\F},\tilde{\P}^0)$  in the sense that it is of the form
\[  
(\Omega,\Fc,\F,\P^0)= \big(\tilde{\Omega}\times \hat{\Omega},\tilde{\Fc} \otimes \hat{\Fc},\tilde{\F}\otimes \hat{\F},\tilde{\P}^0(\mathrm{d}\tilde{\omega})\otimes\hat{\P}(\tilde{\omega}, \mathrm{d}\hat{\omega})\big),
\] for some filtered space $(\hat{\Omega},\hat{\Fc},\hat{\F})$ and some measurable kernel $\hat{\P} : \tilde{\Omega} \longrightarrow \Pc(\hat{\Omega}).$ Assume that the process that the agent takes control of coincides with $\tilde{X},$ that is, $X=\tilde{X}.$ Then the corresponding principal's problem admits a solution, which is given by $\xi^\star=U_a^{-1}(\tilde{Y}_T).$
\end{corollary}
\begin{proof}
It is straightforward that $\xi^\star \in \Ic\Cc$ by \Cref{ICchar} and \Cref{bdMOMstatement}. We have to verify optimality. Using \Cref{largerfiltr} we conclude that $\sup_{\P \in \Rc}\E^{\P}\big[F(X_{\cdot \wedge T},Y_T,U)\big] \geq V_0^p,$ and thus $\xi^\star$ is necessarily optimal.
\end{proof}

\begin{remark} $(i)$ Roughly speaking, the statement asserts that, provided that the principal has access to enough randomisation, there exists an optimal contract. Of course, this result is of rather theoretical nature and in practical situations would have to be checked on a case-by-case basis. 

\medskip $(ii)$ Note that the assumption in {\rm \Cref{col:comparison}} does not only concern the filtered space, but also the probability measure considered on it. This is important since in the original problem we are allowed to employ only equivalent changes of measure while in the weak formulation also mutually singular measures are admissible.

\medskip $(iii)$ Let us also point out that an enlargement of the probability space can indeed in general benefit the principal. A similar observation was already noted in the work of {\rm \citeauthor*{kadan2017existence} \cite[page 785]{kadan2017existence}}. However, in the light of {\rm\Cref{largerfiltr}} we see that there is a limit to it, in the sense that there exists a measurable space which is already optimal for the principal and enlarging it would not increase her utility.
\end{remark}

\subsection{Feedback contracts} \label{sec:feedbackContr}
At last, we treat the case where it is required that the contract is $\Fc_T^{X(\cdot)}$-measurable. This situation might be of interest when one does not wish to allow for space enlargements, and imposes that the contract is measurable with respect to the canonical $\sigma$-algebra $\Fc_T^{X(\cdot)}.$ Indeed, this is natural since the obvious generalisation of the contracts considered \emph{e.g.} in \cite{cvitanic2018dynamic,hernandez2024principal} is to impose that $\xi$ is a function of the state process, which in our case is the martingale measure $X(\cdot).$ The main drawback of this is that in such a case we cannot guarantee closedness of the set of admissible contracts, since this property is not preserved if one takes weak limits. On the other hand, there is no need for the orthogonal martingale part due to the martingale representation theorem.

\medskip Similarly as in \Cref{relPrincip}, we thus consider the weak formulation of the problem control problem \eqref{FX}. Note that we no longer require \Cref{discF}, but we need to slightly modify the setting. Let us again fix $a \in \Upsilon,$ which we thus assume to exists, and let us set for $\phi \in \Cc_c^2(\R^{d+1},\R)$ and $(t,x,y,z) \in [0,T]\times \Cc ([0,T],\R^d) \times \R \times \R^d$
\begin{align*}
\hat{\mathscr{L}}\phi (s,x_{\cdot \wedge s},y_s,z)=\hat{b}(t,x_{\cdot \wedge s},y_s,z)\cdot \mathrm{D}\phi(x_s,y_s) + \frac{1}{2} \hat{a}(s,x_{\cdot \wedge s},z) : \mathrm{D}^2\phi(x_s,y_s),
\end{align*}
where
\begin{align*}
\hat{\bar{\sigma}}(t,x,z)&\coloneqq\begin{pmatrix}
\sigma(t,x) \\
 z^\top \sigma(t,x)
\end{pmatrix},\;
\hat{a}(t,x,z)\coloneqq\hat{\bar{\sigma}}\hat{\bar{\sigma}}^\star(t,x,z),\;
\hat{b}(t,x,y,z)\coloneqq \begin{pmatrix}
\sigma(t,x)\lambda(t,x,a(t,x,y,z)), f(t,x,a(t,x,y,z))
\end{pmatrix}.
\end{align*}

Let us further denote 
\[
\hat{F}(x,y)\coloneqq \begin{cases} U_p(x,U_a^{-1}(y)), \; \text{if}\;  y \in U_a(\R),\\
-\infty, \; \text{\rm otherwise,} \end{cases}\; (x,y) \in \Cc ([0,T],\R^d) \times \R,
\]
and
\[
\hat{h}^{i}(x,y)\coloneqq \begin{cases} g^i(x,U_a^{-1}(y)), \; \text{if}\; y \in U_a(\R),\\
\infty, \; \text{\rm otherwise,} \end{cases}\; (x,y) \in \Cc ([0,T],\R^d) \times \R,\; i \in I. 
\]

We assume the following conditions, which are analogous to the ones introduced in \Cref{assPrincipal}.

\begin{assumption} \label{assPrincip2} We have that
\begin{enumerate}
\item[$(i)$] $\sigma$ is bounded and the functions $a$ and $b$ are continuous in $(x,z)$ for every $t$, and there is a constant $C>0$ such that 
\begin{align*} 
\lVert \hat{a}(t,x,z) \rVert &\leq C \big(1 + \lVert x \rVert + \lvert y \rvert + \lVert z^\top \sigma(t,x) \rVert^2 \big)\;{\rm and}\;\lVert \hat{b}(t,x,y,z) \rVert \leq C \big(1 + \lVert x \rVert + \lvert y \rvert + \lVert z^\top \sigma(t,x) \rVert \big);
\end{align*}
\item[$(ii)$] the function $\hat{F}$ is upper-semicontinuous and there is $C>0$ such that 
\[
\hat{F}^{+}(x,y) \leq C\big(1+\lVert x \rVert +\lvert y \rvert \big),\; (x,y,u) \in \Cc([0,T],\R^d) \times \R;
\]
\item[$(iii)$] the functions $\hat{h}^i$, $i \in I,$ are lower-semicontinuous and there exists $C_i>0$ such that
\begin{equation*}
\hat{h}^i(x,y) \geq -C_i \big(1+ \lVert x \rVert + \lvert y \rvert \big).
\end{equation*}
\end{enumerate}
\end{assumption}

In this section, we work on the canonical space $\Omega^c=\Cc([0,T],\R^{d + 1}) \times \Mc([0,T],\R^d).$ The canonical process shall now be denoted by $(X,Y,m^Z)$. In the current setting, the measure $X(\cdot)$ does not live on the canonical space that we work on. Hence, if we want to speak about $\Fc_T^{X(\cdot)}$-measurable contracts, we need to augment the probability space to incorporate it.

\begin{definition} \label{rule'}
Let $y_0 \in \R$ and $q > 1.$ We say that a probability measure $\P$ on $(\Omega^c, \Fc^c_T)$ belongs to the set of control rules for the principal, denoted by $\hat{\Rc}_q(y_0)$ if
\begin{enumerate}[label = $(\roman*)$]
\item we have 
\[
\E^{\P}\bigg[\bigg(\int_0^T\int_{\R^d} \lVert z^\top \sigma(t,X_{\cdot \wedge t}) \rVert^2 m^Z_t(\mathrm{d}z)\mathrm{d}t \bigg)^\frac{q}{2}\bigg] < \infty;
\]
\item the process 
\[
M^\phi_t\coloneqq \phi(X_t,Y_t)-\int_0^t \int_{\R^\smallfont{d}} \hat{\mathscr{L}}\phi (s,X_{\cdot \wedge s},Y_s,z) m^Z_s(\mathrm{d}z)\mathrm{d}s,\; t \in [0,T], 
\]
is an $(\F^{c},\P)$-martingale for every $\phi \in \Cc_c^2(\R^{d+1},\R);$
\item  $(X_0,Y_0,m^Z_0)=(x_0,y_0,\delta_0),\;\P${\rm--a.s.;}
\item  we have for any $i\in I$
\begin{equation*}
\E^\P \big[\hat{h}^i(X_{\cdot \wedge T},Y_T)\big] \leq 0;
\end{equation*}
\item\label{rule4'} there exists an extension of $(\Omega^c,\Fc^c_T,\F^c,\P),$ denoted $(\Omega,\Fc_T,\F,\Q),$ supporting a $k$-dimensional $(\F,\Q)$-martingale measure $M$ on $V$ with intensity $m^0$ satisfying
\begin{align}
\begin{split} \label{eqn:ext_dynamics}
X_t&=x_0+\int_0^t \int_V \sigma(s,X_{\cdot \wedge s}) \lambda\big(s,X_{\cdot \wedge s},a(s,X_{\cdot \wedge s},Y_s,Z_s(v))\big)m^0(\mathrm{d}v,\mathrm{d}s)+\int_0^t\int_V \sigma(s,X_{\cdot \wedge s}) M(\mathrm{d}v,\mathrm{d}s), \\
Y_t&=y_0 + \int_0^t \int_V \Big(f\big(s,X_{\cdot \wedge s},a(s,X_{\cdot \wedge s},Y_s,Z_s(v))\big)+k\big(s,X_{\cdot \wedge s},a(s,X_{\cdot \wedge s},Y_s,Z_s(v))\big)Y_s \Big)m^0(\mathrm{d}v,\mathrm{d}s) \\
&\quad+\int_0^t \int_V Z_s(v)^\top \sigma(s,X_{\cdot \wedge s}) M(\mathrm{d}v,\mathrm{d}s),\; t \in [0,T],\; \Q\text{\rm--a.s.}
\end{split}
\end{align} 
such that $m^Z$ is $\F^{X(\cdot)}$-adapted, where $Z$ is given by $m^Z=Z_\#m^0$ and $X(\cdot)$ is defined by
\[
X_t(\cdot)\coloneqq x_0+\int_0^t \int_\cdot \sigma(s,X_{\cdot \wedge s}) \lambda\big(s,X_s,a(s,X_{\cdot \wedge s},Y_s,Z_s(v))\big)m^0(\mathrm{d}v,\mathrm{d}s)+\int_0^t\int_\cdot  \sigma(s,X_{\cdot \wedge s}) M(\mathrm{d}v,\mathrm{d}s). 
\]
\end{enumerate}
We further set $\hat{\Rc}(y_0)\coloneqq\bigcup_{q>1}\hat{\Rc}_q(y_0).$
\end{definition}
\begin{remark} 
$(i)$Let us note that an extension supporting $M$ always exists by using an analogue of {\rm\Cref{ext1}}. Unfortunately, it is in general not true that $m^Z$ is $\F^{X(\cdot)}$-adapted.

\medskip $(ii)$ Indeed, if $m^Z$ is $\F^X$-adapted, where $\F^X$ is the filtration generated by the process $X,$ then $m^Z$ is $\F^{X(\cdot)}$-adapted on any extension such that {\rm\Cref{eqn:ext_dynamics}} holds. Moreover, such an extension always exists due to {\rm\Cref{ext1}}.
\end{remark}

\begin{remark} Note that can drop the convexity and concavity assumptions considered in the previous section. We refer to {\rm \Cref{rem:convex}} for discussions on why this is the case.
\end{remark}

Now in this setting, every control rule gives a feedback contract.

\begin{lemma} \label{FXadaptedness} Let {\rm Assumptions \ref{asssigmamu} and \ref{assPrincip2}} hold. For every $\P \in \hat{\Rc}(y_0),$  $Y$ is $\F^{X(\cdot)}$-adapted. In particular, $Y_T$ is $\Fc_T^{X(\cdot)}$-measurable.
\end{lemma}
\begin{proof} Let $\Q$ be the probability from \Cref{rule'}.\ref{rule4'} defined on an extended space. Using Girsanov's theorem, we can find an equivalent measure $\P^0 \sim \Q$ such that 
\begin{align*}
X_t(\cdot)&=x_0+\int_0^t\int_\cdot \sigma(s,X_{\cdot \wedge s}) M^0(\mathrm{d}v,\mathrm{d}s), \\
Y_t&=y_0 + \int_0^t \int_V H(s,X_{\cdot \wedge s},Y_s,Z_s(v)) m^0(\mathrm{d}v,\mathrm{d}s) +\int_0^t \int_V Z_s(v)^\top \sigma(s,X_{\cdot \wedge s})  M^0(\mathrm{d}v,\mathrm{d}s),\; t \in [0,T], \P^0\text{\rm--a.s.},
\end{align*}
for some $M^0$ with intensity $m^0.$ Since $\sigma$ admits an inverse from the left (see \eqref{invsigma}) and \Cref{x} admits unique strong solution, it can readily be seen that $\F^{X(\cdot)}=\F^{M^\smallfont{0}(\cdot)},$ see also \Cref{rem:corresp}. Further, $Z$ is $\F^{X(\cdot)}$-adapted since $m^Z$ is. Using that $H$ is Lipschitz-continuous in the $y$-variable, $Y$ is, in fact, the unique strong solution to the equation 
\[
Y_t=y_0 + \int_0^t \int_V H(s,X_{\cdot \wedge s},Y_s,Z_s(v)) m^0(\mathrm{d}v,\mathrm{d}s) +\int_0^t \int_V Z_s(v)^\top \sigma(s,X_{\cdot \wedge s})  M^0(\mathrm{d}v,\mathrm{d}s),\; t \in [0,T], \P^0\text{\rm--a.s.},
\] and, as such, it is $\F^{X(\cdot)}$-adapted, since the coefficients and the driving noise are.
\end{proof}

Similarly as in the previous section, we get the following compactness and continuity  results. Note that we obtain compactness under slightly weaker assumptions than in the case of the previous section. This is thanks to the simplified dynamics of $Y.$

\begin{lemma} \label{tight'} Let {\rm\Cref{assPrincip2}} hold. Then the set 
\begin{align*}
\hat{K}_{\varepsilon,R}\coloneqq\Bigg\lbrace \P \in \bigcup_{\{y_\smallfont{0} \in \R : y_\smallfont{0} \geq r_\smallfont{0}\}}\hat{\Rc}(y_0) : y_0 + \E^{\P}\bigg[ \int_0^T \int_{\R^\smallfont{d}} \lVert z \rVert^{2+\varepsilon} m^Z_s(\mathrm{d}z)\mathrm{d}s\bigg]\leq R \Bigg\rbrace,
\end{align*}
is relatively compact for the weak topology for every $\varepsilon>0$ and $R>0$. Moreover, every cluster point $\P^\prime \in \overline{\hat{K}_{\varepsilon,R}}$ satisfies 
\[
y_0+\E^{\P^\smallfont{\prime}}\bigg[\int_0^T \int_{\R^\smallfont{d}} \lVert z \rVert^{2+\varepsilon} m^Z_s(\mathrm{d}z)\mathrm{d}s\bigg]\leq R, \] 
as well as all conditions in {\rm\Cref{rule'}} for some $y_0 \in \R$ except possibly {\rm \ref{rule4'}}.
\end{lemma}

\begin{proof}
Analogous arguments as in the proof of \Cref{tight}.
\end{proof}

Using these results, one can obtain a maximiser. Should this maximiser satisfy the required measurability conditions, we get existence of an optimal feedback contract. The arguments are analogous to the previous section.

\begin{corollary} \label{minFX}
Let {\rm\Cref{assPrincip2}} be valid. Then the function
\[ 
\P \longmapsto \E^{\P} \big[\hat{F}(X_{\cdot\wedge T},Y_T)\big],
\] 
admits a maximiser on $\overline{\hat{K}_{\varepsilon,R}}$ for any $q>1$ and $R>0$ such that $\overline{\hat{K}_{\varepsilon,R}}\neq\emptyset.$ Moreover, if the maximiser satisfies {\rm\Cref{rule'}}.\ref{rule4'}, then it belongs to $\hat{K}_{\varepsilon,R}.$ 
\end{corollary}
\begin{proof}
The proof follows the same steps as the proof of \Cref{minim} using \Cref{tight'}.
\end{proof}

\begin{remark} $(i)$ Let us remark that the connection of the two formulations is obvious in this case. Indeed, having optimal $\xi \in \Fc_T^{X(\cdot)}$ immediately yields an optimal contract in the original, strong,  setting. Roughly speaking, this is because we impose that the weak control problem is, in fact, a strong control problem by adding an extra constraint on measurability.

\medskip $(ii)$ An analogue of {\rm \Cref{rem:uniform_integr}} holds in this case as well.
\end{remark}

\begin{remark} Let $\P \in\overline{\hat{K}_{\varepsilon,R}}$ for some $\varepsilon>0$ and $R>0.$ Using {\rm \cite[Lemma 3.7]{haussmann1990existence}}, see also {\rm \cite[Proposition 2.6]{el1987compactification}}, we can without loss of generality assume that $m^Z$ is $\F^{X,Y}$-predictable. To be precise, without changing the value of the objective function, we can replace $\P$ with $\tilde{\P}$ under which $m^Z$ $\tilde{\P}$--almost surely coincides with an $\F^{X,Y}$-predictable process $\tilde{m}^Z$. Thus, there exists a $\Pc\Mc(\F^{X,Y})\otimes \Bc(V)$-measurable $($in fact even predictable$)$ map $Z : [0,T] \times \Cc([0,T],\R^{d+1}) \times V \longrightarrow \R$ such that
\[ m^Z=Z(X,Y,\cdot)_{\#} m^0. \]
Assume that there exists a measurable function $C: [0,T] \times \Cc([0,T],\R^d) \times V \longrightarrow [0,\infty)$  such that \[ \lvert Z_s(x,y,v) - Z_s(x,y^\prime,v)  \rvert\leq C_s(x_{\cdot \wedge s},v)  \| y_{\cdot \wedge s} -y^\prime_{\cdot \wedge s} \|,\;(s,x,y,y^\prime,v)\in [0,T]\times \Cc([0,T],\R^d)\times \Cc([0,T],\R) \times \Cc([0,T],\R)\times V.  \]
Provided that the function $C$ is integrable enough $($\emph{e.g.} uniform boundedness is sufficient$)$, one can analogously as in {\rm \Cref{FXadaptedness}} verify using standard methods that $Y$ is a strong solution to the  given {\rm SDE} on any extension supporting a suitable martingale measure $M^0$. In particular $Y$ as well as $Z$ are $\F^{X(\cdot)}$-adapted. Unfortunately, verifying such regularity of a minimiser is not easily achievable in general, and this would need a case-by-case analysis.
\end{remark}

\section{Additional remarks, extensions and examples} \label{sec:additional}

\subsection{Random time horizon and retirement}
Among possible extensions of our framework let us mention the introduction of a retirement time. That is to say, similarly as in \citeauthor*{sannikov2008continuous} \cite{sannikov2008continuous}, \citeauthor*{possamai2020there} \cite{possamai2020there}, and \citeauthor*{lin2022random} \cite{lin2022random}, a contract offered by the principal consists of a pair $(\xi,\tau),$ where $\tau \in [0,\infty)$ is a, no longer necessarily bounded, $\F$--stopping time representing the retirement of the agent and $\xi,$ as before, represents the lump-sum payment at such a time. In this framework, $\xi$ is thus required to be $\Fc_\tau$-measurable. An obvious generalisation of our results is then that a contract $(\xi, \tau)$ is incentive compatible if and only if $(\xi,\tau)=(U_a^{-1}(Y_\tau),\tau),$ where $Y$ is given by \Cref{ICcontract}, respectively \Cref{ICcontractFB}, with now allowing $t \in [0,\infty),$ thus making use of a BSDE representation of the continuation utility of the agent with a random time horizon. The principal's problem would in such a case further involve maximisation over all $\F$--stopping times $\tau.$ In the weak formulation, this would then translate into an extension of the canonical space by adding the extended half-axis $[0,+\infty]$ with the canonical random variable $\tau,$ see \citeauthor*{{haussmann1990existence}} \cite[Section 3.11]{haussmann1990existence}. This formulation is indeed feasible and a similar weak control--stopping problem was, in fact, treated in \cite{haussmann1990existence}. However, note that further extension of the canonical space in the weak formulation might deepen the measurability issues pointed out earlier.

\subsection{Controlled discount factor} \label{sec:disc_f}
Let us comment on a potential relaxation of \Cref{discF}. As was mentioned in \Cref{rem_disc_fact}, \Cref{discF} allowed us to reformulate the principal's problem without the need to introduce the martingale problem in \Cref{rule} on the Skorokhod space of c\`{a}dl\`{a}g functions, which shall be denoted by $\Dc([0,T],\R)$. In the general case, we would need to consider the probability space $\Omega^{\Dc}\coloneqq\Cc([0,T],\R^{d})\times \Dc([0,T],\R) \times \Mc([0,T],\R^d) \times \Dc([0,T],\R)$  endowed with the canonical $\sigma$-algebra, canonical filtration $\F^\Dc$ and consider the canonical process $(X,Y,m^Z,C).$ We would then impose that the process
\begin{align*} M^{\phi}&\coloneqq \phi(X,Y)-\int_0^\cdot \int_{\R^\smallfont{d}} \hat{\mathscr{L}}\phi (s,X_{\cdot \wedge s},Y_s,z) m^Z_s(\mathrm{d}z)\mathrm{d}s\\
 &\quad- \int_0^\cdot \frac{1}{2}D^2_{y,y}\phi(X_s,Y_s) \mathrm{d}C_s-\sum_{0<s \leq \cdot }\Big[ \phi(X_{s},Y_s)-\phi(X_{s},Y_{s-})-\Delta Y_s  D_{y} \phi(X_s,Y_{s-}) \big],
 \end{align*} 
 is an $(\F^\Dc,\P)$-martingale and the process $C$ is non-decreasing, $\P$--almost surely, where $D_y$ and $D^2_{y,y}\phi$ denotes the first and the second derivative of $\phi$ with respect to the variable $y,$ respectively, and $\Delta Y$ denotes the process of jumps of $Y$. While this approach in principle appears to be feasible, we do not rule out that it might potentially be technically complex, and thus leave it for future research.

\subsection{Exponential utility does not help} \label{sec:exponential}
Let us consider problem \eqref{generalF} and assume that the agent has an exponential utility function. That is, the agent aims to maximise \[ \E^\P\bigg[ -\exp\bigg(-\gamma \bigg(\xi - \int_0^T \int_U \tilde{k}(s,X_{\cdot\wedge s},a)m^A(\mathrm{d}a,\mathrm{d}s)\bigg)\bigg)\bigg],\; A \in \Ac,\] for some $\gamma>0.$ This indeed fits in our framework by setting \[f(s,x,a) \coloneqq 0,\; U_a(c)\coloneqq -\mathrm{e}^{-\gamma c}\;{\rm and}\; k(s,x,u)\coloneqq -\gamma \tilde{k}(s,x,u). \]
A common method (see \emph{e.g.} \citeauthor*{hernandez2023pollution} \cite{hernandez2023pollution}, \citeauthor*{abi2024gaussian} \cite{abi2024gaussian}, \citeauthor*{martin2023risk} \cite{martin2023risk}, \citeauthor*{mastrolia2022agency} \cite{mastrolia2022agency} and numerous others) is to consider the certainty equivalent given by 
\[ \widehat{Y}\coloneqq\frac{-1}{\gamma}\log(-Y) + \int_0^\cdot \int_V \tilde{k}(s,X_{\cdot\wedge s},a(s,X_{\cdot \wedge s},Y_s,Z_s(v)))m^0(\mathrm{d}v,\mathrm{d}s). \]

By It\^{o}'s formula we obtain under $\P^{A^\star_a}$ for some $a \in \Upsilon$ the dynamics
\begin{align}\nonumber \widehat{Y}_t &= \xi- \int_t^T \int_{V} \frac{1}{2\gamma  Y_{s-}^2}Z_s^\top(v) \sigma\sigma^\top(s,X_{\cdot \wedge s})Z_s(v) m^0(\mathrm{d}v,\mathrm{d}s)+ \int_t^T \int_{V} \frac{Z_s^\top(v)}{\gamma  Y_{s-}} \sigma(s,X_{\cdot \wedge s})  \widetilde{M}^{A^\star_a}(\mathrm{d}v,\mathrm{d}s) \\
\nonumber &\quad+\int_t^T \frac{1}{\gamma  Y_{s-}} \mathrm{d}L_s + \int_t^T \frac{1}{2\gamma  Y_{s-}^2}  \mathrm{d}[L]_s + \sum_{s \in (t,T]} \Big[ \frac{1}{\gamma} (\log(-Y_s)-\log(-Y_{s-})) - \Delta Y_s \frac{1}{Y_{s-}} \Big]\\
\nonumber&= \xi - \int_t^T \int_{V}  \frac{1}{2\gamma}\widehat{Z}_{s}^\top(v) \sigma\sigma^\top(s,X_{\cdot \wedge s})\widehat{Z}(v) m^0(\mathrm{d}v,\mathrm{d}s)+ \int_t^T \int_{V} \frac{1}{\gamma} \widehat{Z}^\top(v) \sigma(s,X_{\cdot \wedge s})  \widetilde{M}^{A^\star_a}(\mathrm{d}v,\mathrm{d}s) \\
&\quad+\int_t^T \frac{1}{\gamma Y_s} \mathrm{d}L_s + \int_t^T \frac{1}{2\gamma Y_s^2}  \mathrm{d}[L]_s + \sum_{s \in (t,T]}\frac{1}{\gamma } \bigg( \log(-Y_s)-\log(-Y_{s-}) - \Delta Y_s \frac{1}{Y_{s-}} \bigg),\; t \in [0,T], \label{eqn:jumps}
  \end{align} where $\widehat{Z}_t(v)\coloneqq Z_t(v)/Y_{t-}$. This typically simplifies the right-hand side since, due to an absence of the last line in \eqref{eqn:jumps}, it no longer explicitly depends on the process $\widehat{Y},$ which we now consider as a state process. Moreover, the dynamics of $\widehat{Y}$ are generally simpler and more tractable. However, in our current setting, the presence of the orthogonal martingale $L$ makes such a transformation less interesting, and, in fact, it complicates the problem further due to the presence of the quadratic variation $[L]$ and jumps. Consequently, we encounter the same issue as explained in \Cref{sec:disc_f}, as we end up with a general controlled jump--diffusion.

\subsection{Non-relaxed controls} \label{sec:non-relaxed}
Let $(y_0,m^Z) \in [r_0,\infty)  \times \Vc$ be admissible for problem \eqref{FX} and assume that there exists a progressive process $Z : \Omega \times [0,T] \longrightarrow \R^d$ such that \[m^Z_t(\mathrm{d}z)=\delta_{Z_t}(\mathrm{d}z),\; t \in [0,T],\;\P^0\text{\rm--a.s.}\] Then the optimal responses to the contract $\xi\coloneqq U_a^{-1}\big(Y_T^{Y_\smallfont{0},m^\smallfont{Z},L}\big)$ are given by
\[ 
m^A_t(\mathrm{d}a)=\delta_{a(t,X_{\cdot \wedge t},Y_s,Z_s)}(\mathrm{d}a)\; t \in [0,T],\;\P^0\text{\rm--a.s.,} 
\] for some $a \in \Upsilon.$ That is, roughly speaking, if the principal does not randomise her action, neither does the agent and the problem reduces to the standard setting considered \emph{e.g.} in \cite{cvitanic2018dynamic}. Analogous results can be obtained for the problem \eqref{generalF}. Indeed, this might be of interest in practice since relaxed controls, while being a natural extension, might have unclear interpretation. 

\medskip In the case of the `weakened' principal's problem introduced in \Cref{relPrincip,sec:feedbackContr}, we have the following result in this regard, similar arguments were also used in the work of {\rm \citeauthor*{djete2023stackelberg} \cite{djete2023stackelberg}}. Interested readers may find more references therein as well as in \cite{el1987compactification,haussmann1990existence}.

\begin{example} Let us assume that the set \[ K(t,x,y)\coloneqq \big\{ \big( a(t,x,z),b(t,x,y,z) \big) : z \in \R^d \big\}, \] is convex for $\lambda$--almost every $t \in [0,T]$, for every $(x,y) \in \Cc([0,T],\R^{d+1})$, and that $-\hat{F}$ and $\hat{h}^i$ are convex functions.  Then for every $\P \in \hat{\Rc}(y_0)$ for some $y_0 \geq r_0$ there exists $\bar{\P} \in \hat{\Rc}(y_0)$ such that $\E^{\P} \big[F(X_{\cdot\wedge T},Y_T)\big]=\E^{\bar{\P}} \big[F(X_{\cdot\wedge T},Y_T)\big]$ and \[m^Z_t(\mathrm{d}z)=\delta_{Z_t}(\mathrm{d}z),\; t \in [0,T],\;\bar{\P}\text{\rm --a.s.},\] for some progressive process $Z : \Omega^c \times [0,T] \longrightarrow \R^d$. This follows directly from {\rm\cite[Theorem 3.6]{haussmann1990existence}}. Analogous results can be obtained for $\P \in \Rc.$
\end{example}

\subsection{Bounds on moments are in a certain sense necessary} \label{subsec:compactness}
Let us point our that, due to the nature of the problem, the coefficients $a$ and $\hat{a}$ introduced in \Cref{relPrincip,sec:feedbackContr} have inherently quadratic growth in the $z$ entry. Indeed, since the volatility coefficient of the process $\Yc,$ resp. $Y$, is linear in $z$, the quadratic variation is then necessarily quadratic. As a consequence, we consider it unlikely that the compactness conditions in \Cref{minim}, resp. \Cref{rem:uniform_integr}, and \Cref{minFX}, can be relaxed at this generality, and are in line the results of \citeauthor*{haussmann1990existence} \cite{haussmann1990existence}. As a matter of fact, we have the following observation.

\begin{remark}
Let $\xi \in \Ic\Cc$ be such that $\E^{\P^\smallfont{0}}\big[\lvert U_a(\xi) \rvert^p\big]<\infty,$ where $p>1$ is given in {\rm\Cref{assUkf}}. Using \eqref{bdBSDE} we obtain the following bound
\begin{equation*} 
\E^{\P^\smallfont{0}} \bigg[ \sup_{t \in [0,T]}\lvert Y_t \rvert^p + \bigg( \int_0^T \int_{\R^\smallfont{d}} \big\| z^\top \sigma(x,X_{\cdot \wedge s}) \big\|^2 m^Z_s(\mathrm{d}z)\mathrm{d}s \bigg)^\frac{p}{2} + \sup_{t \in [0,T]} \lvert L_t \rvert^{p} \bigg] \leq K \E^{\P^\smallfont{0}}\bigg[ \lvert U_a(\xi) \rvert^p +  \bigg( \int_0^T \lvert g_s(0,0) \rvert \mathrm{d}s\bigg)^p\bigg].
\end{equation*} Moreover, similarly as in {\rm\Cref{int}} and exploiting once again boundedness of $\lambda,$ one can show that for any $q \in (1,p)$ there exists a constant $c_q>0$ such that for any $A \in \Ac$ we have 
\begin{equation*} 
\E^{\P^\smallfont{A}} \bigg[ \sup_{t \in [0,T]}\lvert Y_t \rvert^q + \bigg(\int_0^T \int_{\R^\smallfont{d}} \big\| z^\top \sigma(x,X_{\cdot \wedge s}) \big\|^2 m^Z_s(\mathrm{d}z)\mathrm{d}s \bigg)^\frac{q}{2} + \sup_{t \in [0,T]} \lvert L_t \rvert^{q} \bigg] \leq c_q.
\end{equation*} It is then clear that for any $a \in \Upsilon$ we have
\[ \lim_{M \rightarrow \infty} \E^{\P^\smallfont{A_a^\star}} \bigg[ \bigg(\int_0^T \int_{\R^\smallfont{d}} \big\| z^\top \sigma(x,X_{\cdot \wedge s}) \big\|^2 \mathbf{1}_{\{\|z\| \geq M \}} m^Z_s(\mathrm{d}z)\mathrm{d}s \bigg)^\frac{q}{2} \bigg]=0, \] by the dominated convergence theorem. In other words, if there exists an optimal contract $\xi^\star \in \Ic\Cc$ in the principal's problem such that $U_a(\xi^\star)$ is sufficiently integrable, $p$ in {\rm\Cref{assUkf}} is large enough and $\sigma$ is uniformly elliptic, then the condition in {\rm \Cref{rem:uniform_integr}} is satisfied with $K \coloneqq \{ \Q^{A_a^\star} \},$ where $\Q^{A_a^\star} \in \Pc(\Omega^c \times \R)$ is the law of $(X,\Yc,m^Z,L_T)$ under $\P^{A_a^\star}.$ Therefore, this condition is in a sense also necessary. However, as pointed out earlier, showing that we can \emph{a priori} restrict to optimisation over a set satisfying such an assumption might not be easily attainable.
\end{remark}

\subsection{Inter-temporal payments}
A possible extension of our setting is to consider inter-temporal payments to the agents. Because these would play a role of a control in the principal's problem, we need to introduce its relaxed version if we wish to extend our methods to this framework.

\medskip
More specifically, this would involve the introduction of contracts as pairs $(\xi, \pi)$, where $\xi \in \Ic\Cc$ and $\pi : [0,T] \times \Omega \times \Bc(V) \longrightarrow Y$ is $\Pc\Mc \otimes \Bc(V)$-measurable, Polish-space--valued and represents the `relaxed' flow of inter-temporal payments. In this case, the agent's problem could be formulated as \[ 
\sup_{A \in \Ac}\E^{\P^{\smallfont{A}}} \bigg[K_{T}^A(X_{\cdot \wedge T}) U_a(\xi)- \int_0^T \int_U K_{s}^A(X_{\cdot \wedge s}) f(s,X_{\cdot \wedge s},A_s(v),\pi_s(v))
m^0(\mathrm{d}v,\mathrm{d}s)  \bigg],\; A \in \Ac,
\] where $f : [0,T]\times \Cc([0,T],\R^d) \times U \times Y \longrightarrow \R$ is the running cost of the agent. The principal's problem would then involve maximisation over measures on $Y \times \R^{d},$ instead of just $\R^d,$ of the form \[ m^Z_s(\mathrm{d}y,\mathrm{d}z)=(\pi_s(v),Z_s(v))_{\#} m^0_s(\mathrm{d}v),\; s \in [0,T]. \]
While our setting can be extended to this framework quite easily, the practical interpretation of the measure-valued inter-temporal payments is unclear. If we, however, have appropriate convexity properties of the involved functions, see \Cref{sec:non-relaxed}, we can recover the usual, non-relaxed setting.

\subsection{Volatility control}
Potential extensions of our framework include a scenario in which the agent controls the volatility coefficient as well. That is, the coefficient $\sigma$ is now a measurable function
\[ \sigma : [0,T] \times \Cc([0,T],\R^d) \times U \longrightarrow \R^{d \times k} \] and if the agent chooses a control $A \in \Ac$, then the state process $X$ follows under $\P^{A}$ the dynamics
\begin{equation*}
X_t=x_0+\int_0^t \int_V\sigma(s,X_{\cdot \wedge s},A_s(v))\lambda(s,X_{\cdot \wedge s},A_s(v)) m^0(\mathrm{d}v,\mathrm{d}s)+\int_0^t\int_V \sigma(s,X_{\cdot \wedge s},A_s(v))\,\widetilde{M}^A(\mathrm{d}v,\mathrm{d}s),\; t \in [0,T].
\end{equation*}
Equivalently,
\begin{equation*}
X_t=x_0+\int_0^t \int_U \sigma(s,X_{\cdot \wedge s},u)\lambda(s,X_{\cdot \wedge s},u) \,m^A(\mathrm{d}u,\mathrm{d}s)+\int_0^t\int_U \sigma(s,X_{\cdot \wedge s},u)\,M^A(\mathrm{d}u,\mathrm{d}s),\; t \in [0,T],\; \P^{A}\text{\rm--a.s.}
\end{equation*}
We expect that our setting extends to this framework. In such a case, we conjecture that the set of all incentive compatible contracts will be, similarly as in \cite[Theorem 3.6]{cvitanic2018dynamic}, characterised by means of a suitable form of second-order BSDE driven by the martingale measure $X(\cdot).$ We leave the study of this case for future research.

\section{Applications and implementability}\label{sec:applications}

This section explains $(i)$ why realistic constraints on contracts make the classical PDE approach essentially intractable, and $(ii)$ how the relaxed formulation yields existence \emph{and} implementability through finitely many observable statistics of the driving martingale measure. We provide three application templates and a constructive finite-observation approximation result.

\subsection{Why realistic constraints defeat the PDE route}\label{subsec:whyPDE}
In the dynamic programming approach à la \cite{cvitanic2018dynamic}, the principal’s value function solves a fully non-linear, degenerate parabolic PDE with state variables $x$ and the agent’s continuation utility $y$. Under realistic restrictions on $\xi$ and on intermediate payments, boundary/obstacle conditions are defined by potentially \emph{discontinuous} functionals of $(X,\xi)$ (\emph{e.g.} liquidation caps, budget/risk constraints), so the HJB gets coupled to additional obstacle equations. There is no general regularity theory in this setting; in particular, neither smoothness nor comparison giving closed-form verification is available in general dimensions. As a result, the PDE route cannot deliver existence of optimal contracts beyond very special, tractable cases.

\medskip
By contrast, the relaxed formulation recasts the principal’s problem as a weak control problem in which the control is the \emph{measure–valued} $m^Z$ encoding the distribution of $Z$ across~$V$. Under \Cref{assPrincipal}, the class of admissible laws is tight/compact and the objective/constraints are stable under weak convergence. Hence existence follows by standard compactness/semicontinuity arguments \emph{without} any PDE regularity.

\subsection{Three application templates}\label{subsec:templates}

Throughout, we fix a maximiser $a\in\Upsilon$, write the principal’s objective as in \Cref{sec:principal}, and we allow a possibly uncountable family of constraints
\[
\mathbb{E}^\P\big[g_i(X_{\cdot\wedge T},\xi)\big]\le 0,\; i\in I,
\]
as well as integral constraints depending on $m^Z$, see \Cref{rem:constraints}

\subsubsection{Limited liability, liquidation, budget/risk caps.}
Impose simultaneously
\[
0\le \xi \le \ell(X_T),\; \mathbb{E}^\P[\xi]\le B,\; \mathrm{CVaR}^\P_\alpha(\xi)\le R_\alpha.
\]
The first pair are the \emph{hard} constraints, already present in \Cref{constraints} and remain admissible as lower semicontinuous functionals. The constraint involving $\mathrm{CVaR}^\P_\alpha$ (the conditional value-at-risk at level $\alpha$) is convex and lower semicontinuous in law, thus stable under weak limits. If the first constraint is potentially reachable via PDEs, as exemplified recently in \citeauthor*{hernandez2024closed} \cite{hernandez2024closed}, it leads to a PDE on domain, whose boundary is given by the solution to another PDE, and for which boundary conditions are themselves given by yet other PDEs. As exemplified in a future work, see \citeauthor*{krsek2025limited} \cite{krsek2025limited}, regularity results for even the simplest version of these problems remain elusive. The other two constraints are worse in the sense that as far as we know, they cannot be simply tackled using PDEs. Notwithstanding, in the relaxed problem, they integrate seamlessly and existence follows from compactness of admissible laws.

\subsubsection{Prudential and pathwise drawdown constraints.}
Let $D_u(X,\xi)$ denote a pathwise drawdown functional at time $u\in[0,T]$ applied either to the agent’s continuation value or the principal’s interim P\&L. Require
\[
\mathbb{E}^\P\big[D_u(X,\xi)\big]\le 0,\; \forall u\in[0,T].
\]
This creates an \emph{uncountable} family $(g_u)_{u\in[0,T]}$, which our framework handles by construction. The PDE route would need to encode a continuum of moving obstacles, which is analytically prohibitive, if even possible.

\subsubsection{Demand–response and procurement with monotonic tariffs and variation caps.}
In settings like electricity demand response or PPP procurement, common policy constraints include:
\begin{itemize}
  \item non–negativity and monotonicity of tariffs in state variables (\emph{e.g.} consumption shortfall or delivery quality);
  \item fairness across time‑of‑use (majorisation-type inequalities);
  \item caps on the intra‑day variation of payment schedules.
\end{itemize}
Similar problems, without constraints have been tackled in \cite{aid2022optimal} or \cite{elie2021mean}. These constraints generically become convex, and l.s.c. in $(X,\xi)$ and, by \Cref{rem:constraints}, may include integral terms in $m^Z$ capturing average intensity of the principal’s signal extraction. Existence again follows from the same compactness arguments, and is completely out of reach for PDE methods.

\subsection{Finite-observation implementability}\label{subsec:finite-obs}

A central practical question is: must the principal observe $X_s(A)$ for \emph{all} $s\in[0, T]$ and all $A\in\mathcal{B}(V)$ to evaluate a contract? The answer is no. The completed canonical $\sigma$-algebra generated by $\{X_s(A)\}$ is \emph{countably} generated: there exists a countable family $\{G_j\}_{j\in\N^{\smallfont \star}}$ of open sets in $\mathcal{B}(V)$ and dyadic times such that
\[
\Fc_T^{X(\cdot)}
=\sigma\big(X_{i2^{\smallfont{-}\smallfont{n}}}(G_j): (i,n,j)\in\mathbb{N}^3\big) \vee \sigma(\mathcal{N}),
\]
see \Cref{FXfiltr}. Consequently, every contract $\xi$ will admit approximants depending on finitely many co‑ordinates of the form $X_{i2^{\smallfont{-}\smallfont{n}}}(G_j)$.

\begin{proposition}[Finite-statistic implementation]\label{prop:finite-obs}
Fix an admissible $\Fc_T^{X(\cdot)}$-measurable contract $\xi$ and assume that there's a finite number of constraints.  Assume that the principal’s objective $U_p$ as well as the functions $g^i$, $i \in \{1,\ldots,N\},$ are continuous in the second entry and of linear growth. Then for every $\varepsilon>0$ there exist positive integers $(n,k)$ and a contract
\[
\xi^{(n,k)} \; \text{\rm measurable with respect to}\; \sigma\big(X_{i2^{\smallfont{-}\smallfont{n}}}(G_j): i\in\{1,\dots, 2^n\},\; j\in\{1,\dots,k\}\big),
\]
such that
\[
\Big| \mathbb{E}^\P\big[U_p(X_{\cdot\wedge T},\xi^{(n,k)})\big] - \mathbb{E}^\P\big[U_p(X_{\cdot\wedge T},\xi)\big] \Big| \le \varepsilon
\enspace {\rm and}\enspace
\sup_{i\in \{1,\ldots, N\} } \mathbb{E}^\P\big[g_i(X_{\cdot\wedge T},\xi^{(n,k)})\big] \le \varepsilon.
\]
\end{proposition}
\begin{proof}
For notational simplicity, let $\Gc_{n}\coloneqq\sigma\big(X_{i2^{\smallfont{-}\smallfont{n}}}(G_j): i\in\{1,\dots, 2^n\},\; j\in\{1,\dots,n\}\big)$. Up to redefining $\xi$ on a $\P$--null set, we may assume that $\xi$ is $(\bigvee_{n \in \N} \Gc_{n})$-measurable, see \Cref{FXfiltr}. Let us set $\xi^{n}\coloneqq \E^{\P}[\xi | \Gc_{n}]$. Since $\E^{\P} [ |\xi|]< \infty$, we have that the process $(\xi^{n})_{n \in \N^\star}$ is a (discrete-time) $\P$--uniformly integrable martingale with respect to the filtration $(\Gc_n)_{n \in \N^\smallfont{\star}}$. By the martingale convergence theorem, we have $\lim_{n \rightarrow \infty}\xi^n=\xi,$ $\P$--almost surely and in $\L^1(\R,\Fc^{X(\cdot)},\P)$.

\medskip
As $U_p$ and $g^i$, $i \in \{1,\ldots,N\}$ have at most linear growth, are continuous in the second entry, and  the sequence $(\xi^n)_{n \in \N^\smallfont{\star}}$ is $\P$--uniformly integrable, we have that
 \[ \lim_{n \rightarrow \infty} \mathbb{E}^\P\big[U_p(X_{\cdot\wedge T},\xi^{n})\big] =  \mathbb{E}^\P\big[U_p(X_{\cdot\wedge T},\xi)\big],\; {\rm and}\; \lim_{n \rightarrow \infty} \mathbb{E}^\P\big[g_i(X_{\cdot\wedge T},\xi^{n})\big] =  \mathbb{E}^\P\big[g_i(X_{\cdot\wedge T},\xi^{n})\big], \; i \in \{1,\ldots,N\}. \] It thus suffices to find an $n \in \N^\star$ large enough to conclude. 
 \end{proof}

Practically speaking, one may fix positive integers $M$ and $N$, choose a time grid $(t_i)_{i\in\{1,\dots,N\}}$ and a finite measurable partition $(A_j)_{j\in\{1,\dots,M\}}$ of $V$ generated by the $G_j$; design a contract as a function of the \emph{observable} statistics $(X_{t_i}(A_j))_{(i,j)\in\{1,\dots,N\}\times\{1,\dots, M\}}$. Then,  \Cref{prop:finite-obs} guarantees that this finite‑statistic contract is $\varepsilon$-close in both objective and constraints to any target admissible contract. Economically, this is the continuous‑time analogue of the one-period randomisation and compactness argument of \cite{kadan2017existence}, now made concrete via a finite menu of statistics rather than the whole $\sigma$-algebra. 

\medskip
\medskip We point out that the probability measure $\P$ in \Cref{prop:finite-obs} might not be optimal for the agent given the contracts $\xi^{n,k}$, $(n,k) \in \N^2.$ Ensuring that the agent’s response remains optimal when approximating is generally a difficult problem due to the two-stage optimization procedure. However, when one wishes to discretize only the `space' component of the martingale measure, one may use the chattering lemma to approximate the relaxed control in the principal’s problem by a sequence of standard (non-relaxed) controls. This provides an approximation that preserves optimality in the agent’s problem and offers a way to interpret the relaxed contract as a cluster point of non-relaxed contracts in a suitable sense. We refer to \citeauthor*{el1987compactification} \cite[Section 4]{el1987compactification} and the references therein for further details, as well as 
\citeauthor*{mezerdi2002necessary} \cite{mezerdi2002necessary}, \citeauthor*{bahlali2006approximation} \cite{bahlali2006approximation} and \citeauthor*{bahlali2018relaxed} \cite{bahlali2018relaxed}. The main reasons we do not show this in our setting are the relatively strict conditions required for the chattering lemma to hold and the additional difficulties arising from the presence of constraints, since there is generally no guarantee that a constrained contract can be approximated in this way.

\subsection{When randomisation is \texorpdfstring{not}{not} needed}\label{subsec:no-rand}
Randomisation is a tool for compactness; if the primitives are convex/concave in the appropriate variables, the relaxed and pure problems coincide. \Cref{sec:non-relaxed} details conditions under which the principal can recover a pure $Z$ from $m^Z$ (standard measurable selection and Carathéodory convexity). Conversely, as the one–period literature (see \cite{kadan2017existence}) emphasises, randomisation is sometimes \emph{necessary} for existence; our dynamic framework inherits this feature via $m^Z$ while preserving implementability through \Cref{prop:finite-obs}.

{\small
\bibliography{bibliographyDylan}}
\appendix 
\section{BSDEs}
\subsection{Preliminary results}
In this part of the appendix, we work in the setting of \Cref{agent}. We fix $p>1$ throughout this section.
\begin{lemma}\label{k-w} The space $\L^p_d(\F,\P^0,X(\cdot))$ is a Hilbert space and every martingale $M \in \M^p(\F,\P^0)$ admits unique decomposition
\begin{align*}
M_t =\int_0^t \int_V Z_s(v)\cdot X(\mathrm{d}v,\mathrm{d}s)+\int_0^t\mathrm{d}L_s,\; t \in [0,T],\;\P^0\text{\rm--a.s.,}
\end{align*} 
for some $Z \in \L^p_d(\F,\P^0,X(\cdot))$ and $L \in \M^p_{X(\cdot)^\smallfont{\perp}}(\F,\P^0).$
\end{lemma}
\begin{proof} The lemma is a straightforward generalisation of the Kunita--Watanabe decomposition, see for instance \citeauthor*{protter2005stochastic} \cite[page 183]{protter2005stochastic}.
\end{proof}

\begin{lemma} \label{blumethal} The $\sigma$-algebra $\Fc_0^{X(\cdot)}$ is $\P^0$-trivial, that is to say
\[
\P^0(A) \in \{0,1 \},\; A \in\Fc_0^{X(\cdot)}.
\]
\end{lemma}
\begin{proof} First, let us note that since $\sigma$ is assumed to be of full rank and \Cref{x} admits unique strong solution, we can w.l.o.g. assume $x_0=0$ and that $d=k$ and $\sigma \equiv \mathrm{I}_{d \times d}$ for simplicity (see also \Cref{rem:corresp}). Let us fix an arbitrary positive integer $n$, and fix $(A_1,\ldots A_n) \in \Bc(V)^n.$ 

\medskip
\emph{Step 1:}
in the first step, we show that the process that for any $s\in(0,T)$
\[
M_t^s(A_1,\ldots,A_n)\coloneqq \big(X_{t}(A_1)-X_{s}(A_1),\ldots,X_{t}(A_n)-X_{s}(A_n) \big), \; t \in (0,T-s],
\]
is $\P^0$-independent of $\Fc_s^{X(\cdot)}$. If $\mu(A_i)=0$ for some $i \in \{1,\ldots,n \},$ then 
\[
\E^{\P^\smallfont{0}} \big[X_t(A_i)^2\big] = \mu(A_i)t=0,\; t \in [0,T]. 
\] 
Hence, $X(A_i)\equiv 0.$ We therefore w.l.o.g. assume $\mu(A_i)>0,$ for every $i \in \{1,\ldots,n \}.$ For simplicity of notation, assume that $n=2$.  The general case can be proved in the very same fashion. It is clear that there is a bijective relation between the process $M^s(A_1,A_2)$ and 
\begin{align}\label{brm} 
\bigg(\frac{X_t(A_1 \setminus A_2)-X_s(A_1 \setminus A_2)}{\sqrt{\mu(A_1 \setminus A_2)}}, \frac{X_t(A_2 \setminus A_1)-X_s(A_2 \setminus A_1)}{\sqrt{\mu(A_2 \setminus A_1)}} , \frac{X_t(A_1 \cap A_2)-X_s(A_1 \cap A_2)}{\sqrt{\mu(A_1 \cap A_2)}} \bigg)_{t \in (0,T-s]}. 
\end{align} 
It is clear that \eqref{brm} is a $3d$-dimensional $(\F^{X(\cdot)},\P^0)$--Brownian motion by L\'{e}vy's characterisation. Independence with respect to $\Fc_s^{X(\cdot)}$ is then immediate by the Markov property. 

\medskip
 \emph{Step 2:} let us have arbitrary $0<t_1\leq \ldots \leq t_n$, as well as $\psi \in \Cc_b(\R^{d \cdot n}, \R)$ and $B \in \Fc_0^{X(\cdot)},$ where $\Cc_b(\R^{d \cdot n}, \R)$ denotes the set of all continuous bounded functions from $\R^{d \cdot n}$ to $\R.$ It follows from continuity and boundedness that
\begin{align*}
\E^{\P^\smallfont{0}} \Big[\mathbf{1}_B \psi\big(X_{t_1}(A_1),\ldots,X_{t_n}(A_n)\big)&=\lim_{s \searrow 0}\E^{\P^\smallfont{0}} \mathbf{1}_B \psi\big(X_{t_1}(A_1)-X_{s}(A_1),\ldots,X_{t_n}(A_n)-X_{s}(A_n)\big)\Big] \\
&=\P^0[B]\lim_{s \searrow 0} \E^{\P^\smallfont{0}}\big[\psi\big(X_{t_1}(A_1)-X_{s}(A_1),\ldots,X_{t_n}(A_n)-X_{s}(A_n)\big)\big]\\
&=\P^0[B]\E^{\P^\smallfont{0}}\big[ \psi\big(X_{t_1}(A_1),\ldots,X_{t_n}(A_n)\big)\big].
\end{align*} 
Hence, $\Fc_0^{X(\cdot)}$ is $\P^0$-independent of $\sigma ((X_s(A_i) : i \in \{1,\ldots,n\},\; t>0).$ It is clear to see from continuity of paths that 
\[
\sigma ((X_s(A_i) : i \in \{1,\ldots,n\},\; t>0)=\sigma ((X_s(A_i) : i \in \{1,\ldots,n\},\; t\geq0).
\]

\emph{Step 3:} by \Cref{FXfiltr}, the filtration $\F^{X(\cdot)}$ is the augmentation of $\F^o=(\Fc^o_t)_{t\in[0,T]}$ with
\[ 
\Fc^{o}_t\coloneqq \bigvee_{k \in \N^\smallfont{\star}}\Fc_t^k,\; t \in [0,T], \; \text{where}\; \Fc_t^k=\sigma \big( X_s(G_j) : s \in [0,t],\; j\in\{0,\ldots,k\} \big) \vee \sigma(\Nc),\; t \in [0,T],\; k\in\N^\star, 
\]
for a countable system of open sets $\{ G_j : j \in \N^\star \}$ and where $\Nc$ denotes the system of all $\P^0$-null sets. We note that $\cup_{k \in \N^\smallfont{\star}}\Fc_t^k$ is a $\pi$-system generating $\Fc^{o}_t,$ which is $\P^0$-independent of $\Fc_0^{X(\cdot)}.$ Using Dynkin's lemma, we conclude that $\Fc^{o}_t$ is $\P^0$-independent of $\Fc_0^{X(\cdot)}$ for any $t>0.$

\medskip
\emph{Step 4:} by definition
\[
\Fc_0^{X(\cdot)}=\bigcap_{k \in \N^\smallfont{\star}} \Fc^{o}_{1/k} \vee \sigma(\Nc).
\]
It is then simple to see that, necessarily, $\Fc_0^{X(\cdot)}$ is $\P^0$-independent of itself, which concludes the proof.
\end{proof}

\begin{lemma} \label{martRep} 
Assume that $\F=\F^{X(\cdot)}$. Let $M \in \M^2(\F,\P^0).$ Then, there exists a unique $Z \in \L^2_d(\F,\P^0,X(\cdot))$ such that 
\begin{align} \label{mrep}
M_t =\int_0^t \int_V Z_s(v)\cdot X(\mathrm{d}v,\mathrm{d}s),\; t \in [0,T],\; \P^0\text{\rm--a.s.}
\end{align}
\end{lemma}

\begin{proof} Let us first prove uniqueness. Let us have two $(Z,\tilde{Z}) \in \L^2_d(\F,\P^0,X(\cdot))\times \L^2_d(\F,\P^0,X(\cdot))$ satisfying the property. Then
\begin{align*}
0=\E^{\P^\smallfont{0}} \bigg[ \bigg(\int_0^T \int_V \big(Z_s(v)-\tilde{Z}_s(v)\big)\cdot X(\mathrm{d}v,\mathrm{d}s) \bigg)^2\bigg]=\E^{\P^\smallfont{0}} \int_0^T \int_V \big\| \big(Z_s(v)-\tilde{Z}_s(v)\big)^\top\sigma(s,X_{\cdot \wedge s})\big\|^2 m^0(\mathrm{d}v,\mathrm{d}s),
\end{align*} 
which yields $Z=\tilde{Z}.$

\medskip
Now we prove existence. It is sufficient to show that for every random variable $\zeta \in \L^2(\Fc_T,\P^0),$ such that $\E^{\P^\smallfont{0}}[\zeta]=0,$ there exists $Z \in \L^2_d(\F,\P^0,X(\cdot))$ such that 
\begin{align*}
\zeta=\int_0^T \int_V Z_s(v)\cdot X(\mathrm{d}v,\mathrm{d}s),\quad t \in [0,T],\; \P^0\text{\rm--a.s.}
\end{align*} 
Because $\sigma$ is assumed to be of full rank and \Cref{x} admits unique strong solution, we can w.l.o.g. assume that $d=k$ and $\sigma \equiv \mathrm{I}_{d \times d}$ (see \eqref{invsigma} and \Cref{rem:corresp}).

\medskip
\emph{Step 1:} let us first assume $\zeta=g(X_{\cdot \wedge T}(A))$ for some measurable $g: \Cc([0,T],\R^d) \longrightarrow \R$ and $A \in \Bc(V).$ If $m^0_t(A)=\mu(A)=0,\, t \in [0,T],$ then 
\[ 
\E^{\P^\smallfont{0}} \big[X_t(A)^2\big] = \E^{\P^\smallfont{0}} \big[\mu(A)t\big]=0,\; t \in [0,T].
\]
Hence, $X(A)\equiv 0$ and we can choose $Z \equiv 0.$ If $\mu(A) > 0,$ then by L\'{e}vy's characterisation, $W_t\coloneqq X_t(A)/\sqrt{\mu(A)}, t \in [0,T],$ is a $d$-dimensional Brownian motion and hence (recall that $\mu(A)$ is deterministic), we can use the standard martingale representation in the $\P^0$-completed canonical filtration of $W,$ denoted by $\F^W,$ to get existence of $\tilde{Z} \in \L^2_d(\F^W, \P^0,W)$ such that 
\[ 
g\big(\mu(A)W_{\cdot \wedge T}\big)=\int_0^T \tilde{Z}_s\cdot\mathrm{d}W_s,\; \P^0\text{\rm--a.s.}
\]
It then suffices to take $Z_s(v)=\frac{1}{\sqrt{\mu(A)}}\tilde{Z}_s\mathbf{1}_{\{ v \in  A \}},$ $t \in [0,T]$, $v \in V$, to get
\[ 
g\big(X_{\cdot \wedge T}(A)\big)=g\big(\mu(A)W_{\cdot \wedge T}\big)=\int_0^T \tilde{Z}_s\cdot\mathrm{d}W_s= \int_0^T \int_V Z_s(v)\cdot X(\mathrm{d}v,\mathrm{d}s),\;  \P^0\text{\rm--a.s.}
\]

\emph{Step 2:} let us assume next that $\zeta=g(X_{\cdot \wedge T}(A),X_{\cdot \wedge T}(B))$ for some measurable $g: \Cc([0,T],\R^d)^2 \longrightarrow \R$ and $(A,B) \in \Bc(V)^2.$ Similarly as before, we can w.l.o.g. assume $\mu(A)>0$, and $\mu(B)>0.$ We note that we can rewrite 
\[
g\big(X_{\cdot \wedge T}(A),X_{\cdot \wedge T}(B)\big)=\tilde{g}\big(X_{\cdot \wedge T}(A \setminus B),X_{\cdot \wedge T}(B \setminus A),X_{\cdot \wedge T}(A \cap B)\big),
\]
for $\tilde{g}(a,b,c)\coloneqq g(a+c,b+c).$ Again, by L\'{e}vy's characterisation, the process 
\[
(W^1_t,W^2_t,W^3_t)\coloneqq \bigg(\frac{1}{\sqrt{\mu(A \setminus B)}}X_t(A \setminus B), \frac{1}{\sqrt{\mu(B \setminus A)}} X_t(B \setminus A), \frac{1}{\sqrt{\mu(A \cap B)}}X_t(A \cap B) \bigg),\; t \in [0,T],
\] 
is a $3d$-dimensional Brownian motion because $(A \setminus B)$, $(B \setminus A)$, and $(A \cap B)$ are pairwise disjoint. Hence, similarly as before 
\[
\tilde{g}\big(X_{\cdot \wedge T}(A \setminus B),X_{\cdot \wedge T}(B \setminus A),X_{\cdot \wedge T}(A \cap B)\big)= \sum_{i=1}^3 \int_0^T \tilde{Z}^i_s\mathrm{d}W^i_s,
\]
for some $\tilde{Z}^i \in \L^2_d(\F^{W^i},\P^0,W^i)$, $i \in \{1,2,3\}.$ Therefore we can take
\[ 
Z_s(v)\coloneqq \frac{1}{\sqrt{\mu(A \setminus B)}}\tilde{Z}_s^1\mathbf{1}_{\{ v \in  A \setminus B\}}+\frac{1}{\sqrt{\mu(B \setminus A)}}\tilde{Z}_s^2\mathbf{1}_{\{ v \in  B \setminus A\}}+\frac{1}{\sqrt{\mu(A \cap B)}}\tilde{Z}_s^3\mathbf{1}_{\{ v \in  A \cap B\}}.
\]

\emph{Step 3:} if $\zeta=g\big((X_{\cdot \wedge T}(A_i))_{i\in\{1,\ldots,\ell\}}\big)$ for some Borel-measurable $g: \Cc([0,T],\R^d)^\ell \longrightarrow \R$, $\ell \in \N^\star$, and $A_i\in \Bc(V),$ for $i \in \{1, \ldots, \ell\},$ we can proceed exactly as in the previous step.

\medskip
\emph{Step 4:} by \Cref{FXfiltr}, there exists a system of open sets $\{ G_j: j \in \N^\star \}$ such that $\Fc_T=\bigvee_{k \in \N^\smallfont{\star}} \Fc^{k}_T$, with 
\[
\Fc^{k}_T=\sigma \big( X_{\cdot \wedge T}(G_j) : j\in\{1,\ldots,k\}\big) \vee \sigma(\Nc) . 
\] 

Let now $\zeta \in \Fc_T$ be bounded and define $\zeta^{k}\coloneqq \E^{\P^\smallfont{0}}[ \zeta \vert \Fc_T^{k}],$ for $k \in \N^\star.$ By the Doob--Dynkin lemma, for any $k\in\N^\star$
\[
\zeta^{k}=g^{k}\big((X_{\cdot \wedge T}(G_j))_{j\in\{1,\ldots,k\}}\big),\; \P^0\text{\rm--a.s.,}
\] 
for some Borel-measurable $g^{k}.$ Hence, we can use \emph{Step 3} to conclude that  there exists $Z^{k} \in \L^2_d(\F,\P^0,X(\cdot))$ such that
\[ 
\zeta^{k}=\int_0^T \int_V Z^{k}_s(v)\cdot X(\mathrm{d}v,\mathrm{d}s),\; \P^0\text{\rm--a.s.}
\] 
Since by martingale convergence theorem and boundedness we have $\zeta^{k} \longrightarrow \zeta$ in $\L^2(\Fc_T,\P^0)$ for $k \longrightarrow \infty,$ we can easily verify that $(Z^{k})_{k\in\N^\smallfont{\star}}$ is a Cauchy sequence in $\L^2_d(\F,\P^0,X(\cdot))$ and, hence, admits a limit $Z$ there. Consequently
\[ 
\zeta=\int_0^T \int_V Z_s(v)\cdot X(\mathrm{d}v,\mathrm{d}s),\; \P^0\text{\rm--a.s.}
\]

\emph{Step 5:} if $\zeta \in \L^2(\Fc_T,\P^0)$ is in general unbounded, we define $\zeta^n\coloneqq (-n \vee \zeta) \wedge n$, for $n \in \N^\star.$ We use step 4 to get the desired representation for $\zeta^n$ and, since again $\zeta^n \longrightarrow \zeta$ in $\L^2(\Fc_T,\P^0),$ we again conclude by passing to the limit.

\end{proof}

\begin{remark} By standard localisation techniques, one can show that for any $(\F,\P^0)$--local martingale $M$ with $M_0=0$, there is $Z \in \L^2_{d,\mathrm{loc}}(\F,\P^0,X(\cdot)),$ such that {\rm \Cref{mrep}} holds. Moreover, if $M \in \M^p(\F,\P^0),$ then $Z \in \L^p_d(\F,\P^0,X(\cdot)).$
\end{remark}

\subsection{Well-posedness for BSDEs}
\begin{assumption}\label{lips} \begin{enumerate}
\item[$(i)$] The process $g_\cdot(y,z)$ is $\F$--progressively measurable for any $(y,z)\in \R \times \R^d$, and 
\[
\E^{\P^\smallfont{0}}\bigg[ \bigg(\int_0^T \lvert g_s(0,0) \rvert^2\mathrm{d}s \bigg)^\frac{p}{2}\bigg] < \infty. 
\]
\item[$(ii)$] There exists $C>0$ such that for all $(\omega,t) \in \Omega \times [0,T]$, and for all $(y,z,y^\prime,z^\prime)\in \R \times \R^d\times\R\times\R^d$
\[
\big| g_t (\omega,y,z)-g_t (\omega,y^\prime,z^\prime)\big| \leq C\big(\lvert y-y^\prime\rvert + \big\| (z-z^\prime)^\top\sigma(t,X_{\cdot \wedge t}(\omega)) \big\| \big).
\] 
\end{enumerate} 
\end{assumption}

\begin{theorem}\label{Exbsde} Let $\zeta \in \L^p(\Fc_T,\P^0).$ If {\rm\Cref{lips}} holds, then there is a unique solution $(Y,Z,L) \in \S^p(\F,\P^0) \times \L^p_d(\F,\P^0,X(\cdot))\times \M^p_{X(\cdot)^\smallfont{\perp}}(\F,\P^0)$ to the {\rm BSDE}
\begin{equation}\label{bsde} 
Y_t=\zeta-\int_t^T \int_V g_{s}(Y_s,Z_s(v))m^0(\mathrm{d}v,\mathrm{d}s) - \int_t^T \int_V Z_s(v) \cdot X(\mathrm{d}v,\mathrm{d}s) - \int_t^T\mathrm{d}L_s,\; t \in [0,T],\; \P^0\text{\rm--a.s.} 
\end{equation}
Moreover, there exists a constant $K>0$ such that
\begin{equation} \label{bdBSDE}
\| Y \|_{\S^\smallfont{p}(\P^\smallfont{0})}^p + \| Z \|^p_{\L^\smallfont{p}_\smallfont{d}(\F,\P^\smallfont{0},X(\cdot))} + \| L \|_{\M^\smallfont{p}(\P^\smallfont{0})}^p \leq K \E^{\P^\smallfont{0}}\bigg[ \lvert \zeta \rvert^p +  \bigg( \int_0^T \lvert g_s(0,0) \rvert \mathrm{d}s\bigg)^p\bigg].
\end{equation}
\end{theorem} 
\begin{proof} 
Let us first assume that $p=2$ and let us define $\H^2(\F,\P^0)$ as the space of all $\R$-valued, $\F$-optional processes $Y$ such that
\[ \E^{\P^\smallfont{0}}\bigg[ \int_0^T Y_s^2\mathrm{d}s\bigg] < \infty. \] Let us for $\beta>0$ and for $(Y,Z) \in \H^2(\F,\P^0) \times \L^2_d(\F,\P^0,X(\cdot))$ define
\[ 
\lVert (Y,Z) \rVert_\beta^2 \coloneqq \E^{\P^\smallfont{0}}\bigg[ \int_0^T \int_V \mathrm{e}^{\beta s} \big( Y_s^2+\lVert Z_s(v)^\top\sigma(s,X_{\cdot \wedge s}) \rVert^2 \big)m^0(\mathrm{d}v,\mathrm{d}s)\bigg]. \] 
It is straightforward to verify that $\left(\H^2(\F,\P^0) \times \L^2_d(\F,\P^0,X(\cdot)),  \lVert (\,\cdot\,,\,\cdot\,) \rVert_\beta\right)$ is a Banach space for any $\beta>0.$ Let us further define $\Theta: \H^2(\F,\P^0) \times \L^2_d(\F,\P^0,X(\cdot)) \rightarrow \H^2(\F,\P^0) \times \L^2_d(\F,\P^0,X(\cdot))$ by the following construction. For $(y,z) \in \H^2(\F,\P^0) \times \L^p_d(\F,\P^0,X(\cdot)),$ we apply Lemma \ref{k-w} to get the decomposition 
\[ \zeta - \int_0^T\int_V g_s(y_s,z_s(v))m^0(\mathrm{d}v,\mathrm{d}s)=\E^{\P^\smallfont{0}}\bigg[\zeta - \int_0^T\int_V g_s(y_s,z_s(v))m^0(\mathrm{d}v,\mathrm{d}s) \bigg]+\int_0^T\int_V Z_s(v) \cdot X(\mathrm{d}v,\mathrm{d}s)+\int_0^T\mathrm{d}L_s,
\]
and define 
\[
Y_t\coloneqq \E^{\P^\smallfont{0}} \bigg[ \zeta - \int_0^T\int_V g_s(y_s,z_s(v))m^0(\mathrm{d}v,\mathrm{d}s) \bigg\vert \Fc_t\bigg]+\int_0^t\int_V g_s(y_s,z_s(v))m^0(\mathrm{d}v,\mathrm{d}s).
 \]
We set $\Theta(y,z)\coloneqq (Y,Z).$ It is straightforward to verify that $(Y,Z,L) \in \H^2(\F,\P^0) \times \L^2_d(\F,\P^0,X(\cdot)) \times \M^2_{X(\cdot)^\smallfont{\perp}}(\F,\P^0).$ From the construction, we see that 
\[
Y_t=\zeta-\int_t^T\int_V g_s(y_s,z_s(v))m^0(\mathrm{d}v,\mathrm{d}s)-\int_t^T\int_V Z_s(v)\cdot X(\mathrm{d}v,\mathrm{d}s)-\int_t^T\mathrm{d}L_s.
\]
 Thus, it suffices to verify that $\Theta$ has a unique fixed-point. To that end, let us show that $\Theta$ is a contraction. Let us have $(y,z,y^\prime,z^\prime) \in (\H^2(\F,\P^0) \times \L^2_d(\F,\P^0,X(\cdot)))^2$ and let us denote by $(Y,Z)$ and $(Y^\prime,Z^\prime)$ the respective images under $\Theta$, and $L$, $L^\prime$ the corresponding orthogonal martingales. Let us set $\delta Y\coloneqq Y-Y^{\prime},$ $\delta y\coloneqq y-y^{\prime},$ and $\delta g\coloneqq g(y,z(v))-g(y^{\prime},z^{\prime}(v))$. Processes $\delta L, \delta Z$ and $\delta z$ are defined analogously.  We have
\begin{align*}
\delta Y_t^2=\delta Y_0^2 +\int_0^t 2\delta Y_{s-}\mathrm{d}\delta Y_s + \int_0^t \mathrm{d}[\delta Y]_s &=\delta Y_0^2+\int_0^t\int_V 2\delta Y_{s-} \delta g_s m^0(\mathrm{d}v,\mathrm{d}s) \\
&\quad +\int_0^t\int_V 2 \delta Y_{s-} \delta Z_s(v)\cdot X(\mathrm{d}v,\mathrm{d}s) +\int_0^t 2\delta Y_{s-} \mathrm{d}\delta L_s \\
&\quad +\int_0^t\int_V \lVert \delta Z_s(v)^\top \sigma(s,X_{\cdot \wedge s})\rVert^2 m^0(\mathrm{d}v,\mathrm{d}s)+\int_0^t\mathrm{d}[\delta L]_s. \\
\end{align*}
Consequently
\begin{align*}
\mathrm{e}^{\beta T}\delta Y_T^2 -  \delta Y_0^2 
&=\int_0^T \beta \mathrm{e}^{\beta s}\delta Y_s^2\mathrm{d}s  + \int_0^T\int_V 2\mathrm{e}^{\beta s}\delta Y_{s-}\delta g_s m^0(\mathrm{d}v,\mathrm{d}s) \\
&\quad +\int_0^t\int_V 2\mathrm{e}^{\beta s}\delta Y_{s-} \delta Z_s(v) \cdot X(\mathrm{d}v,\mathrm{d}s) +\int_0^t 2\mathrm{e}^{\beta s}\delta Y_{s-} \mathrm{d}\delta L_s \\
&\quad +\int_0^t\int_V \mathrm{e}^{\beta s}\lVert \delta Z_s(v)^\top \sigma(s,X_{\cdot \wedge s})\rVert^2 m^0(\mathrm{d}v,\mathrm{d}s)+\int_0^t \mathrm{e}^{\beta s}\mathrm{d}[\delta L]_s. 
\end{align*} 
It is straightforward to verify that both stochastic integrals are true martingales. Consequently, since $Y_T=Y^{\prime}_T=\zeta,$ we have
\begin{align*}
&\E^{\P^\smallfont{0}}\bigg[ \int_0^T \int_V  \mathrm{e}^{\beta s} \big(\beta \delta Y_s^2+\lVert \delta Z_s(v)^\top \sigma(s,X_{\cdot \wedge s}) \rVert^2 \big) m^0(\mathrm{d}v,\mathrm{d}s) \bigg] \\
&= \E^{\P^\smallfont{0}}\bigg[ -\delta Y_0^2 - \int_0^T \mathrm{e}^{\beta s}\mathrm{d}[\delta L]_s  -\int_0^T\int_V 2\mathrm{e}^{\beta s}\delta Y_{s-}\delta g_s m^0(\mathrm{d}v,\mathrm{d}s)\bigg] \\
& \leq \E^{\P^\smallfont{0}}\bigg[\int_0^T\int_V 2\mathrm{e}^{\beta s}\lvert \delta Y_{s-} \rvert \lvert \delta g_s \rvert m^0(\mathrm{d}v,\mathrm{d}s)\bigg] \\
& \leq \E^{\P^\smallfont{0}}\bigg[\int_0^T\int_V 2C \mathrm{e}^{\beta s}\lvert \delta Y_{s-} \rvert \big(\lvert \delta y_s \rvert + \lVert \delta z_s(v)^\top \sigma(s,X_{\cdot \wedge s}) \rVert \big) m^0(\mathrm{d}v,\mathrm{d}s)\bigg] \\
&\leq C \alpha \E^{\P^\smallfont{0}}\bigg[\int_0^T \mathrm{e}^{\beta s} \delta Y_{s-}^2 \mathrm{d}s + 2\frac{C}{\alpha}\int_0^T\int_V \mathrm{e}^{\beta s} \delta y_s^2 + \lVert \delta z_s(v)^\top \sigma(s,X_{\cdot \wedge s}) \rVert^2  m^0(\mathrm{d}v,\mathrm{d}s)\bigg],
\end{align*} for any $\alpha >0.$ Note that we can write $\delta Y$ instead of $\delta Y_-$ in the last expression. We conclude that
\begin{align*}
&\E^{\P^\smallfont{0}}\bigg[ \int_0^T \int_V  \mathrm{e}^{\beta s} \big((\beta -\alpha C ) \delta Y_s^2+\lVert \delta Z_s(v)^\top \sigma(s,X_{\cdot \wedge s}) \rVert^2 \big) m^0(\mathrm{d}v,\mathrm{d}s) \bigg] \\
&\leq 2\frac{C}{\alpha}\E^{\P^\smallfont{0}}\bigg[\int_0^T\int_V \mathrm{e}^{\beta s} \big(\delta y_s^2 + \lVert \delta z_s(v)^\top \sigma(s,X_{\cdot \wedge s}) \rVert^2 \big) m^0(\mathrm{d}v,\mathrm{d}s)\bigg].
\end{align*}
It suffices to take $\alpha>0$ such that $2\frac{C}{\alpha}<1$ and then $\beta$ such that $\beta -\alpha C=1$ to get a contraction. We conclude that there exists a unique solution $(Y,Z,L) \in \H^2(\F,\P^0) \times \L^2_d(\F,\P^0,X(\cdot)) \times \M^2_{X(\cdot)^\smallfont{\perp}}(\F,\P^0)$. The fact that $Y \in \S^2(\F,\P^0)$ as well as bound \eqref{bdBSDE} can be shown by standard techniques, see \emph{e.g.} \citeauthor*{bouchard2018unified} \cite{bouchard2018unified}.

\medskip
Let now $p \geq 2.$ Then, by the first part, there exists a unique solution $(Y,Z,L) \in \S^2(\F,\P^0) \times \L^2_d(\F,\P^0,X(\cdot)) \times \M^2_{X(\cdot)^\smallfont{\perp}}(\F,\P^0).$ Moreover, we have $(Y,Z,L) \in \S^p(\F,\P^0) \times \L^p_d(\F,\P^0,X(\cdot)) \times \M^p_{X(\cdot)^\smallfont{\perp}}(\F,\P^0)$ again by usual arguments. Uniqueness in $\S^p(\F,\P^0) \times \L^p_d(\F,\P^0,X(\cdot)) \times \M^p_{X(\cdot)^\smallfont{\perp}}(\F,\P^0)$ holds trivially.

\medskip Finally, if $p \in (1,2),$ one can once again use standard techniques, hence we omit the proof here. We refer to \cite[Section 4]{bouchard2018unified} for details.
\end{proof}

\begin{corollary} \label{Ex2bsde} Assume that $\F=\F^{X(\cdot)}.$ Let $\zeta \in \L^p(\Fc_T,\P^0)$ and let {\rm\Cref{lips}} hold. Then there is a unique solution $(Y,Z) \in \S^p(\F,\P^0) \times \L^p_d(\F,\P^0,X(\cdot))$ to the {\rm BSDE} 
\begin{equation}
Y_t=\zeta-\int_t^T \int_V g_{s}(Y,Z(v)) m^0(\mathrm{d}v,\mathrm{d}s) - \int_t^T \int_V Z_s(v) \cdot X(\mathrm{d}v,\mathrm{d}s),\; t \in [0,T],\; \P^0\text{\rm--a.s.}
\end{equation}
Moreover, there exists a constant $K>0$ such that
\begin{equation*}
\| Y \|_{\S^\smallfont{p}(\P^0)}^p + \| Z \|^p_{\L^\smallfont{p}_\smallfont{d}(\F,\P^0,X(\cdot))} \leq K \E^{\P^\smallfont{0}}\bigg( \lvert \zeta \rvert^p + \bigg[\bigg( \int_0^T \lvert g_s(0,0) \rvert \mathrm{d}s\bigg)^p\bigg] \bigg).
\end{equation*}
\end{corollary}
\begin{proof}
This follows from \Cref{martRep} and \Cref{Exbsde}.
\end{proof}

\section{Proofs for Section \ref{prelim}} \label{proofApp}

\begin{proof}[Proof of \Cref{pf2}] Let us fix $\omega \in \Omega$ and let us first assume that $g(\omega)$ is simple. That is, it is of the form
\[ 
g_s(f)(\omega)=\sum_{i=1}^N g_i \mathbf{1}_{\{ (s,f) \in B_\smallfont{i} \}},\; (s,f) \in   [0,T] \times F, \] for some $g_i \in \R$ and $B_i \in \Bc([0,T]) \otimes \Bc(F)$ for $i \in \{1,\ldots,N\},$ $N \in \N^\star.$ For such $\omega \in \Omega,$ we have
\begin{align*}
\int_0^T \int_E  g_s(Z_s(e))  m(\mathrm{d}e,\mathrm{d}s) = \int_0^T \int_E \sum_{i=1}^N g_i \mathbf{1}_{\{ (s,Z_\smallfont{s}(e)) \in B_\smallfont{i} \}}  m(\mathrm{d}e,\mathrm{d}s) &= \sum_{i=1}^N g_i \int_0^T \int_E  \mathbf{1}_{\{ ( s,Z_\smallfont{s}(e)) \in B_\smallfont{i} \}}  m(\mathrm{d}e,\mathrm{d}s) \\
&=\sum_{i=1}^N g_i  n( B_i)=\int_0^T \int_F g_s(f) n(\mathrm{d}f,\mathrm{d}s).
\end{align*}

Let now $g(\omega)$ be a general positive $\Bc([0,T]) \otimes \Bc(F)$-measurable function. We can find a sequence of simple functions $g^j :[0,T] \times F \longrightarrow \R$, $j\in\N^\star$, such that $g_s^j(f)  \nearrow  g_s(f)(\omega), \; (s, f) \in [0,T] \times F.$ We have
\begin{align*}
\int_0^T \int_E  g_s^j(Z_s(e)) m(\mathrm{d}e,\mathrm{d}s) =\int_0^T \int_F g_s^j(f) n(\mathrm{d}f,\mathrm{d}s),\; j\in \N^\star.
\end{align*} 
Hence, by passing to the limit and using the monotone convergence theorem
\begin{align} \label{eqposit}
\int_0^T \int_E  g_s(Z_s(e))  m(\mathrm{d}e,\mathrm{d}s) =\int_0^T \int_F g_s(f)  n(\mathrm{d}f,\mathrm{d}s).
\end{align} 
The first statement follows immediately. As for the second one, we can write for general $g$ using \eqref{eqposit} that
\begin{align*}
\int_0^t \int_{Z_\smallfont{s}^{\smallfont{-}\smallfont{1}}(A)}  g_s(Z_s(e)) m(\mathrm{d}e,\mathrm{d}s)&=\int_0^T \int_E  g_s(Z_s(e))\mathbf{1}_{\{ (s,Z_\smallfont{s}(e)) \in (0,t] \times A \}}  m(\mathrm{d}e,\mathrm{d}s)\\
&=\int_0^T \int_E  g_s^+(Z_s(e))\mathbf{1}_{\{ (s,Z_\smallfont{s}(e)) \in (0,t] \times A \}}  m(\mathrm{d}e,\mathrm{d}s) \\
&\quad- \int_0^T \int_E  g_s^-(Z_s(e))\mathbf{1}_{\{ (s,Z_\smallfont{s}(e)) \in (0,t] \times A \}}  m(\mathrm{d}e,\mathrm{d}s)\\
&=\int_0^T \int_A  g_s^+(f)\mathbf{1}_{\{ (s,f) \in (0,t] \times A \}} n(\mathrm{d}f,\mathrm{d}s)-\int_0^T \int_A  g_s^-(f)\mathbf{1}_{\{ (s,f) \in (0,t] \times A \}} n(\mathrm{d}f,\mathrm{d}s)\\
&=\int_0^t \int_A  g_s(f) n(\mathrm{d}f,\mathrm{d}s),
\end{align*} 
which concludes the second part. Further, by the above, we have that 
\[
 \int_0^T \int_E \lvert g_s(Z_s(e)) \rvert^2 m(\mathrm{d}e,\mathrm{d}s)<\infty,\; \text{\rm if and only if}\;  \int_0^T \int_F \lvert g_s(f) \rvert^2 n(\mathrm{d}f,\mathrm{d}s)<\infty,
 \] which immediately yields that $(g \circ Z) \in \L^2_{\rm loc}(\F,\P,M(\cdot))$ is equivalent to $g \in \L^2_{\rm loc}(\F,\P,N(\cdot)).$ 
 
\medskip For the last claim, let us first assume that $g$ is an elementary adapted process. That is, it is of the form
\[
g_s(f)(\omega)=\sum_{i=1}^N g_i(\omega) \mathbf{1}_{\{ s \in [s_\smallfont{i},t_\smallfont{i}) \}} \mathbf{1}_{\{ f \in B_\smallfont{i} \}},\; (\omega, s,f) \in \Omega \times [0,T] \times F, \] for some $N \in \N^\star,$ $g_i \in \L^2(\Fc_{s_i}, \P),$ $0\leq s_1 \leq t_1 \leq \ldots \leq s_n \leq t_n \leq T$ and $B_i \in \Bc(F)$ for $i \in \{1,\ldots,N\}.$ We have
\begin{align*}
\int_0^t \int_{Z_\smallfont{s}^{\smallfont{-}\smallfont{1}}(A)}  g_s(Z_s(e)) M(\mathrm{d}e,\mathrm{d}s) &= \int_0^T \int_E \sum_{i=1}^N g_i \mathbf{1}_{\{ s \in [s_\smallfont{i} \wedge t,t_\smallfont{i} \wedge t) \}} \mathbf{1}_{\{ Z_\smallfont{s}(e) \in B_\smallfont{i} \cap A \}} M(\mathrm{d}e,\mathrm{d}s) \\
&= \sum_{i=1}^N g_i \int_{s_\smallfont{i} \wedge t}^{t_\smallfont{i} \wedge t} \int_E  \mathbf{1}_{\{ Z_\smallfont{s}(e) \in B_\smallfont{i} \cap A \}} M(\mathrm{d}e,\mathrm{d}s) \\
&=\sum_{i=1}^N g_i \int_{s_\smallfont{i}\wedge t}^{t_\smallfont{i} \wedge t} \int_E  \mathrm{d}N_s(B_i \cap A)=\int_0^t \int_A g_s(f)  N(\mathrm{d}f,\mathrm{d}s).
\end{align*} 

For $g \in \L^2(\F,\P,N(\cdot))$, there exists a sequence of elementary processes $(g^j)_{j\in\N^\smallfont{\star}}$ such that $g^j \longrightarrow g$ in  $\L^2(\F,\P,N(\cdot)).$ By the previous parts, $g^j \circ Z \longrightarrow g \circ Z$ in $\L^2(\F,\P,M(\cdot)).$ The equality
\[
\int_0^t \int_{Z_\smallfont{s}^{\smallfont{-}\smallfont{1}}(A)}  g_s^j(Z_s(e)) M(\mathrm{d}e,\mathrm{d}s)=\int_0^t \int_A  g_s^j(f) N(\mathrm{d}f,\mathrm{d}s),\; j \in \N^\star, \] then once again yields 
\[
\int_0^t \int_{Z_\smallfont{s}^{\smallfont{-}\smallfont{1}}(A)}  g_s(Z_s(e)) M(\mathrm{d}e,\mathrm{d}s)=\int_0^t \int_A  g_s(f) N(\mathrm{d}f,\mathrm{d}s), \] by passing to the limit in $\L^2(\Fc_t,\P^0)$ on both sides.

\medskip Finally, if $g \in \L^2_{\rm loc}(\F,\P,N(\cdot)),$ we use localisation arguments together with the previous results to conclude.
\end{proof}

\begin{proof}[Proof of \Cref{FXfiltr}]  Let us first note that for any system $\{ G_j:j \in \N^\star \} \subset \Bc(E)$ and for any $k \in \N^\star,\; t \in [0,T]$
\[ 
\sigma \big( M_s(G_j) : s \in [0,t], \; j\in\{1,\ldots,k\} \big)\vee \sigma(\Nc)=\bigvee_{n \in \N^\smallfont{\star}} \sigma \big( M_{it/2^\smallfont{n}}(G_j) : i\in\{0,\ldots,2^n\},\; j\in\{1,\ldots,k\} \big)\vee \sigma(\Nc),\]
by continuity of paths. Hence, it is sufficient to find a system of open sets such that
\[
\sigma \big( M_s(B) : s \in [0,t],\;  B\in\Bc(E) \big) \vee \sigma(\Nc)=\bigvee_{k \in \N^\smallfont{\star}} \sigma \big( M_s(G_j) : s \in [0,t],\; j\in\{1,\ldots,k\} \big) \vee \sigma(\Nc),\; t \in [0,T]. 
\]
Since $E$ is assumed to be Polish, there exists a countable topological base $\{ H_j:j \in \N^\star \}.$ We claim that it is sufficient to define $\{ G_j:j \in \N^\star \}$ as the set of all finite unions of $\{ H_j:j \in \N^\star \}.$ Indeed, if $G \in \Bc(E)$ is open, there exists a sequence $(H_k)_{k \in \N^\smallfont{\star}} \subset \{ H_j:j \in \N^\star \}$ such that $G= \cup_{k=1}^\infty H_k.$ It follows that for every $t \in [0,T]$
\begin{align*}
\E^{\P}\bigg[ \bigg(M_t(G)-M_t\bigg(\bigcup_{k=1}^n H_k\bigg) \bigg)^2\bigg]&=\E^{\P}\bigg[\bigg[ M(G)-M\bigg(\bigcup_{k=1}^n H_k\bigg)\bigg ]_t\bigg]\\
&=\E^{\P} \bigg[ [ M(G) ]_t - 2\bigg[ M(G),M\bigg(\bigcup_{k=1}^n H_k\bigg)\bigg ]_t + \bigg[ M\bigg(\bigcup_{k=1}^n H_k\bigg) \bigg]_t  \bigg] \\
&= \E^{\P} \bigg[ [ M(G) ]_t - \bigg[ M\bigg(\bigcup_{k=1}^n H_k\bigg)\bigg ]_t  \bigg] \\
&=\E^{\P}  \bigg[ m([0,t] \times G) - m\bigg([0,t] \times \bigcup_{k=1}^n H_k\bigg) \bigg] \searrow 0, \; \text{as}\; n\longrightarrow\infty.
\end{align*}
Hence, up to choosing a subsequence 
\[
M_t(G)=\lim_{n \rightarrow \infty} M_t\bigg(\bigcup_{k=1}^n H_k\bigg),\; \P\text{\rm--a.s.} 
\]

Let us define a measure $\nu$ on $[0,T] \times E$ by $\nu([0,t] \times A)=\E^{\P}m([0,t]\times A),$ $t \in [0,T], A \in \Bc(E).$ It is straightforward to verify that $\nu$ is indeed a well-defined finite measure. Let now $A \in \Bc(E)$. Since $E$ is Polish, $\nu$ is a regular measure. Hence, $\nu(A)=\inf\{ \nu(G) : A \subset G,\; G \in\Bc(E) \; \text{open} \}.$ In particular, there is a sequence of open sets $(\widetilde{G}_n)_{n\in\N^\smallfont{\star}}$ such that 
\[
\nu\big(\widetilde{G}_n\big) \searrow \nu(A),\; \text{as}\; n\longrightarrow\infty.  
\]
Similarly as before, one can show that up to a subsequence $M_t(A)=\lim_{n \rightarrow \infty} M_t\big( \widetilde{G}_n\big),\;  \P^0\text{\rm--a.s.}$ This concludes the proof.
\end{proof}

\section{Convergence of measures}
\begin{lemma} \label{fatou} Let $X$ be a Polish space and let $(\mu_n,\mu) \in \Pc(X)\times\Pc(X)$, $n \in \N^\star,$ be such that $\mu_n \longrightarrow \mu$ weakly. Let $f: X \longrightarrow \R$ be upper-semicontinuous and such that $f^+$ is asymptotically uniformly integrable w.r.t. $(\mu_n)_{n \in \N^\smallfont{\star}}.$ That is
\[ \lim_{K \rightarrow \infty} \limsup_{n \rightarrow \infty} \int f^+(x) \mathbf{1}_{\{f^+(x)>K\}}\mu_n(\mathrm{d}x)=0.  \] 
Then 
\[ \limsup_{n \rightarrow \infty} \int f(x)  \mu_n(\mathrm{d}x) \leq \int f(x) \mu(\mathrm{d}x). \]
Consequently, if $f: X \rightarrow \R$ is continuous and asymptotically uniformly integrable w.r.t. $(\mu_n)_{n \in \N^\smallfont{\star}},$ then 
\[ \lim_{n \rightarrow \infty} \int f(x)  \mu_n(\mathrm{d}x) = \int f(x) \mu(\mathrm{d}x). \]
\end{lemma}
\begin{proof} See \citeauthor*{feinberg2020weakfatou} \cite[Theorem 2.4 and Corollary 2.8]{feinberg2020weakfatou}
\end{proof}
\begin{remark} \label{rem:UI_convergence} If \[\sup_{n \in \N^\smallfont{\star}} \int \lvert f(x) \rvert^{1+\varepsilon} \mu_n(\mathrm{d}x) < \infty\] for some $\varepsilon>0,$ then $f$ is asymptotically uniformly integrable w.r.t. $(\mu_n)_{n \in \N^\smallfont{\star}}.$ Indeed, we have
\begin{align*} \lim_{K \rightarrow \infty} \limsup_{n \rightarrow \infty} \int \lvert f(x) \rvert \mathbf{1}_{\{\lvert f(x) \rvert >K\}}\mu_n(\mathrm{d}x) &\leq  \lim_{K \rightarrow \infty} \limsup_{n \rightarrow \infty} \int \lvert f(x) \rvert \frac{\lvert f(x) \rvert^\varepsilon}{K^\varepsilon} \mathbf{1}_{\{\lvert f(x) \rvert >K\}}\mu_n(\mathrm{d}x) \\
&\leq \lim_{K \rightarrow \infty} \limsup_{n \rightarrow \infty} \int \lvert f(x) \rvert \frac{\lvert f(x) \rvert^\varepsilon}{K^\varepsilon} \mu_n(\mathrm{d}x) \\
&\leq \lim_{K \rightarrow \infty} \bigg( \sup_{n \in \N^\smallfont{\star}} \int \lvert f(x) \rvert^{1+\varepsilon}\mu_n(\mathrm{d}x) \bigg) \frac{1}{K^\varepsilon} =0.
\end{align*}
\end{remark}

Let us denote by $\Cc^2_b(\R^{d+1},\R)$ the set of all bounded twice continuously differentiable functions with bounded derivatives up to order 2 from $\R^{d+1}$ to $\R$ and further define for $\varepsilon>0$ and $M \geq 0$ the stopping times
\begin{align*}
\sigma^{M,\varepsilon}&\coloneqq
\inf \bigg\lbrace t \in [0,T] : \max \bigg\lbrace \lVert (X_t,\Yc_t) \rVert, \int_0^t \int_{\R^l} \lVert z \rVert^{2+\varepsilon} m_s^Z(\mathrm{d}z)\mathrm{d}s \bigg\rbrace \geq M \text{ or } t=T \bigg\rbrace.
\end{align*}

\begin{lemma}\label{tightness} Let $K \subset \big\{ \P\big\vert_{\Fc_\smallfont{T}^c} : \P \in \Rc \}.$ Then $K$ is tight if there exists $\varepsilon>0$ such that the following conditions hold
\begin{enumerate}[label = (\roman*)]
\item \label{T1} we have \begin{align*}
\lim_{M \rightarrow \infty} \sup_{\P \in K} \P \left[ \sigma^{M,\varepsilon}<T \right]=0;
\end{align*}
\item \label{T2}for any $M \in [0,\infty)$ and any $\phi \in \Cc^2_b(\R^{d+1},\R)$ with $\phi \geq 0,$ there exists a constant $A_{\phi,M} \geq 0$ such that the process
\[
\phi(X_{\cdot \wedge \sigma^\smallfont{M,\varepsilon}},\Yc_{\cdot \wedge \sigma^\smallfont{M,\varepsilon}})+A_{\phi,M} \int_0^{\cdot \wedge \sigma^\smallfont{M,\varepsilon}}\int_{\R^\smallfont{d}} \big(1+\lVert z \rVert^2\big) m_s^Z(\mathrm{d}z)\mathrm{d}s,
\]
is an $(\F^c,\P)$--sub-martingale for every $\P \in K.$ Moreover, if $\psi$ is a translate of $\phi$ $($\emph{i.e.} $\psi(\cdot)=\phi(\cdot - a)$ for some $a \in \R^{d+1})$, we can take $A_{\psi,M}=A_{\phi,M}.$
\end{enumerate}
\end{lemma}
\begin{proof}
Follows from the proof of \cite[Theorem A.8]{haussmann1990existence}. 
\end{proof}

\begin{lemma} \label{lem:continuity_of_phi} Let $\phi : [0,T] \times \Cc([0,T],\R^{d+1}) \times \R^d \longrightarrow \R$ be bounded and measurable. Let $M>0$ be fixed and assume further that $\phi(t,\cdot,\cdot)$ is continuous for $\lambda$--almost every $t \in [0,T].$ Then the map
\[ (x,m^Z) \longmapsto \int_0^T \int_{\R^d} \mathbf{1}_{\{ \|z \|<M \}} \phi(t,x_{\cdot \wedge t},z) m^{Z}_t(\mathrm{d}z)\mathrm{d}t,\] is continuous on the set \[ \bigg\{ (x,m^Z) \in \Cc([0,T],\R^{d+1})\times \Mc([0,T],\R) : \int_0^T \int_{\R^d} \mathbf{1}_{\{ \|z \|=M \}} m^Z_t(\mathrm{d}z)\mathrm{d}t=0 \bigg\}. \]
\end{lemma}

\begin{proof} Follows from \citeauthor*{haussmann1990existence} \cite[Lemma A.3, $(ii)$]{haussmann1990existence}. Note that while the authors assume that $\phi(t,\cdot,\cdot)$ is uniformly continuous, it is, in fact, not needed. Indeed, if $ \lim_{n \rightarrow \infty }x^n= x^\infty$ in $\Cc([0,T],\R^{d+1}),$ then we can restrict $\phi(t,\cdot,\cdot)$ to the set $K \coloneqq \{ (x^n,z) \in \Cc([0,T],\R^{d+1})\times \R^d : \| z \|\leq M,\; n \in \N^\star \cup \{ \infty \} \}.$ Then, $K$ is clearly a compact set in $\Cc([0,T],\R^{d+1})\times \R^d$ and thus $\phi(t,\cdot,\cdot)$ is uniformly continuous on $K.$ Applying \cite[Lemma A.3, $(ii)$]{haussmann1990existence}  to $\phi$ restricted to $[0,T]\times K$ then yields 
\[ \int_0^T \int_{\R^d} \mathbf{1}_{\{ \|z \|<M \}} \phi(t,x^n_{\cdot \wedge t},z) m^{Z,n}_t(\mathrm{d}z)\mathrm{d}t \longrightarrow \int_0^T \int_{\R^d} \mathbf{1}_{\{ \|z \|<M \}} \phi(t,x^\infty_{\cdot \wedge t},z) m^{Z,\infty}_t(\mathrm{d}z)\mathrm{d}t \] for any sequence $m^{Z,n}$ and $m^{Z,\infty}$ such that $\lim_{n \rightarrow \infty} m^{Z,n}=m^{Z,\infty}$ in $\Mc([0,T],\R^d)$ and \[\int_0^T \int_{\R^d} \mathbf{1}_{\{ \|z \|=M \}} m^{Z,\infty}_t(\mathrm{d}z)\mathrm{d}t=0.\] This gives the desired continuity.
\end{proof}

\section{Supportive results for Theorem \ref{minim}} \label{subsec:proof_of_THM}

In this section we present the technical results supporting the statement from the previous section. We also provide some complementary results. We assume throughout this whole subsection that {\rm Assumptions \ref{discF}} and {\rm \ref{assPrincipal}} hold. As such, we will not explicitly mention them in the statements of the results.

\medskip Note that, while the probability space in \Cref{sec:principal} supported a suitable martingale measure as the driving noise, in this section we consider a smaller space. The following lemma shows that we can always extend the probability space to obtain the dynamics of our processes involving a driving martingale measure.

\begin{lemma} \label{ext1} Let $(X,\Yc,m^Z)$ be an $\F$-adapted process with values in $\Omega^c$ defined on some filtered probability space $(\Omega,\Fc,\F,\P)$ satisfying that the process \[ M^\phi_t=\phi(X_t,\Yc_t)-\int_0^t \int_{\R^d} \mathscr{L}\phi (s,X_{\cdot \wedge s},\Yc_s,z) m^Z_s(\mathrm{d}z)\mathrm{d}s,\; t \in [0,T], \] is an $(\F,\P)$-martingale for every $\phi \in \Cc_c^2(\R^{d+1},\R).$ Then, there exists an extension of $(\Omega,\Fc,\F,\P),$ say $(\tilde{\Omega},\tilde{\Fc},\tilde{\F},\tilde{\P}),$ supporting a $k$-dimensional $(\tilde{\F},\tilde{\P})$-martingale measure $\widetilde{M}^0$ on $\Bc(V)$ with intensity $m^0$ satisfying
\begin{align*} 
X_t&=x_0+\int_0^t \int_V \sigma(s,X_{\cdot \wedge s}) \lambda\big(s,X_{\cdot \wedge s},a(s,X_{\cdot \wedge s},Z_s(v))\big) m^0(\mathrm{d}v,\mathrm{d}s)+\int_0^t\int_V \sigma(s,X_{\cdot \wedge s}) \widetilde{M}^0(\mathrm{d}v,\mathrm{d}s), \; t\in[0,T],\; \bar\P\text{\rm--a.s.},\\
\Yc_t &= \int_0^t \int_{V} \mathrm{e}^{-\int_\smallfont{0}^\smallfont{s} k(u,X_{\smallfont{\cdot} \smallfont{\wedge} \smallfont{u}}) \mathrm{d}u} f\big(s,X_{\cdot \wedge s},a(s,X_{\cdot \wedge s},Z_s(v))\big)m^0(\mathrm{d}v,\mathrm{d}s) \\
&\quad+\int_0^t \int_{V} \mathrm{e}^{-\int_\smallfont{0}^\smallfont{s} k(u,X_{\smallfont{\cdot} \smallfont{\wedge} \smallfont{u}}) \mathrm{d}u} Z_s(v)^\top \sigma(s,X_{\cdot \wedge s}) \widetilde{M}^0(\mathrm{d}v,\mathrm{d}s),\; t \in [0,T],\; \tilde{\P}\text{\rm-a.s.},
\end{align*} where $Z$ is determined by $m^Z=Z_\# m^0.$ If, moreover, $L$ is an $(\F,\P)$-martingale such that
$[ M^\phi,L ]^{\F,\P}\equiv 0,$ then the extension can be chosen such that the martingale measure $\widetilde{M}^0$ satisfies $[ \widetilde{M}^0(\cdot),L]^{\tilde{\F},\tilde{\P}} \equiv 0.$
\end{lemma}
\begin{proof}
We use \Cref{pf} to find $Z$ such that $m^Z=Z_\#m^0.$ We then make use of \cite[Theorem IV-2]{el1990martingale} to get an extension supporting $M^0$. The last statement follows from the construction of $\widetilde{M}^0$ (see \cite[Theorem III-7]{el1990martingale}), since $\widetilde{M}^0$ is obtained using an auxiliary independent martingale measure.
\end{proof}

We have the following bounds on the moments of the process $(X,\Yc)$ under $\P \in \Rc.$

\begin{lemma} \label{bdMOMstatement} For every $t \in [0,T]$ and $p\geq 1$ there exists a constant $c>0$  such that for every $\P \in \Rc$ we have
\begin{align} \label{bdMOM}
\E^{\P} \bigg[\sup_{s \in [0,t]} \lVert (X_s,\Yc_s) \rVert^p\bigg] \leq c \bigg(1+ \lVert x_0 \rVert^p + \E^\P\bigg[ \bigg( \int_0^t \int_{\R^\smallfont{d}}  \lVert z^\top \sigma(t,X_{\cdot \wedge t}) \rVert^2  m^Z(\mathrm{d}z,\mathrm{d}s)\bigg)^\frac{p}{2}\bigg] \bigg).
\end{align}
\end{lemma}

\begin{proof}
Let us assume that the right-hand side is finite since the statement holds trivially otherwise. Let us define for $N\in\N^\star$ the stopping times $\tau_N\coloneqq \inf \lbrace t \in [0,T] : \lVert (X_t,\Yc_t) \rVert \geq N \text{ or } t=T \rbrace.$ Let us denote by $c$ a generic positive constant depending on $t \in [0,T]$ and $p \geq 1,$ which may change from line to line. Since $a$ and $b$ satisfy \eqref{boundFA}, we get for every $\P \in \Rc,$ $t \in [0,T]$ and $p\geq 1$ using the Burkholder--Davis--Gundy's inequality
\begin{align*}
\E^{\P}\bigg[ \sup_{s \in [0,t\wedge \tau_\smallfont{N}]}\lVert (X_s,\Yc_s) \rVert^p\bigg] &
 \leq c\bigg( \lVert x_0 \rVert^p + \E^{\P} \bigg[\bigg( \int_0^{t \wedge \tau_\smallfont{N}} \int_{\R^\smallfont{d}} \lVert  b(s,X_{\cdot \wedge s},z) \rVert  m^Z_s(\mathrm{d}z)\mathrm{d}s \bigg)^p\bigg]\bigg)\\
 &\quad+c\E^\P\bigg[\bigg(\int_0^{t \wedge \tau_\smallfont{N}} \int_{\R^\smallfont{d}}  \lVert a(s,X_{\cdot \wedge s},z) \rVert m^Z_s(\mathrm{d}z)\mathrm{d}s\bigg)^\frac{p}{2} \bigg]  \\
& \leq c\bigg(1+ \lVert x_0 \rVert^p + \E^{\P} \bigg[\bigg( \int_0^{t \wedge \tau_\smallfont{N}} \int_{\R^\smallfont{d}} \lVert X_{\cdot \wedge s} \rVert + \lVert z^\top \sigma(s,X_{\cdot \wedge s}) \rVert  m^Z_s(\mathrm{d}z)\mathrm{d}s \bigg)^p\bigg]\bigg) \\
&\quad + c \E^{\P} \bigg[\bigg( \int_0^{t \wedge \tau_\smallfont{N}} \int_{\R^\smallfont{d}} \lVert X_{\cdot \wedge s} \rVert + \lVert z^\top \sigma(s,X_{\cdot \wedge s}) \rVert^2  m^Z_s(\mathrm{d}z)\mathrm{d}s \bigg)^\frac{p}{2}\bigg] \\
&\leq c \bigg(1+ \lVert x_0 \rVert^p + \E^\P \bigg[\bigg( \int_0^t \int_{\R^\smallfont{d}}  \lVert z^\top \sigma(s,X_{\cdot \wedge s}) \rVert^{2}  m^Z_s(\mathrm{d}z)\mathrm{d}s \bigg)^\frac{p}{2}\bigg]\bigg) \\
&\quad + c \int_0^t  \E^\P\bigg[  \sup_{u \in [0,s \wedge \tau_\smallfont{N}]}\lVert (X_u,\Yc_u) \rVert^{p}\bigg] \mathrm{d}s .
\end{align*}
By Gronwall's lemma, we get
\begin{align*}
\E^{\P}\bigg[ \sup_{s \in [0,t \wedge \tau_\smallfont{N}]} \lVert (X_s,\Yc_s) \rVert^p\bigg] \leq c \bigg(1+ \lVert x_0 \rVert^p + \E^\P\bigg[ \bigg(  \int_0^{t} \int_{\R^\smallfont{d}}  \lVert z^\top \sigma(s,X_{\cdot \wedge s}) \rVert^{2}  m^Z_s(\mathrm{d}z)\mathrm{d}s\bigg)^\frac{p}{2}\bigg] \bigg),\; \P \in \Rc.
\end{align*} The result then follows from monotone convergence theorem by letting $N \longrightarrow \infty.$
\end{proof}

Since the orthogonality condition in \Cref{rule}.\ref{rule5} may be difficult to deal with, we reformulate it in the following way.

\begin{lemma} \label{OGmart} Assume that $\P \in \Pc(\Omega^c \times \R)$ satisfies all conditions of {\rm\Cref{rule}} with $1/q+1/q^\prime\leq 1$ except $(v)$. Then, $\P \in \Rc_{q,q^\smallfont{\prime}}$ if and only if 
\[
\E^{\P}\big[h_s U ( M^\phi_t - M^\phi_s )\big]=0,
\]
holds for all $\phi \in \Cc_c^2(\R^{d+1},\R),$ $s \leq t,$ and $h_s : \Omega^c \longrightarrow \R$, $\Fc_s^c$-measurable and bounded. Moreover, it suffices to verify the property with continuous functions $h_s.$
\end{lemma}
\begin{proof}
Since $1/q+1/q^{\prime}\leq 1,$ it is easy to verify that that
\begin{align*}
&\big[ M^\phi_\cdot, \E^{\P}[ U \vert \Fc_\cdot^c ] \big]^{\F^\smallfont{c},\P} \equiv 0, \\
 \iff& M^\phi_\cdot \E^{\P}[ U \vert \Fc_\cdot^c ] \text{ is an } (\F^c,\P)\text{-martingale}, \\
 \iff& \forall s \in[0, t];\; \forall h_s \;\text{which are}\; \Fc_s^c\text{-measurable and bounded, } \E^{\P}\big[h_s\big(M^\phi_t \E^{\P}[ U \vert \Fc_t^c]-M^\phi_s \E^{\P}[ U \vert \Fc_s^c ] \big)\big]=0.
\end{align*}
Moreover
\begin{align*}
\E^{\P}\big[h_s\big(M^\phi_t \E^{\P}[ U \vert \Fc_t^c ]-M^\phi_s \E^{\P}[ U \vert \Fc_s^c ] \big)\big]=0 &\iff \E^{\P}\big[h_s M^\phi_t \E^{\P}[ U \vert \Fc_t^c ]\big]=\E^{\P}\big[h_s M^\phi_s \E^{\P}[ U \vert \Fc_s^c ]\big] \\
&\iff \E^{\P}[h_s M^\phi_t U] =\E^{\P}[h_s M^\phi_s U] \\
&\iff \E^{\P}\big[h_s U ( M^\phi_t - M^\phi_s )\big]=0.
\end{align*}
The proof of the last statement is straightforward. Since we work on Polish spaces, it is sufficient to consider those functions $h_s$ that are continuous.
\end{proof}

Because we work on a different probability space, one might wonder whether it is beneficial for the principal to consider larger filtrations. The following lemma shows that this is not the case. That is, the principal cannot profit from taking a filtration larger than the canonical one $\F^c.$ 

\begin{lemma} \label{largerfiltr} Let $(\tilde{\Omega},\tilde{\Fc}, \tilde{\F},\tilde{\P})$ be a filtered probability space with a $(d+1)$-dimensional adapted continuous process $(\tilde{X},\tilde{\Yc}),$ $\Mc([0,T],\R^d)$-valued adapted random variable $\tilde{m}^Z$ and $\tilde{\Fc}_T$-measurable random variable $\tilde{U}$ satisfying
\begin{enumerate}
\item[$(i)$]  there exist $q>1$ and $q^\prime>1$ such that 
\[
\E^{\tilde{\P}}\bigg[\bigg(\int_0^T\int_{\R^\smallfont{d}} \lVert z^\top \sigma(t,X_{\cdot \wedge t}) \rVert^2 \tilde{m}^Z_t(\mathrm{d}z)\mathrm{d}t\bigg)^\frac{q}{2}\bigg] < \infty, \; {\rm and}\; \E^{\tilde{\P}}\big[\lvert \tilde{U}\rvert^{q^\smallfont{\prime}}\big] < \infty;
\]
\item[$(ii)$] the process 
\[
\tilde{M}^\phi_t\coloneqq\phi(\tilde{X}_t,\tilde{\Yc}_t)-\int_0^t \int_{\R^\smallfont{d}} \mathscr{L}\phi (s,\tilde{X}_{\cdot \wedge s},\tilde{\Yc}_s,z) \tilde{m}^Z_s(\mathrm{d}z)\mathrm{d}s,\; t \in [0,T], 
\] is an $(\tilde{\F},\tilde{\P})$-martingale for every $\phi \in \Cc_c^2(\R^{d+1},\R);$
\item[$(iii)$]  $(\tilde{X}_0,\tilde{\Yc}_0,\tilde{m}^Z_0)=(x_0,0,\delta_0),\;\tilde{\P}${\rm--a.s.}$;$
\item[$(iv)$] for any $i\in I$, we have
\begin{equation*}
\E^{\tilde{\P}} \big[h^i(\tilde{X}_{\cdot \wedge T},\tilde{\Yc}_t,\tilde{U})\big] \leq 0;
\end{equation*}
\item[$(v)$]  $\forall \phi \in \Cc_c^2(\R^{d+1},\R)$, we have
\[ 
\big[ \tilde{M}^\phi_\cdot, \E^{\P}[ \tilde{U} \vert \tilde{\Fc}_\cdot ] \big]^{\tilde{\F},\tilde{\P}} \equiv 0,\; \text{\rm and}\; \E^{\tilde{\P}}[ \tilde{U}] \geq r_0. 
\]
\end{enumerate}
Then there exists $\P \in \Pc_{q,q^\smallfont{\prime}}$ and an $\Fc_T^c$-measurable random variable $\hat{U}$ such that $U=\hat{U}$, $\P$--{\rm a.s.} and
\[ 
\E^{\tilde{\P}} \big[F(\tilde{X}_{\cdot \wedge T},\tilde{\Yc}_T,\tilde{U}) \big]\leq \E^\P \big[F(X_{\cdot \wedge T},\Yc_T,U)\big]. 
\] 
\end{lemma}

\begin{proof}
Let us define \[\bar{U}\coloneqq\E^{\tilde{\P}}\big[ \tilde{U} \big\vert \Fc_T^{\tilde{X},\tilde{\Yc},\tilde{m}^Z} \big] \] and let us note that, since $F$ is concave in the last entry, Jensen's inequality yields

\begin{align*}
\E^{\tilde{\P}}\big[F(\tilde{X}_{\cdot \wedge T},\tilde{\Yc}_T,\tilde{U})\big]&\leq \E^{\tilde{\P}} \Big[F\big(\tilde{X}_{\cdot \wedge T},\tilde{\Yc}_T,\E^{\tilde{\P}}[ \tilde{U} \vert \Fc_T^{\tilde{X},\tilde{\Yc},\tilde{m}^Z} ]\big)\Big]=\E^{\tilde{\P}} \big[F(\tilde{X}_{\cdot \wedge T},\tilde{\Yc}_T,\bar{U})\big].
\end{align*} 
Similarly, by convexity, for any $i\in\{1,\dots,k\}$
\[ 
0 \geq  \E^\P \big[h^i(\tilde{X}_{\cdot \wedge T},\tilde{\Yc}_t,\tilde{U})\big] \geq \E^\P \big[h^i(\tilde{X}_{\cdot \wedge T},\tilde{\Yc}_t,\bar{U})\big], \]
and $\E^{\tilde{\P}}[\bar{U}]=\E^{\tilde{\P}}[\tilde{U}] \geq r_0.$ It remains to verify that $\bar{U}$ defines an orthogonal martingale.
Let us first assume that $1/q+1/q^\prime \leq 1.$ In this case, for $\phi \in \Cc_b^2(\R^{d+1},\R)$ the condition \[ \big[ \tilde{M}^\phi_\cdot, \E^{\tilde{\P}}[ \tilde{U} \vert \tilde{\Fc}_\cdot ] \big]^{\tilde{\F},\tilde{\P}} \equiv 0,\] is by \Cref{OGmart} equivalent to $\E^{\tilde{\P}}[\tilde{h}_s \tilde{M}^\phi_t\tilde{U}] =\E^{\tilde{\P}}[\tilde{h}_s \tilde{M}^\phi_s  \tilde{U}],$ for every $s \in[0, t]$ and $\tilde{h}_s$ which is $\tilde{\Fc}_s$-measurable and bounded. In particular, this implies that
\[\E^{\tilde{\P}}\big[\bar{h}_s \tilde{M}^\phi_t \tilde{U}\big]=\E^{\tilde{\P}}\big[\bar{h}_s \tilde{M}^\phi_s \tilde{U}\big].\] for every $s \in[0, t]$ and $\bar{h}_s$ being $\Fc^{\tilde{X},\tilde{\Yc},\tilde{m}^\smallfont{Z}}_s$-measurable and bounded. This in turn by the tower property of conditional expectation yields
\[\E^{\tilde{\P}}\big[\bar{h}_s \tilde{M}^\phi_t \bar{U}\big]=\E^{\tilde{\P}}\big[\bar{h}_s \tilde{M}^\phi_s \bar{U}\big],\] for every $s \in[0, t]$ and $\bar{h}_s$ that is $\Fc^{\tilde{X},\tilde{\Yc},\tilde{m}^\smallfont{Z}}_s$-measurable and bounded. This is, again by \Cref{OGmart}, equivalent to 
\[ \big[ \tilde{M}^\phi_\cdot, \E^{\tilde{\P}}[ \bar{U} \vert \Fc^{\tilde{X},\tilde{\Yc},\tilde{m}^\smallfont{Z}}_\cdot ] \big]^{F^{\smallfont{\tilde{X}}\smallfont{,}\smallfont{\tilde{\Yc}}\smallfont{,}\smallfont{\tilde{m}}^\tinyfont{Z}}{,}{\tilde{\P}}} \equiv 0.\] 

If $1/q+1/q^\prime \leq 1$ is not satisfied,  we can use localisation techniques by setting $\tau_n \coloneqq \inf\{ t \in [0,T] : \lvert \tilde{M}^\phi_t \rvert \geq n \; {\rm or }\; t =T \}$ for $n \in \N.$ It is clear to see that $\tau_n$ is then an $\F^{\tilde{X},\tilde{\Yc},\tilde{m}^\smallfont{Z}}$--stopping time and the proof of \Cref{OGmart} as well as the above goes through with $\tilde{M}^\phi$ replaced by the stopped process $\tilde{M}^\phi_{\cdot \wedge \tau_\smallfont{n}}$ for any $n \in \N.$ Since $\lim_{n \rightarrow \infty} \tau_n=T,$ we conclude by passing to the limit that \[ \big[ \tilde{M}^\phi_\cdot, \E^{\tilde{\P}}[ \bar{U} \vert \Fc^{\tilde{X},\tilde{\Yc},\tilde{m}^\smallfont{Z}}_\cdot ] \big]^{F^{\smallfont{\tilde{X}}\smallfont{,}\smallfont{\tilde{\Yc}}\smallfont{,}\smallfont{\tilde{m}}^\tinyfont{Z}}{,}{\tilde{\P}}} \equiv 0.\] 

\medskip In any case, it suffices to take $\P$ as the pushforward of $\tilde{\P}$ under the map $\omega \longmapsto (\tilde{X},\tilde{\Yc},\tilde{m}^Z,\bar{U})(\omega).$
\end{proof}

Let us now introduce the main compactification results. In \Cref{tight}, we show that a certain set is relatively compact, but not every element in the closure might have the right properties concerning the random variable $U$. Subsequently, in \Cref{lem:compactness}, we present a reinforced version of this set, where closedness is actually guaranteed.

\begin{lemma} \label{tight} The set 
\begin{align*}
K_{\varepsilon,R}\coloneqq\bigg\lbrace \P \in \Rc : \E^{\P}\bigg[ \int_0^T \int_{\R^\smallfont{d}} \lVert z \rVert^{2+\varepsilon} m^Z(\mathrm{d}z,\mathrm{d}s) + \lvert U \rvert^{1+\varepsilon} \bigg]\leq R \bigg\rbrace,
\end{align*} is relatively compact in the weak topology for every $\varepsilon>0$ and $R>0$. Moreover, every cluster point $\P^\prime \in \overline{K_{\varepsilon,R}}$ satisfies \[ 
\E^{\P^\smallfont{\prime}}\bigg[ \int_0^T \int_{\R^\smallfont{d}} \lVert z \rVert^{2+\varepsilon} m^Z(\mathrm{d}z,\mathrm{d}s) + \lvert U \rvert^{1+\varepsilon} \bigg]\leq R, \] as well as all conditions in {\rm \Cref{rule}} except possibly \ref{rule5}.
\end{lemma}

\begin{proof} Let us fix $\varepsilon>0$ and $R>0$ and let us denote $K\coloneqq K_{\varepsilon,R}$ for brevity. Let us assume that $K$ is non-empty, since the claim is trivial otherwise. For relative compactness, it is sufficient to verify that \[ 
\big\lbrace \P\big\vert_{\Fc_\smallfont{T}^c} : \P \in K \big\rbrace\; {\rm and}\; \big\lbrace \P\vert_{\Bc(\R)} : \P \in K \big\rbrace ,
\]
are tight. To prove the former, we verify Conditions $\ref{T1}$ and $\ref{T2}$ of \Cref{tightness}. Let us take $M > 0.$ Then (see \Cref{tightness} for the definition of $\sigma^{M,\varepsilon}$),
\begin{align*}
\P\left[ \sigma^{M,\varepsilon} < T \right] &\leq \P\bigg[ \sup_{t \in [0,T]} \lVert (X_t,\Yc_t) \rVert \geq M \bigg] +\P\bigg[ \int_0^T \int_{\R^\smallfont{d}} \lVert z \rVert^{2+\varepsilon} m^Z_s(\mathrm{d}z)\mathrm{d}s \geq M \bigg] \\
& \leq \frac{1}{M}\E^\P\bigg[\sup_{t \in [0,T]}\lVert (X_t,\Yc_t)\rVert + \int_0^T \int_{\R^\smallfont{d}} \lVert z \rVert^{2+\varepsilon} m^Z_s(\mathrm{d}z)\mathrm{d}s \bigg],\; \P \in K.
\end{align*} The bound \eqref{bdMOM} and the assumption yield
 \[ 
 \sup_{\P \in K} \E^{\P}\bigg[\sup_{t \in [0,T]}\lVert (X_t,\Yc_t)\rVert +  \int_0^T \int_{\R^\smallfont{d}} \lVert z \rVert^{2+\varepsilon} m^Z_s(\mathrm{d}z)\mathrm{d}s \bigg] < \infty, \] and thus 
\[ 
\lim_{M \rightarrow \infty} \sup_{\P \in K} \P\big[ \sigma^{M,\varepsilon} < T \big]=0.
\] 
Hence, Condition $\ref{T1}$ in \Cref{tightness} is established. Let us now have $\phi \in \Cc_b^2(\R^{d+1},\R),$ where $\Cc_b^2(\R^{d+1},\R)$ denotes the space of all twice continuously differentiable, bounded real--valued functions with bounded derivatives up to order $2$ on $\R^{d+1}.$ Using \eqref{boundFA}, we have the following bound for $s \in \llbracket 0,\sigma^{M,\varepsilon} \rrbracket$
\begin{align}
\begin{split} \label{eqn:bound_L}
\lvert\mathscr{L}\phi (s,X_{\cdot \wedge s},\Yc_s,z) \rvert &\leq \lvert b(s,X_{\cdot \wedge s},z)\cdot \mathrm{D}\phi(X_s,\Yc_s)\rvert + \frac{1}{2} \lvert a(s,X_{\cdot \wedge s},z) : \mathrm{D}^2\phi(X_s,\Yc_s) \rvert \\
&\leq C \sup_{t \in \llbracket 0,\sigma^{\smallfont{M,\varepsilon}} \rrbracket} \bigg( \lVert  \mathrm{D}\phi(X_t,\Yc_t) \rVert +\frac{1}{2}\lVert \mathrm{D}^2\phi(X_t,\Yc_t) \rVert + \|\sigma(t,X_{\cdot \wedge t})\|\bigg)\big(1 + \lVert X_{\cdot \wedge s} \rVert + \lvert \Yc_s \rvert + \lVert z \rVert^2 \big) \\
&\leq C \sup_{t \in \llbracket 0,\sigma^{\smallfont{M,\varepsilon}} \rrbracket} \bigg( \lVert  \mathrm{D}\phi(X_t,\Yc_t) \rVert +\frac{1}{2}\lVert \mathrm{D}^2\phi(X_t,\Yc_t) \rVert+ \|\sigma(t,X_{\cdot \wedge t})\| \bigg)\big(1 + 2M + \lVert z \rVert^2 \big).
\end{split}
\end{align} It is then straightforward to verify Condition $\ref{T2}$ in \Cref{tightness} since we can set 
\[
A_{\phi,M}= \mathrm{const.} (T) C \sup_{t \in \llbracket 0,\sigma^{\smallfont{M,\varepsilon}} \rrbracket} \bigg( \lVert  \mathrm{D}\phi(X_t,\Yc_t) \rVert +\frac{1}{2}\lVert \mathrm{D}^2\phi(X_t,\Yc_t) \rVert+ \|\sigma(t,X_{\cdot \wedge t})\| \bigg)(1+ 2M).
\]
If $\psi(\cdot)=\phi(\cdot-a),$ we can clearly take $A_{\psi,M}=A_{\phi,M}.$ Tightness of the set $\lbrace \P\vert_{\Bc(\R)} : \P \in K \rbrace$ follows directly from the assumption.

\medskip 
We now prove the second part. Let us have a sequence $(\P_n)_{n\in\N^\smallfont{\star}}$ valued in $K$ having a weak limit $\P^\prime.$ Clearly
\[ 
\E^{\P^\smallfont{\prime}}\bigg[ \int_0^T \int_{\R^\smallfont{d}} \lVert z \rVert^{2+\varepsilon} m^Z_s(\mathrm{d}z)\mathrm{d}s + \lvert U \rvert^{1+\varepsilon} \bigg]\leq \liminf_{n \rightarrow \infty}\E^{\P_\smallfont{n}}\bigg[ \int_0^T \int_{\R^\smallfont{d}} \lVert z \rVert^{2+\varepsilon} m^Z_s(\mathrm{d}z)\mathrm{d}s + \lvert U \rvert^{1+\varepsilon} \bigg] \leq R,\] by the Portmanteau theorem for lower-bounded lower-semicontinuous functions. Conditions \ref{rule1} and \ref{rule3} in \Cref{rule} are clearly satisfied by $\P^\prime.$ As for  \Cref{rule}.\ref{rule2}, let us define for $M>0$ and $\phi \in \Cc_c^2(\R^{d+1},\R)$
\begin{align*}
M_t^{\phi,M}&\coloneqq\phi(X_t,\Yc_t)-\int_0^t \int_{\lVert z \rVert < M} \mathscr{L}\phi (s,X_{\cdot \wedge s},\Yc_s,z) m^Z_s(\mathrm{d}z)\mathrm{d}s, \\
N_t^{\phi,M}&\coloneqq-\int_0^t \int_{\lVert z \rVert \geq M} \mathscr{L}\phi (s,X_{\cdot \wedge s},\Yc_s,z) m^Z_s(\mathrm{d}z)\mathrm{d}s,\; t \in [0,T].
\end{align*} We have for any  $ s \in[0, t]$ and $h_s$ which is $\Fc_s^c$-measurable, continuous and bounded, that
\begin{align} \label{mart} \begin{split}
\big\lvert \E^{\P^\smallfont{\prime}} \big[h_s(M_t^\phi-M_s^\phi )\big] \big\rvert &\leq \big\lvert \E^{\P^\smallfont{\prime}} [h_s(M_t^{\phi,M}-M_s^{\phi,M} )] -\E^{\P_\smallfont{n}} [h_s(M_t^{\phi,M}-M_s^{\phi,M} )] \big\rvert \\
&\quad +\big\lvert \E^{\P^\smallfont{\prime}} [h_s(N_t^{\phi,M}-N_s^{\phi,M})] \big\rvert + \big\lvert \E^{\P_\smallfont{n}} [h_s(N_t^{\phi,M}-N_s^{\phi,M} )] \big\rvert. \end{split}
\end{align}
Note that we have using analogous arguments as in \eqref{eqn:bound_L} that
\begin{align*}
\mathbf{1}_{\{\|z\|\geq M \}}\lvert\mathscr{L}\phi (s,x_{\cdot \wedge s},y_s,z) \rvert\leq const.(T,\phi,\sigma) C \mathbf{1}_{\{\|z\|\geq M \}} \lVert z \rVert^2 .
\end{align*} holds for every $s \in [0,T]$ and $(x,y_s) \in \Cc([0,T],\R^{d}) \times \R$ and $M\geq 1$ large enough such that \[ {\rm supp}( \phi )\subset \{ (x,y) \in \R^{d+1} : \|x \|+\|y\|\leq M \}. \]

Thus, for any $\P \in \Pc(\Omega^c \times \R)$ and $M>0$ large enough we have
\begin{align} \label{mar} \begin{split}
\big\lvert \E^{\P} [h_s(N_t^{\phi,M}-N_s^{\phi,M} )] \big\rvert \leq \| h \|_{\infty}  \E^{\P}\big[\big|  N_t^{\phi,M}-N_s^{\phi,M}\big|\big] &\leq \| h \|_{\infty}  \E^{\P}\bigg[ \int_s^t \int_{\lVert z \rVert \geq M}  \lvert\mathscr{L}\phi (s,X_{\cdot \wedge s},\Yc_s,z) \rvert m^Z_s(\mathrm{d}z)\mathrm{d}s\bigg]  \\
&\leq \mathrm{const.} (T,\phi,\sigma) C \| h \|_{\infty} \E^{\P}\bigg[ \int_0^T \int_{\lVert z \rVert \geq M} \lVert z \rVert^2  m^Z_s(\mathrm{d}z)\mathrm{d}s\bigg]. \end{split}
\end{align}
Since by analogous arguments as in \Cref{rem:UI_convergence} we obtain
\begin{equation} \label{eqn:unif_conv.}
\sup\bigg\lbrace \E^{\P} \bigg[\int_0^T \int_{\lVert z \rVert \geq M} \lVert z \rVert^2  m^Z_s(\mathrm{d}z)\mathrm{d}s\bigg] : \P \in \Pc(\Omega^c \times \R), \; \E^{\P}\bigg[  \int_0^T \int_{\R^d} \lVert z \rVert^{2+\varepsilon} m^Z_s(\mathrm{d}z)\mathrm{d}s\bigg] \leq R  \bigg\rbrace \overset{M \rightarrow \infty}{\xrightarrow{\hspace*{1cm}}}0,\end{equation} we have that the last two terms in \eqref{mart} go to $0$ as $M \longrightarrow \infty$, uniformly in $n \in \N^\star.$ Moreover, for any fixed $M>0$ such that \[ \E^{\P'}\bigg[\int_0^T \int_{\R^d} \mathbf{1}_{\{\|z\|=M\}} m^Z_s(\mathrm{d}z)\mathrm{d}s\bigg]=0 \] the first term in \eqref{mart} converges to $0$ due to \Cref{lem:continuity_of_phi} applied to $\phi=\mathscr{L}\phi$ and boundedness of $\mathscr{L}\phi$ on the given set. Hence, Condition \ref{rule2} is verified.

\medskip
 It remains to verify Condition \ref{rule4}. Because the $h^i,\;i\in I,$ are lower-semicontinuous, we have by \Cref{fatou} that 
 \[
 \E^{\P^\smallfont{\prime}} \big[h^i(X_{\cdot \wedge T},\Yc_T,U) \big]\leq \liminf_{n \rightarrow \infty} \E^{\P_\smallfont{n}} \big[h^i(X_{\cdot \wedge T},\Yc_T,U)\big] \leq 0,\; i \in I. \] Here, due to \Cref{assh} it is easy to verify that the negative part $h^{i,-}$ is uniformly asymptotically integrable with respect to the sequence $(\P_n)_{n \in \N^\smallfont{\star}}$ for every $i \in I.$ This concludes the proof.  
\end{proof}

\begin{lemma} \label{lem:compactness} The set 
\begin{align*}
K_{q,q^\smallfont{\prime},R}\coloneqq\bigg\lbrace \P \in \Rc : \E^{\P}\bigg[ \int_0^T \int_{\R^d} \lVert z \rVert^{2q} m^Z_s(\mathrm{d}z)\mathrm{d}s  + \lvert U \rvert^{q^{\prime}} \bigg]\leq R \bigg\rbrace,
\end{align*} is compact in the weak topology for every $ (q,q^\prime) \in (1,\infty)^2$ satisfying \[ \frac{1}{q}+\frac{1}{q^\prime}< 1. \]
\end{lemma}

\begin{proof} The assumptions of \Cref{tight} are clearly satisfied and hence its conclusion is valid as well. Let $(\P_n)_{n\in\N^\smallfont{\star}}$ be a sequence valued in $K_{q,q^\smallfont{\prime},R}$ such that $\P_n \longrightarrow \P^\prime$ weakly for some $\P^\prime.$ It hence suffices to show that $\P^\prime$ satisfies \Cref{rule}.\ref{rule5}. To this end, let us now fix $\phi \in \Cc_c^2(\R^{d+1},\R).$ From \Cref{OGmart} we have that \Cref{rule}.\ref{rule5} is equivalent to
 \[\E^{\P}\big[h_s U ( M^\phi_t - M^\phi_s )\big]=0,\] for all $s \in[0, t]$ and $h_s$ which is $\Fc_s^{c}$-measurable, continuous and bounded. Similarly as before, (see \eqref{mart} for definitions)
\begin{align}
\begin{split} \label{eqn:compactness_proof}
\big\lvert \E^{\P^\smallfont{\prime}}[ h_s U (M_t^\phi-M_s^\phi )]\big \rvert &\leq \big\lvert \E^{\P^\smallfont{\prime}} [h_s U (M_t^{\phi,M}-M_s^{\phi,M} )] -\E^{\P_\smallfont{n}}[ h_s U (M_t^{\phi,M}-M_s^{\phi,M} )] \big\rvert \\
&\quad +\big\lvert \E^{\P^\smallfont{\prime}} [h_s U (N_t^{\phi,M}-N_s^{\phi,M} )] \big\rvert + \big\lvert \E^{\P_\smallfont{n}}[ h_s U (N_t^{\phi,M}-N_s^{\phi,M} )]\big \rvert.
\end{split}
\end{align} 
We can use the same argument as in the proof of the last part of \Cref{tight}. We just need to verify the uniform asymptotic integrability condition from \Cref{fatou}.  We can find $\varepsilon>0$ such that $(1+\varepsilon)/q+(1+\varepsilon)/q^\prime=1.$ Using H\"{o}lder's inequality and similar argument as in \eqref{mar}, we get for any $\P \in \{\P^n : n \in \N^\star\} \cup \{\P^\prime\}$
\begin{align*}
 \E^{\P}\big[\lvert h_s U (M_t^\phi-M_s^\phi ) \rvert^{1+\varepsilon} \big]&\leq \mathrm{const.} (\phi,C,h,\sigma)\bigg(1+ \E^{\P}\bigg[\bigg\lvert U\int_0^T \int_{\R^\smallfont{d}} \lVert z \rVert^2  m^Z_s(\mathrm{d}z)\mathrm{d}s  \bigg\rvert^{1+\varepsilon}\bigg]\bigg) \\
 &\leq \mathrm{const.} (\phi,C,h,\sigma)\bigg(1+\big( \E^{\P}[ \lvert U\rvert^{q^\prime} ]\big)^\frac{1+\varepsilon}{q^\prime} \bigg( \E^{\P}\bigg[\bigg\lvert \int_0^T \int_{\R^\smallfont{d}} \lVert z \rVert^2  m^Z_s(\mathrm{d}z)\mathrm{d}s \bigg\rvert^{q}\bigg] \bigg)^\frac{1+\varepsilon}{q}\bigg) \\
  & \leq \mathrm{const.} (\phi,C,h, T, R,\sigma)\bigg(1+\big( \E^{\P}[ \lvert U\rvert^{q^\prime} ]\big)^\frac{1+\varepsilon}{q^\prime} \bigg( \E^{\P}\bigg[ \int_0^T \int_{\R^\smallfont{d}} \lVert z \rVert^{2q}  m^Z_s(\mathrm{d}z)\mathrm{d}s\bigg] \bigg)^\frac{1+\varepsilon}{q}\bigg) \\
 & \leq \mathrm{const.} (\phi,C,h, T, R,\sigma).
\end{align*} The rest of the proof can be done in a similar fashion as what follows after \eqref{mart}. Note that in this case the function $(X,\Yc,m^Z,U) \longmapsto h_s U (M_t^{\phi,M}-M_s^{\phi,M} )$ is still continuous on the appropriate set due to \Cref{lem:continuity_of_phi}, but is not bounded. To conclude the convergence of the first term in \Cref{eqn:compactness_proof} to 0, we thus employ \Cref{fatou} together with the bound above, which shows asymptotic uniform integrability of the given function. Further, as for the second two terms, we can write
\[ \E^{\P} \bigg[\lvert U \rvert \int_0^T \int_{\lVert z \rVert \geq M} \lVert z \rVert^2  m^Z_s(\mathrm{d}z)\mathrm{d}s\bigg] \leq \frac{1}{M^\gamma} \E^{\P}  \bigg[\lvert U \rvert \int_0^T \int_{\R^d} \lVert z \rVert^{2+\gamma}  m^Z_s(\mathrm{d}z)\mathrm{d}s\bigg],\; \gamma>0. \] The expectation on the right-hand side can, as before, be bounded from above for some $\gamma>0$ sufficiently small. Indeed, we can do the computations above with $\tilde{q}\coloneqq (2q)/(2+\gamma)$ instead of $q,$ which still satisfies $1/\tilde{q}+1/{q^\prime}<1$ for $\gamma>0$ small enough. Thus, we obtain uniform convergence for $M \longrightarrow \infty$ as in \eqref{eqn:unif_conv.}.
\end{proof}

To conclude the argument, we show that the objective function is upper-semicontinuous.

\begin{lemma} \label{lem:upper-sc}
The map $\P \longmapsto \E^{\P}[F(X_{\cdot \wedge T},\Yc,U)]$ is upper-semicontinuous on $K_{\varepsilon,R}$ for any $\varepsilon>0$ and $R>0.$
\end{lemma}
\begin{proof} Let us set $K\coloneqq K_{\varepsilon,R}$ for brevity. We have that the map $ (x,y,u) \longmapsto F(x,y,u)$ is upper-semicontinuous by assumption. Moreover
\begin{align*}
&\sup_{\P\in K}\E^{\P}\big[(F^+(X_{\cdot \wedge T},\Yc,U))^{1+\gamma}\big]  \leq \mathrm{const.} (\gamma) \bigg(1+  \sup_{\P\in K} \E^{\P}\big[\lVert X_{\cdot \wedge T} \rVert^{1+\gamma}\big] +  \sup_{\P\in K} \E^{\P}\big[\lvert Y_T \rvert^{1+\gamma} \big]+  \sup_{\P\in K} \E^{\P}\big[\lvert U \rvert^{1+\gamma}\big] \bigg) < \infty,
\end{align*} 
from \eqref{bdMOM} for some $\gamma>0$ sufficiently small. The result is then immediate from \Cref{fatou}.
\end{proof}

\end{document}